\documentclass[onefignum,onetabnum,final]{siamart250211}
\usepackage{lipsum}
\usepackage{amsmath,amssymb,amsfonts,latexsym,stmaryrd}
\usepackage{graphicx,float}
\usepackage{mathrsfs} 
\usepackage{color}
\usepackage{epstopdf}
\usepackage{algorithmic}
\usepackage{multirow}
\ifpdf
\DeclareGraphicsExtensions{.eps,.pdf,.png,.jpg}
\else
\DeclareGraphicsExtensions{.eps}
\fi
\usepackage{diagbox}

\usepackage{amsmath,amssymb,epsfig,float,color,xcolor}
\def\CT{\mathcal{T}}
\newcommand\eff{\texttt{eff}}

\numberwithin{equation}{section}

\newcommand\N{\mathbb{N}}
\newcommand\R{\mathbb{R}}

\renewcommand\dim{\mathop{\mathrm{dim}}\nolimits}

\renewcommand\O{\Omega}
\newcommand\G{\Gamma}

\renewcommand\H{\mathrm{H}}

\newcommand\HcuO{\H_{\G}^1(\O)}

\newcommand\HurO{\H^{1+r}(\O)}

\newcommand\HusO{\H^{1+s}(\O)}

\renewcommand\S{\Sigma}
\renewcommand\sp{\mathop{\mathrm{sp}}\nolimits}

\def\l{\lambda}
\def\CE{{\mathcal E}}
\def\CT{{\mathcal T}}

\newcommand{\vertiii}[1]{{\left\vert\kern-0.25ex\left\vert\kern-0.25ex\left\vert #1 
		\right\vert\kern-0.25ex\right\vert\kern-0.25ex\right\vert}}
\allowdisplaybreaks

\newcommand{\vertiH}[1]{{\left\vert\kern-0.25ex\left\vert #1 
		\right\vert\kern-0.25ex\right\vert_{\H(h)}}}
\newsiamremark{remark}{Remark}
\newsiamremark{hypothesis}{Hypothesis}
\newsiamremark{problem}{Problem}
\crefname{hypothesis}{Hypothesis}{Hypotheses}
\newsiamthm{claim}{Claim}

\headers{Nitsche method for the Laplace eigenvalue problem}{Arbaz Khan, David Mora and Jesus Vellojin}

\title{Nitsche-based FEM for the Laplace eigenvalue problem: spectral approximation and a posteriori error analysis\thanks{Submitted to the editors DATE.
		\funding{DM was partially supported by project Centro de Modelamiento Matemático (CMM), FB210005, BASAL funds for centers of excellence, by the National Agency for Research and Development, ANID-Chile through project Anillo of Computational Mathematics for Desalination Processes ACT210087 and by FONDECYT project 1220881. }}}

\author{Arbaz Khan \thanks{Department of Mathematics, Indian Institute of Technology (IIT) Roorkee, India. \email{arbaz@ma.iitr.ac.in}}\and David Mora\thanks{GIMNAP, Departamento de Matem\'atica, Universidad del B\'io-B\'io, Concepci\'on, Chile
and CI$^2$MA, Universidad de Concepci\'on, Chile. \email{dmora@ubiobio.cl}}
	\and Jesus Vellojin\thanks{Departamento de Ciencias, Universidad T\'ecnica Federico Santa Mar\'ia, Valpara\'iso, Chile. \email{jesus.vellojinm@usm.cl}.}}

\usepackage{amsopn}

\begin{document}
	
	\maketitle
	
	\begin{abstract}
		In this paper, we present the numerical analysis of an elliptic eigenvalue problem in which the essential boundary condition is imposed weakly by means of the Nitsche method. The resulting discrete eigenvalue problem is studied within the framework of compact operator theory. We prove norm convergence of the discrete solution operator and derive error estimates for the eigenvalues and eigenfunctions, with rates depending on the chosen Nitsche variant. In addition, we develop an a posteriori error analysis and propose a residual-based estimator suitable for adaptive refinement. Several numerical experiments are presented to assess the convergence, stability and robustness of the method, including the influence of the Nitsche stabilization parameter and the performance of the adaptive strategy.		
	\end{abstract}
	
	
	% REQUIRED
	
%	\keywords{elasticity eigenvalue problem \and Stokes eigenvalue problem \and
%		finite elements \and a priori error estimates \and a posteriori error bounds}
%	\subclass{65N30 \and 65N12 \and 76D07 \and 65N15}

	%************************************************************************************************
	\section{Introduction}
	%************************************************************************************************
	
	Eigenvalue problems arise in numerous scientific and engineering contexts, including structural vibrations, quantum mechanics, and stability analysis of physical systems  \cite{rannacher1979nonconforming,bermudez1995finite,csendes1970numerical,davies1982finite,bossavit1990solving}. These problems typically involve finding eigenvalues and eigenfunctions of differential operators subject to appropriate boundary conditions. Typically, analytical solutions are challenging to find, providing the need of developing efficient numerical methods \cite{MR2652780,SUN-BOOK}. Among these,
	the finite element method \cite{BO,MR2652780}, Discontinuous Galerkin methods (DG) \cite{antonietti2006discontinuous,Buffa}, and Virtual element methods \cite{meng2020virtual,MV2020,MRR1} have emerged as a robust and versatile frameworks for approximating eigenvalue problems, offering systematic approaches to handle complex geometries, boundary conditions, and varying material properties. However, we note that the approximation
	of eigenvalue problems by stabilized methods can introduce artificial (spurious)
	modes that do not correspond to the true physical behavior of the system \cite{boffi2020approximation}.
	
	%boffi2020approximation

	Continuous finite element spaces have long been employed to approximate eigenvalue problems, with robust theoretical background that guarantees convergence and accuracy of the methods. We refer to \cite{SUN-BOOK,MR2652780} for an extensive review of finite element approximations for eigenvalue problems. In recent years, DG methods have gained significant attention as an alternative to continuous finite element spaces. DG methods offer several advantages, including the ability to handle nonconforming meshes, complex boundary conditions, and higher-order accuracy \cite{arnold2002unified, riviere2008discontinuous}. For eigenvalue problems, DG methods provide greater flexibility in enforcing inter-element continuity and controlling numerical dispersion \cite{hesthaven2007nodal, houston2004mixed}. Moreover, hybridized versions of DG methods have been developed to improve computational efficiency while preserving the theoretical convergence properties \cite{cockburn2009}. In \cite{antonietti2006discontinuous} the authors investigate the DG approximation of the Laplace eigenvalue problem, providing theoretical and computational insights into its performance. Their work demonstrates that DG methods can achieve optimal convergence rates and avoid spurious eigenvalues by properly designing stabilization terms. Additionally, the paper highlights the versatility of DG methods in handling irregular geometries and their potential for adaptive mesh refinement (see also \cite{Buffa,MR3962898,LMSisc2020}).
	
	In this paper we focus on an intermediate method between continuous finite element schemes and DG methods. More precisely, this study proposes the numerical analysis of eigenvalue problems through the weak imposition of boundary conditions with the Nitsche method. 	The Nitsche method is a numerical technique introduced in \cite{nitsche1971variationsprinzip} to address variational formulations involving essential boundary conditions without explicitly enforcing them. It is often considered a precursor to DG methods due to its ability to handle discontinuities and its flexibility in weakly imposing constraints \cite{arnold2002unified,juntunen2009nitsche,burman2012penalty}. The method has found applications in various models, such as  interface problems or slip boundary conditions on fluid problems \cite{hansbo2005nitsche,MR4744101}. One of the key features of the Nitsche method is its ability to accommodate problems involving discontinuities or coupling conditions across interfaces, making it particularly suitable for interface and contact problems in computational mechanics \cite{hansbo2002discontinuous,chouly2015symmetric,chouly2015nitsche,chouly2020hybrid}. For instance, in problems where materials with different properties are in contact, the Nitsche method allows for a seamless integration of the interface conditions into the variational formulation \cite{chouly2018unbiased,burman2020nitsche}.

	Particularly, Nitsche's method and its variants have been widely used for the imposition of different boundary conditions \cite{chouly2024finite}. Recent developments also show the usefulness of Nitsche-type techniques in mixed formulations for flow problems. For instance, in \cite{anaya2026nitsche}, a Nitsche-based finite element method with grad-div stabilization is proposed for the velocity-pressure formulation of the Brinkman problem with mixed boundary conditions on the pressure. In \cite{juntunen2009nitsche}, the authors study the extension of the symmetric method to generalized boundary conditions, providing also an a posteriori analysis. The authors also show the clear advantage of the Nitsche technique over the traditional penalty method. In \cite{burman2012penalty}, the author introduces a variation of Nitsche's method that eliminates the need for penalty parameters when imposing boundary conditions weakly. The study demonstrates that this nonsymmetric approach maintains stability without requiring a penalty term. The author provides optimal error estimates in the $H^1$ norm and suboptimal estimates in the $L^2$ norm, which are half an order less accurate.  In the realm of eigenvalue problems, there appear to be only a few numerical studies, \cite{harari2018spectral} and \cite{albocher2021spectral}, where the behavior of the spectrum in the symmetric variant of the Nitsche method is explored in different scenarios such as non-conforming meshes, anisotropic meshes, or convex and non-convex domains. A detailed study comparing the standard method, namely, imposing the boundary condition strongly, and the Nitsche technique is presented. More recently, in \cite{boffi2026maxwell-nitsche} the authors analyzed the Nitsche method for the prescription of essential boundary conditions for conforming finite element approximations of Maxwell's problem. Their work is particularly relevant because it treats Nitsche's method beyond the standard coercive elliptic setting, considering both an inf-sup stable Galerkin formulation and an augmented-stabilized formulation that allows the use of nodal finite element interpolations. Stability and convergence results are established for the corresponding boundary value problem. Nevertheless, this contribution concerns source problems for Maxwell's equations and does not address the spectral approximation of eigenvalue problems.
	
	    The preceding discussion shows that, although Nitsche's method has been extensively analyzed for boundary value problems and has also been explored numerically in spectral computations, a rigorous spectral approximation theory for eigenvalue problems with boundary conditions imposed by Nitsche's method is still missing. To the best of the authors' knowledge, no proof of convergence is available for the Laplace eigenproblem with mixed boundary conditions when the Dirichlet part is imposed weakly through Nitsche's technique. Thus, the main goal of this work is to fill the gap from the theoretical point of view. We propose and analyze a finite element method for the study of eigenvalue problems using the Nitsche technique. We will concentrate on the study of a simple problem such as the Laplace eigenproblem, where we consider mixed boundary conditions. A primal formulation is used, so that the definition of the solution operator and the spectral characterization is followed by arguments at the continuous level. However, the discrete scheme considers the imposition of the Dirichlet condition by means of the three variants of Nitsche's method: symmetric \cite{burman2023augmented}, incomplete \cite[Section 37.1]{ern2021finite}, and skew-symmetric \cite{freund1995weakly,burman2012penalty}. {Moreover, our analysis can be adapted to that of \cite{burman2012penalty} to cover a penalty-free type method (cf. Remark \ref{rem:burman-penalty-free})}. Being a method where there is a stabilization parameter that guarantees the coercivity of the corresponding bilinear forms, we study how the discrete solution operator converges in norm to the continuous as $h$ goes to zero. Hence, we can use the so-called Babu\v ska-Osborne abstract spectral approximation theory (see \cite{BO}).
	We prove optimal order error estimates (see Theorem~\ref{gapr}) for the eigenfunctions and
	a double order for the eigenvalues in the symmetric case, and a suboptimal
	order for the nonsymmetric methods. In addition, we develop a residual-based a posteriori error analysis for the symmetric Nitsche formulation. The proposed estimator includes the element residuals, the jumps of the normal derivative across interior facets, and the boundary residuals induced by the weak imposition of the Dirichlet condition. Under a saturation assumption, we prove reliability for simple eigenvalues, while local efficiency is obtained by standard bubble-function arguments. This analysis provides the theoretical basis for the adaptive refinement strategy used in the numerical experiments. We also discuss about the appearance of spurious eigenvalues	and eigenfunctions if the Nitsche stabilization parameter is not correctly tuned.
	Finally, we mention that the present analysis constitutes
	a stepping stone towards the more challenging goal of analyzing Nitsche methods
	for other eigenvalue problems, such as the Stokes eigenvalue problem with slip boundary conditions \cite{berselli20253d,falocchi2022remarks}, among others.
	
	The remainder of this paper is organized as follows: Section \ref{sec:model} provides a detailed review of the mathematical formulation of the Laplace eigenproblem with mixed boundary conditions. Section \ref{sec:fem} introduces and discusses the theoretical analysis for the Nitsche method. Section \ref{SEC:approximation} deals with some continuity and consistency properties to obtain error estimates. In Section \ref{sec:apost} we present an a posteriori error analysis for the proposed Nitsche method. Section \ref{sec:numerical-experiments} presents numerical experiments that validate the theoretical results in two and three dimensions. Finally, Section \ref{sec:conclusions} concludes with a summary of findings and perspectives for future research in this area.
	
	\subsection{Notations}\label{sec:notations}
	Let $\mathcal{O}$ be a subset of $\mathbb{R}^2$ with Lipschitz boundary $\partial\mathcal{O}$.   For $r \geq 0$ and $p \in[1, \infty]$, we denote by $\mathrm{L}^p(\mathcal{O})$ the usual Lebesgue space of maps from $\mathcal{O}$ to $\mathbb{R}$ endowed with the norm $\|\cdot\|_{\mathrm{L}^p(\mathcal{O})}$, while $\H^r(\mathcal{O})$ denotes a Hilbert space. We write $|\cdot|_{r, \mathcal{O}}$ and $\Vert \cdot\Vert_{r,\Omega}$ to denote the seminorm and norm in Hilbert spaces.
	%%%%%%%%%%%%%%%%%%%%%%%%%%%%%%%%%%%%
	
	\section{The model problem}
	\label{sec:model}
	
	Let $\O\subset\R^2$ be a convex bounded domain with polygonal boundary
	$\partial\O$. Let $\G$ and $\S$ be disjoint open subsets of
	$\partial\O$ such that $\partial\O=\bar{\G}\cup\bar{\S}$ and
	$\left|\G\right|\ne0$. We denote by $n$ the outward unit normal vector
	to $\partial\O$ and by $\partial_n$ the normal derivative.

	The eigenvalue problem with mixed boundary conditions reads as follows:
	Find $(\l,u)\in\R\times\HcuO$, $u\ne0$, such that
	\begin{equation}\label{eq:strong-formulation}
		\begin{split}
			-\Delta u=\l u\quad\text{in }\O,\\
			u=0\quad\text{on }\G,\\
			\partial_n u=0\quad\text{on }\S,
		\end{split}
	\end{equation}
	where
	$$
	\HcuO:=\{v\in\H^1(\Omega): v=0 \text{ on } \Gamma\}.
	$$
	By testing the first equation above with $v\in\HcuO$ and integrating by
	parts, we arrive at the following equivalent variational formulation:
	
	\begin{problem}
		\label{P1}
		Find $(\l,u)\in\R\times\HcuO$, $u\ne0$, such that
		$$
		\int_{\O}\nabla u\cdot\nabla v
		=\l\int_{\O}uv\qquad\forall v\in\HcuO.
		$$
	\end{problem}
	
	We rewrite the problem as follows:
	\begin{problem}
		\label{P2}
		Find $(\l,u)\in\R\times\HcuO$, $u\ne0$, such that
		$$
		A(u,v)=\l B(u,v)\qquad\forall v\in\HcuO,
		$$
	\end{problem}
	where the bilinear form $A:\HcuO\times\HcuO\rightarrow\mathbb{R}$
	is defined by
	\begin{align}\label{formA}
		A(u,v)
		& :=\int_{\O}\nabla u\cdot\nabla v \qquad u,v\in\HcuO,
	\end{align}
	and the bilinear form $B:L^2(\Omega)\times L^2(\Omega)\rightarrow\mathbb{R}$
	is defined by
	\begin{align*}
		B(u,v)
		& :=\int_{\O}uv \qquad u,v\in L^2(\Omega).
	\end{align*}
	All the previous bilinear forms are bounded and symmetric.
	
	Next, we define the solution operator associated with Problem~\ref{P2}:
	\begin{align*}
		T:\ L^2(\Omega) & \longrightarrow L^2(\Omega),
		\\
		f & \longmapsto Tf:=w,
	\end{align*}
	where $w\in\HcuO$ is the solution of the following source problem:
	\begin{equation}
		\label{T1}
		A(w,v)=B(f,v)
		\qquad\forall v\in\HcuO.
	\end{equation}
	
	The next result establishes that $A(\cdot,\cdot)$ (cf. \eqref{formA}) is $\HcuO$-elliptic.
	\begin{lemma}
		\label{elipt}
		There exists a constant $C>0$, depending on $\O$, such that
		$$
		A(v,v)
		\ge C\left\|v\right\|_{1,\O}^2
		\qquad\forall v\in\HcuO.
		$$
	\end{lemma}
	Thus, as a consequence of the Lax-Milgram Theorem, we have that
	formulation \eqref{T1} is well-posed.
	
	We deduce that the linear operator $T$ is well
	defined and bounded. Also, $T$  is self-adjoint with respect to
	the inner products $A(\cdot,\cdot)$ in $\HcuO$ and $B(\cdot,\cdot)$ in $L^2$.
	
	Notice that $(\l,u)\in\R\times\HcuO$ solves
	Problem~\ref{P2} (and hence Problem~\ref{P1}) if and only if $Tu=\mu u$
	with $\mu\neq0$ and $u\ne0$, in which case $\mu:=\frac{1}{\l}$.
	
	The following well-known additional regularity result for the solution of
	problem~\eqref{T1} and consequently, for the eigenfunctions of $T$,
	is stated as follows.
	
	%has been proved in \cite[Lemma 2.2]{mora2015virtual}.
	
	\begin{lemma}
		\label{LEM:REG}
		There exists $r_{\O}>1/2$ such that the following results hold:
		\begin{itemize}
			\item[i)] for all $f\in L^2(\O)$ and for all $r_1\in[0,r_\O)$, the
			solution $w$ of problem~\eqref{T1} satisfies $w\in\HurO$ with
			$r:=\min\left\{r_1,1\right\}$ and there exists $C>0$ such that
			$$
			\left\|w\right\|_{1+r,\O}
			\le C\left\|f\right\|_{0,\O};
			$$
			\item[ii)] if $u$ is an eigenfunction of Problem~\ref{P1} with
			eigenvalue $\l$, for all $s\in[\frac12,r_\O)$, $u\in\HusO$ and there
			exists $\widehat{C}>0$ such that
			$$
			\left\|u\right\|_{1+s,\O}
			\le \widehat{C}\left\|u\right\|_{0,\O},
			$$
			where $\widehat{C}$ depends on the eigenvalue.
		\end{itemize}
	\end{lemma}
	%\begin{proof}
	%The proof of (i) follows from the classical regularity result for the
	%Laplace equation with Neumann boundary conditions (cf. \cite{G}). The
	%proof of (ii) follows from the same arguments and the fact that $w$ is
	%the solution of problem~\eqref{T1} with $f=\l w$, combined with a
	%bootstrap trick.
	%\end{proof}
	\begin{remark}
		The constant $r_{\O}>1/2$ is the Sobolev exponent for the Poisson
		problem with mixed boundary conditions.
		%If $\O$ is convex, then
		%$r_\O>1$, whereas, otherwise, $r_\O:=\pi/\omega$ with $\omega$
		%being the largest reentrant angle of $\O$ (see \cite{G}).
		Moreover, $T$ is a compact operator since the
		inclusion $\HurO\hookrightarrow\HcuO$ is compact.
	\end{remark}

	As a consequence of all the previous results, we have the following spectral
	characterization for the operator $T$.
	
	\begin{theorem}
		\label{CHAR_SP}
		The spectrum of $T$ decomposes as follows:
		$\sp(T)=\left\{0\right\}\cup\left\{\mu_k\right\}_{k\in\N}$, where
		$\left\{\mu_k\right\}_{k\in\N}$ is a
		sequence of finite-multiplicity eigenvalues of $T$ which converge to $0$,
		all of them are real and positive,
		and their corresponding eigenspaces lie in $\HusO$. Moreover, $\mu=0$ is not an eigenvalue of $T$.
	\end{theorem}

	\section{Numerical discretization}
	\label{sec:fem}
	In this section we study a numerical scheme to approximate the eigenvalue problem, by using the Nitsche technique. Let us consider a shape-regular family of partitions of $\O$, denoted by $\{\mathcal{T}_h\}_{h>0}$.  Let $h_K$ be the diameter of a triangle  $K\in\CT_h$ and let us define $h:=\max\{h_K\,:\, K\in \CT_h\}$.  This partition induces a mesh, denoted by $\CE_h$, over the boundary $\Gamma$. We define $\CE_h$ as the set of facets along $\Gamma$. For each edge $F\in \CE_{h}$, its diameter is denoted by $h_F$.
	
	%	For each inner facet $F\in \CE_{\O}$ and for any  sufficiently smooth  function
	%	$v$, we define the jump of its normal derivative on $F$ by
	%	$$\left[\!\!\left[ \partial_n v\right]\!\!\right]_F:=\nabla (v|_{K})  \cdot \boldsymbol{n}_{K}+\nabla ( v|_{K'}) \cdot \boldsymbol{n}_{K'} ,$$
	%	where $K$ and $K'$ are  the two elements in $\CT_{h}$  sharing the
	%	facet $\ell$ and $\boldsymbol{n}_{K}$ and $\boldsymbol{n}_{K'}$ are the respective outer unit normal vectors.
	
	We now introduce the following finite element space
	\begin{equation*}
		%\label{eq:Vh-space}
		V_h:=\left\{v_h\in C(\overline{\Omega}) \,:\, v_h\vert_K\in\mathbb{P}_\ell(K)\quad
		\forall K\in\CT_{h}\right\},
	\end{equation*}
	where $\mathbb{P}_\ell$ denotes the space of polynomials of degree at most $\ell\geq1$.
	
	%	During this section, we will need the inverse and trace inequality estimates.
	%	\begin{lemma}[Inverse inequality]\label{lem:inverse-inequality}
		%		Given $0\leq m \leq l$, for $m,l\in\mathbb{N}$,
		%		and for all $v_h\in V_h$ and $K\in\mathcal{T}_h$,
		%		there exists a constant $C_{inv}>0$, independent of $K$, such that
		%		$$
		%		\vert v_h\vert_{l,K}\leq C_{inv} h_K^{m-l}\vert v_h\vert_{m,K}.
		%		$$
		%		\end{lemma}
	%		\begin{proof}
		%			The proof can be found in \cite[Lemma 12.1]{ern2021finite}.
		%			\end{proof}
	%		
	%		\begin{lemma}[Trace inequalities]\label{lem:trace-inequality}
		%		Let $K\in\mathcal{T}_h$ and $E\in\partial K$. The following statements hold.
		%		
		%		\begin{enumerate}
			%			\item Given $v_h\in V_h$, there exists a positive constant $C>0$, independent of $K\in\mathcal{T}_h$, such that
			%			$$
			%			\Vert v_h\Vert_{0,F}\leq C h_K^{-1/2}\Vert v_h\Vert_{0,K}.
			%			$$
			%			\item If $v\in \H^1(K)$, then for any $E\in\partial K$, there exists $C>0$, independent of $K$, such that
			%			$$
			%			\Vert v\Vert_{0,F}\leq C\left\{h_K^{-1/2}\Vert v\Vert_{0,K} + h_K^{1/2}\Vert \nabla v\Vert_{0,K} \right\}.
			%			$$
			%		\end{enumerate}
		%	\end{lemma}
	%		\begin{proof}
		%			The first estimate follows from \cite[Lemma 12.8]{ern2021finite}, while Young's inequality and \cite[Lemma 12.15]{ern2021finite} allow to obtain the second estimate.
		%			\end{proof}
	Now, we introduce the following mesh dependent norms:
	\begin{align}
		\label{nomsdf}
		&\Vert v\Vert_h^2:=\Vert\nabla v\Vert_{0,\O}^2+\sum_{F\in\CE_h}h_F^{-1}\Vert v\Vert_{0,F}^2,\\
		&\vertiii{v}_h^2:=\Vert v\Vert_{h}^2
		+\sum_{F\in\CE_h}h_F\left\Vert\partial_n v_h\right\Vert_{0,F}^2.\label{nomsdf2}
	\end{align}
	
	The result given below allows to conclude that in the discrete space $V_h$ these two norms are equivalent (see for example \cite{juntunen2009nitsche}).
	\begin{lemma}\label{lem:inverse-inequality-normal-derivative}
		There exists a positive constant $C_I$ such that
		$$
		\sum_{F\in\CE_h} h_F\Vert \partial_n v_h\Vert_{0,F}^2\leq C_I\Vert \nabla v_h\Vert_{0,\Omega}^2, \qquad \forall v_h\in V_h.
		$$
	\end{lemma}
	
	In the following, we present the discrete eigenvalue problem by means of the Nitsche technique.
	
	\subsection{The discrete eigenvalue problem}
	We are now in position to state the Nitsche method for the discrete eigenvalue problem. From the strong problem \eqref{eq:strong-formulation} we multiply by test functions in $\H^1(\Omega)$ instead of $\H_{\Gamma}^1(\Omega)$ and integrate by parts with no imposition of boundary conditions over $\Gamma$. The goal is to relax and penalize this Dirichlet boundary condition such that it is satisfied in the limit of small mesh size.
	
	Starting from the boundary condition $u=0$ on  $\Gamma$, we reformulate it to be imposed weakly as follows
	$$
	\alpha\langle u,v \rangle_{1/2,h,\Gamma} - \varepsilon\langle \partial_n v,u \rangle_{1/2,\Gamma} = 0,
	$$
	where the parameter $\alpha>0$ is the so-called Nitsche parameter, $\varepsilon\in\{-1,0,1\}$ and
	$$\langle f,g\rangle_{1/2,h,\Gamma}:=\sum_{F\in\CE_{h}}h_F^{-1}\int_F f\,g, \qquad \langle f,g\rangle_{1/2,\Gamma} := \sum_{F\in\CE_{h}}\int_F f\,g.$$
	%The parameter $\alpha>0$ is the so-called Nitsche parameter.
	
	The resulting discrete eigenvalue problem is given as follows.
	\begin{problem}
		\label{P2-discrete}
		Find $(\l_h,u_h)\in\mathbb{C}\times V_h$, $u_h\ne0$, such that
		$$
		A_h(u_h,v_h)=\l_h B(u_h,v_h)\qquad\forall v_h\in V_h,
		$$
	\end{problem}
	where the discrete sesquilinear form $A_h:V_h\times V_h\rightarrow\mathbb{C}$
	is defined by
	\begin{align*}
		A_h(u_h,v_h)
		& :=A(u_h,v_h) - \langle \partial_n u_h,v_h \rangle_{1/2,\Gamma} - \varepsilon\langle \partial_n v_h,u_h \rangle_{1/2,\Gamma} + \alpha\langle u_h,v_h \rangle_{1/2,h,\Gamma}.
	\end{align*}
	The bilinear form $B(\cdot,\cdot)$ remains the same as in the continuous case.
	
	The three different choices of $\varepsilon$ yields to the well-known Nitsche variants: $\varepsilon=1$ corresponds to the symmetric version of the Nitsche method (see \cite{nitsche1971variationsprinzip,burman2023augmented}), $\varepsilon=0$ is the incomplete formulation, which has less terms, but we loose the symmetry \cite[Section 37.1]{ern2021finite}, and the skew-symmetric version is when $\varepsilon=-1$. This variant has been proven to have discrete ellipticity for all $\alpha>0$ (see, for instance, \cite{freund1995weakly}). Moreover, in \cite{burman2012penalty} it is shown that a completely penalty-free scheme is feasible.
	
	\begin{remark}
		It is important to note that  $V_h$ is not a subspace of $\H^{1}_\G(\Omega)$. Hence, the Nitsche technique to impose the Dirichlet boundary induces naturally a non-conforming finite element method. However, the solution operator is defined such that the regularizing property is satisfied. This fact becomes relevant when studying the spectral correctness, as we will see below.
	\end{remark}
	
	%		The propose method in Problem \ref{P2-discrete} is consistent. We state this fact in the following Lemma.
	
	%		\begin{lemma}[Consistency]\label{lem:consistency}
		%			Let $(\lambda, u)\in \mathbb{R}\cap V$ be the solution to Problem \ref{P2}. Assume that $u\in \H^2(\Omega)$, then
		%			$$
		%			A_h(u,v_h)=\lambda B(u,v_h), \qquad \forall v_h\in V_h.
		%			$$
		%			\end{lemma}
	%			\begin{proof}
		%				Let $u=0$ on $\Gamma$ and $v_h\in V_h$. Then, from Green's identity, integration by parts and \eqref{eq:strong-formulation} we have
		%				$$
		%				\begin{aligned}
			%					A_h(u,v_h) - \lambda B(u,v_h)&= A(u,v_h) - \langle \partial_n u,v_h \rangle_{1/2,\Gamma} - \varepsilon\langle \partial_n v_h,u \rangle_{1/2,\Gamma} + \alpha\langle u,v_h \rangle_{1/2,h,\Gamma}- \lambda B(u,v_h)\\
			%					&=\int_{\O} \nabla u \cdot\nabla v_h - \langle \partial_n u,v_h \rangle_{1/2,\Gamma} - \lambda B(u,v_h)\\
			%					&=-\int_{\O} \Delta u \, v_h - \lambda B (u,v_h) \equiv 0.
			%				\end{aligned}
		%				$$
		%				\end{proof}
	
	The next result establishes that the discrete form $A_h(\cdot,\cdot)$ is
	elliptic in the $\Vert\cdot\Vert_{h}$ norm for $\varepsilon\in\{-1,0,1\}$. The proof follows standard arguments for proving ellipticity in general DG methods, but we include it for the sake of completeness.
	%For any $\varepsilon\in\{-1,0,1\}$.
	%well-posedness of problem \eqref{T1}
	%for any value of the parameter $\varepsilon$.
	\begin{lemma}
		\label{eliptD}
		For any $\varepsilon\in\{-1,0,1\}$,
		there exists $C>0$, depending on the Nitsche parameter $\alpha$, such that
		$$
		A_h(v_h,v_h)
		\ge C\Vert v_h\Vert_{h}^2
		\qquad\forall v_h\in V_h.
		$$
		Moreover, if $\varepsilon=-1$, the result holds for any $\alpha>0$.
	\end{lemma}
	\begin{proof}
		From the Cauchy-Schwarz inequality it follows, for any $\eta>0$, that
		$$
		\begin{aligned}
			A_h(v_h,v_h)&=\Vert \nabla v_h\Vert_{0,\O}^2- (\varepsilon+ 1) \langle \partial_n v_h,v_h\rangle_{1/2,\Gamma} + \alpha  \sum_{F\in\CE_h} h_F^{-1}\Vert v_h\Vert_{0,F}^2\\
			&\geq \Vert \nabla v_h\Vert_{0,\O}^2- (\varepsilon+ 1)\left(\sum_{F\in\CE_h} h_F^{1/2}\Vert \partial_n v_h \Vert_{0,F} h_F^{-1/2}\Vert v_h\Vert_{0,F}\right) + \alpha  \sum_{F\in\CE_h} h_F^{-1}\Vert v_h\Vert_{0,F}^2\\
			&\geq \Vert \nabla v_h\Vert_{0,\O}^2- (\varepsilon+ 1)\left(\sum_{F\in\CE_h} \frac{h_F}{2\eta}\Vert \partial_n v_h \Vert_{0,F}^2 +  \frac{\eta h_F^{-1}}{2}\Vert v_h\Vert_{0,F}^2\right) + \alpha  \sum_{F\in\CE_h} h_F^{-1}\Vert v_h\Vert_{0,F}^2.\\
		\end{aligned}
		$$
		An application of Lemma \ref{lem:inverse-inequality-normal-derivative} to the above estimate gives
		$$
		A_h(v_h,v_h)\geq \left(1 - \frac{C_I(\varepsilon+1)}{2\eta}\right)\Vert \nabla v_h\Vert_{0,\O}^2  + \left(\alpha - \frac{\eta(\varepsilon+1)}{2}\right)  \sum_{F\in\CE_h} h_F^{-1}\Vert v_h\Vert_{0,F}^2.\\
		$$
		Note that the positive requirement of the second term implies that $\alpha > \frac{\eta(\varepsilon + 1)}{2}$. Then, taking $\eta=C_I(\varepsilon+1)$ yields to
		$$
		A_h(v_h,v_h)\geq \min\{C_1,C_2\}\Vert v_h\Vert_{h}^2,
		$$
		where $C_1:=\frac{1}{2}$ and $C_2=\left(\alpha - \frac{C_I(\varepsilon+1)^2}{2}\right)$.
		
		Thus, the proof is complete.
	\end{proof}
	
	Next, we define the discrete solution operator associated with Problem~\ref{P2-discrete}:
	\begin{align*}
		T_h^\varepsilon:\ L^2(\Omega) & \longrightarrow L^2(\Omega),
		\\
		f & \longmapsto T_h^\varepsilon f:=w_h^\varepsilon,
	\end{align*}
	where $w_h^\varepsilon\in V_h$ is the solution of the following source problem:
	\begin{equation*}
		\label{T1D}
		A_h(w_h^\varepsilon,v_h)=B(f,v_h)
		\qquad\forall v_h\in V_h.
	\end{equation*}

	Notice that Lemma~\ref{eliptD} implies that the linear operator $T_h^\varepsilon$ is well
	defined and bounded uniformly with respect to $h$ and $\varepsilon\in\{-1,0,1\}$. Moreover, as in the
	continuous case, $(\l_h,u_h)\in\mathbb{C}\times V_h$ solves Problem~\ref{P2-discrete}
	if and only if $T_h^\varepsilon u_h=\mu_h u_h$ with
	$\mu_h\neq0$ and $u_h\ne0$, in which case $\mu_h:= 1/\l_h$.
	Moreover, $T_h^\varepsilon$ is self-adjoint for $\varepsilon=1$.
	
	In what follows, we write $T_h$ instead of $T_h^\varepsilon$, for simplicity.
	
	As a consequence, we have the following spectral characterization for $T_h$.
	
	\begin{theorem}
		\label{CHAR_SP_DISC}
		The spectrum of $T_h$ consists of $M_h:=\dim(V_h)$ eigenvalues
		$\mu_{h}^{(k)}\in\mathbb{C}$ repeated according to their respective multiplicities.
	\end{theorem}

	\setcounter{equation}{0}
	\section{Convergence and error estimates}
	\label{SEC:approximation}

	In order to prove that the solutions of the discrete
	problem converge to those of the continuous
	problem, we will follow the standard procedure
	for spectral theory for compact operators \cite{BO},
	which consist in showing that $T_h$
	converges in norm to $T$ as $h$ tends to zero.
	
	Before stating the convergence in norm of the solution operators, we need several auxiliary results. First, we state a continuity estimate for $A_h$ in the $\vertiii{\cdot}_h$ norm defined in \eqref{nomsdf2}. The reason for using this norm instead of $\Vert \cdot\Vert_h$ is because $A_h(\cdot,\cdot)$ is not continuous in $H^1(\Omega)$. 
	
	\begin{lemma}\label{lem:continuity-in-triple-norm}
		Let $w\in\HusO$ for $s>1/2$. Then, for all $v_h\in V_h$, there exists $C>0$, independent of $h$, such that
		$$
		|A_h(w,v_h)|\leq C \vertiii{w}_h\Vert v_h\Vert_h.
		$$
	\end{lemma}
	\begin{proof}
		Thanks to Lemma \ref{lem:inverse-inequality-normal-derivative} we have
		$$
		\begin{aligned}
			|A_h(w,v_h)| &\leq \left\vert \int_{\O} \nabla w \cdot\nabla v_h \right\vert + \left\vert \langle \partial_n w,v_h\rangle_{1/2,\Gamma}\right\vert  + |\varepsilon| \left \vert \langle \partial_n v_h,w\rangle_{1/2,\Gamma} \right \vert  + \alpha|\langle w, v_h\rangle_{1/2,h,\Gamma}|\\
			&\leq \Vert \nabla w\Vert_{0,\O}\Vert \nabla v_h\Vert_{0,\O} + \sum_{F\in\CE_h} h_F^{1/2} \Vert \partial_n w\Vert_{0,F} h_F^{-1/2}\Vert v_h\Vert_{0,F}  \\
			&\hspace{2cm}+ |\varepsilon|\sum_{F\in\CE_h}h_F^{1/2} \Vert \partial_n v_h\Vert_{0,F} h_F^{-1/2}\Vert w\Vert_{0,F}+ \alpha\sum_{F\in\CE_h} h_F^{-1/2} \Vert  w\Vert_{0,F} h_F^{-1/2}\Vert v_h\Vert_{0,F}\\
			&\leq C \left(\Vert \nabla w\Vert_{0,\O}^2 + \sum_{F\in\CE_h} h_F\Vert \partial_n w\Vert_{0,F}^2+ h_F^{-1}\Vert w\Vert_{0,F}^2\right)^{1/2}\\
			&\hspace{3cm}\times \left(\Vert\nabla v_h\Vert_{0,\O}^2+\sum_{F\in\CE_h}h_F\Vert \partial_n v_h\Vert_{0,F}^2+ h_F^{-1}\Vert v_h\Vert_{0,F}^2\right)^{1/2}\\
			&\leq C \vertiii{w}_h\vertiii{v_h}_h\leq C\vertiii{w}_h\Vert v_h\Vert_h.
		\end{aligned}
		$$
	\end{proof}
	
	In turn, the following result state the consistency of the solution for the source problem \eqref{T1}.
	
	\begin{lemma}[Consistency]\label{lem:consistency}
		Let $w$ be the solution to \eqref{T1}. Then
		$$
		A_h(w,v_h)=B(f,v_h), \qquad \forall v_h\in V_h.
		$$
	\end{lemma}
	\begin{proof}
		We have that $w=0$ on $\Gamma$. Then, from Green's formula,
		integration by parts and \eqref{T1} we have
		$$
		\begin{aligned}
			A_h(w,v_h) - B(f,v_h)&= A(w,v_h) - \langle \partial_n w,v_h \rangle_{1/2,\Gamma} - \varepsilon\langle \partial_n v_h,w \rangle_{1/2,\Gamma} + \alpha\langle w,v_h \rangle_{1/2,h,\Gamma}- B(f,v_h)\\
			&=\int_{\O} \nabla w \cdot\nabla v_h - \langle \partial_n w,v_h \rangle_{1/2,\Gamma} -  B(f,v_h)\\
			&=-\int_{\O} \Delta w \, v_h - B (f,v_h) \equiv 0.
		\end{aligned}
		$$
	\end{proof}
	
	We also need the following interpolation estimate \cite[Lemma 3.4]{juntunen2009nitsche}.
	\begin{lemma}\label{lem:interpolation-estimate}
		Let $z\in\HusO$ for $s>1/2$. Then, it holds that
		$$
		\inf_{v_h\in V_h}\vertiii{z-v_h}_h\leq Ch^{\min\{s,\ell\}}\Vert z \Vert_{1+s,\O},
		$$
		where $\ell\ge1$ represents the polynomial degree of the method.
	\end{lemma}
	
	The following auxiliary result will be used to prove convergence
	of the proposed discretization. The proof use the same argument
	as those in \cite{juntunen2009nitsche}.
	
	\begin{lemma}
		There exists a constant $C > 0$, independent of $h$ and $\varepsilon$, such that for any
		$\varepsilon\in\{- 1, 0, 1\}$ and for all $f\in L^2(\Omega)$,
		if $w=Tf$ and $w_h^\varepsilon=T_hf$, then
		%\begin{equation}
		$$
		\Vert(T-T_h)f\Vert_h=\Vert w-w_h^\varepsilon\Vert_h\le C h^r\Vert f\Vert_{0,\O},
		$$
		%\end{equation}
		for $r$ as in Lemma~\ref{LEM:REG}(i).
		\end{lemma}
	\begin{proof}
		From Lemma \ref{eliptD} and Lemma \ref{lem:consistency} we have that
		$$
		\Vert w^{\varepsilon}_h - v_h\Vert_h^2\leq C A_h(w^{\varepsilon}_h
		-v_h,w^{\varepsilon}_h-v_h)= CA_h(w-v_h,w^{\varepsilon}_h-v_h) \qquad \forall v_h\in V_h.
		$$
		Next, from Lemma \ref{lem:continuity-in-triple-norm}, we have that
		$$
		A_h(w-v_h,w^{\varepsilon}_h-v_h) \leq C\vertiii{ w-v_h}_h \Vert w^{\varepsilon}_h-v_h\Vert_h \qquad \forall v_h\in V_h.
		$$
		Hence
		$$
		\Vert w_h^{\epsilon}-v_h\Vert_h\leq C \vertiii{w-v_h}_h \qquad \forall v_h\in V_h.
		$$
		Then the result follows by triangle inequality and Lemma~\ref{lem:interpolation-estimate}.
	\end{proof}
	
	We are now in a position to prove the convergence in norm of $T_h$ to $T$ as $h\to0$. This is stated below.
	%By using the Aubin-Nitsche trick, we have the following result.
	
	\begin{theorem}\label{convergT}
		There exists a constant $C > 0$, independent of $h$ and $\varepsilon$, such that for any
		$\varepsilon\in\{- 1, 0, 1\}$ and for all $f\in L^2(\Omega)$,
		if $w=Tf$ and $w^{\varepsilon}_h=T_hf$, then
		%\begin{equation}
		$$
		\Vert(T-T_h)f\Vert_{0,\O}\le C (h^r+\vert\varepsilon-1\vert h^{1/2})h^r\Vert f\Vert_{0,\O},
		$$
		%\end{equation}
		for $r$ as in Lemma~\ref{LEM:REG}(i).
	\end{theorem}
	\begin{proof}
		We consider the following auxiliary problem
		\begin{equation}\label{aux-problem}
			\begin{split}
				-\Delta z&=w-w^{\varepsilon}_h\quad\text{in }\O,\\
				z&=0\quad\qquad\text{on }\G,\\
				\partial_n z&=0\quad\qquad\text{on }\S.
			\end{split}
		\end{equation}
		From Lemma~\ref{LEM:REG}(i), we have that
		$z\in\HurO$ and there exists $C>0$
		%there exists $r\in(\frac{1}{2},1]$
		such that
		$$\Vert z\Vert_{1+r,\O}\le C\Vert w-w^{\varepsilon}_h\Vert_{0,\O}.$$
		
		Next, by testing \eqref{aux-problem} with $(w-w_h^{\varepsilon})$
		and using the boundary conditions for $z$, we have that
		\begin{equation*}
			\begin{split}
				\Vert w-w^{\varepsilon}_h&\Vert_{0,\O}^2=\int_{\O}\nabla(w-w^{\varepsilon}_h)\cdot\nabla z-\int_{\G}(w-w^{\varepsilon}_h)\partial_nz\\
				&=\int_{\O}\nabla(w-w^{\varepsilon}_h)\cdot\nabla z- \langle \partial_n (w-w^{\varepsilon}_h),z \rangle_{1/2,\Gamma} - \varepsilon\langle \partial_n z,(w-w^{\varepsilon}_h) \rangle_{1/2,\Gamma} + \alpha\langle (w-w^{\varepsilon}_h),z \rangle_{1/2,h,\Gamma}\\
				&\quad+(\varepsilon-1)\langle \partial_n z,(w-w^{\varepsilon}_h)\rangle_{1/2,\Gamma}\\
				&=A_h(w-w^{\varepsilon}_h,z)+(\varepsilon-1)\langle \partial_n z,(w-w^{\varepsilon}_h)\rangle_{1/2,\Gamma}.
			\end{split}
		\end{equation*}
		Now by using that $A_h(w-w^{\varepsilon}_h,v_h)=0$ for all $v_h\in V_h$, we get
		$$
		\Vert w-w^{\varepsilon}_h\Vert_{0,\O}^2
		=A_h(w-w^{\varepsilon}_h,z-z_h)+(\varepsilon-1)\langle \partial_n z,(w-w^{\varepsilon}_h)\rangle_{1/2,\Gamma},
		$$
		with $z_h$ the Scott-Zhang interpolant of $z$ \cite{scott1990finite}. Then, it follows that
		$$
		\Vert w-w^{\varepsilon}_h\Vert_{0,\O}^2\le C h^r\Vert w-w^{\varepsilon}_h\Vert_h\Vert z\Vert_{1+r,\O}
		+(\varepsilon-1)\langle \partial_n z,(w-w^{\varepsilon}_h)\rangle_{1/2,\Gamma},
		$$
		Next, it can be proved that
		$$
		\vert (\varepsilon-1)\langle \partial_n z,(w-w^{\varepsilon}_h)\rangle_{1/2,\Gamma}\vert\le
		\vert\varepsilon-1\vert\Vert\partial_n z\Vert_{0,\G}\Vert w-w^{\varepsilon}_h\Vert_{0,\G}
		\le\vert\varepsilon-1\vert h^{1/2}\Vert z\Vert_{1+r,\O}\Vert w-w^{\varepsilon}_h\Vert_{h}.
		$$
		Thus, we obtain
		\begin{equation}\label{dfghtrf}
			\Vert w-w^{\varepsilon}_h\Vert_{0,\O}\le C (h^r+\vert\varepsilon-1\vert h^{1/2})h^r\Vert f\Vert_{0,\O}.
		\end{equation}
	\end{proof}

	Next, an immediate consequence of
	Theorem~\ref{convergT} is that isolated parts of $sp(T)$ are approximated
	by isolated parts of $sp(T_h)$.
	It means that if $\mu$ is a nonzero eigenvalue of $T$ with algebraic
	multiplicity $m$, hence there exist $m$ eigenvalues
	$\mu_h^{(1)},\ldots,\mu_h^{(m)}$ of $T_h$
	(repeated according to their respective multiplicities)
	that will converge to $\mu$ as $h$ goes to zero.
	
	Now, let us denote by $\mathcal{S}$ and $\mathcal{S}_h$ the
	eigenspace associated to the eigenvalue $\mu$ and the spanned
	of the eigenspaces associated to $\mu_h^{(1)},\ldots,\mu_h^{(m)}$, respectively.
	
	We also recall the definition of the \textit{gap} $\hat{\delta}$ between two closed
	subspaces $\mathcal{X}$ and $\mathcal{Y}$ of $L^2(\O)$:
	$$
	\hat{\delta}(\mathcal{X},\mathcal{Y})
	:=\max\left\{\delta(\mathcal{X},\mathcal{Y}),\delta(\mathcal{Y},\mathcal{X})\right\},$$
	where
	$$
	\delta(\mathcal{X},\mathcal{Y})
	:=\sup_{\mathbf{x}\in\mathcal{X}:\
		\left\|x\right\|_{0,\O}=1}\delta(x,\mathcal{Y}),
	\quad\text{with }\delta(x,\mathcal{Y}):=
	\inf_{y\in\mathcal{Y}}\|x-y\|_{0,\O}.$$
	We also define  $$\gamma_h:=\sup\limits_{v\in \mathcal{S}:\Vert v\Vert_{0,\O}=1}
	\Vert(T-T_h)v\Vert_{0,\O}.$$

	The following error estimates for the approximation
	of eigenvalues and eigenfunctions hold true.
	The result can be obtained from Theorems~7.1 and~7.3 from \cite{BO}.
	
	\begin{theorem}
		\label{gap}
		There exists a strictly positive constant $C$ such that
		\begin{align}
			\hat{\delta}(\mathcal{S},\mathcal{S}_h)
			& \leq C \gamma_h,\nonumber\\
			\left|\mu-\mu_h^{(j)}\right|
			& \le C \gamma_h \qquad \forall j=1,\ldots,m.\nonumber
		\end{align}
	\end{theorem}
	
	Moreover, employing the additional regularity of the eigenfunctions (cf. Lemma~\ref{LEM:REG}(ii)),
	we immediately obtain the following bound.
	
	\begin{theorem}\label{gapr}
		There exist $s> 1/2$ and $C>0$, independent of $h$, such that
		\begin{align}
			&\Vert(T-T_h)v\Vert_{0,\O}
			\le C h^{\min\{s,\ell\}+\frac{1}{2}}||v||_{0,\O}\qquad \forall
			v\in \mathcal{S}.\label{bou_gamma_h}
		\end{align}
		Moreover, if $\varepsilon  = 1$, there exists $C > 0$,
		independent of $h$, such that
		\begin{align}
			&\Vert(T-T_h)v\Vert_{0,\O}
			\le C h^{2\min\{s,\ell\}}||v||_{0,\O}\qquad \forall
			v\in \mathcal{S},\label{bou2_gamma_h}
		\end{align}
		and as a consequence,
		%\begin{align*}
		$$
		\begin{aligned}
			&\gamma_h \leq C h^{\min\{s,\ell\}+\frac{1}{2}}\qquad \varepsilon\in\{0,-1\}, %\label{bound1r3}\\
			&\gamma_h \leq C h^{2\min\{s,\ell\}}\qquad \varepsilon=1, %\label{bound1rwe}
		\end{aligned}
		$$
		%\end{align*}
		where $\ell\ge1$ represents the polynomial degree of the method.
	\end{theorem}
	
	\begin{proof}
		The inequality~\eqref{bou_gamma_h} is obtained repeating the proof of Theorem~\ref{convergT}.
		Estimate \eqref{bou2_gamma_h} also follows from Theorem~\ref{convergT}
		by noticing that $\vert\varepsilon-1\vert=0$ in \eqref{dfghtrf}.
	\end{proof}
	
	The error estimate from the previous theorem yields a similar one
	for the eigenvalues $\l=\frac{1}{\mu}$ of Problem~\ref{P2}.
	
	\begin{remark}\label{rem:burman-penalty-free}
		%We observe that since the method is unconditionally stable because of Lemma~\ref{eliptD}, there will be no spurious eigenvalues. 
%		Note that the theoretical convergence rates proposed in Theorem~\ref{gapr}
%		are satisfied for all $\alpha>0$. 
		We note that if we consider $\alpha=0$ and $\varepsilon=-1$,
		we have the penalty-free scheme proposed in \cite{burman2012penalty}. In this case,
		we need an inf-sup result for $A_h(\cdot,\cdot)$ and the theory proposed in this section
		can be adapted to show the convergence of the method.
		We will report some numerical test illustrating the behavior of this scheme
		(cf. Section~\ref{subsec:rectangle-domain}).
	\end{remark}
	
		\section{A posteriori error analysis}
		\label{sec:apost}
The aim of this section is to introduce and analyze a residual-based a posteriori error estimator for the Laplace eigenvalue problem discretized by the symmetric Nitsche method. In what follows, we restrict the analysis to the fully Dirichlet case considered in Problem~\ref{P2-discrete}, and to simple eigenvalues. Let $(\lambda,u)$ be a simple eigenpair of the continuous problem and let $(\lambda_h,u_h)\in\mathbb{R}\times V_h$ be the corresponding discrete eigenpair, normalized by
\[
   \|u\|_{0,\Omega}=1,\qquad \|u_h\|_{0,\Omega}=1,
\]
with the sign chosen so that $(u,u_h)_{0,\Omega}>0$.
\\
We denote by $\mathcal E_h^0$ the set of interior facets of $\mathcal T_h$, while $\mathcal E_h$ denotes the set of boundary facets contained in $\Gamma$, as in Section~\ref{sec:fem}. For each facet $F$, we set
\[
   h_F:=\operatorname{diam}(F).
\]
If $F=K^+\cap K^-\in\mathcal E_h^0$, with unit outward normals $ n^\pm$ to $K^\pm$, we define
\[
   [\![\nabla v_h\cdot n]\!] := \nabla v_h|_{K^+}\cdot n^+  + \nabla v_h|_{K^-}\cdot n^-.
\]
For a boundary facet $F\subset\Gamma$, we recall that $ n$ denotes the unit outward normal to $\Omega$.

%		The aim of this section is to introduce a suitable
%		residual-based error estimator for the Nitsche FEM for the laplace
%		eigenvalue problem. Moreover, on the forthcoming analysis we will focus only on single eigenvalues. \cred{Definir las caras internas y la $h_{\mathcal{F}}$}.
		
		%As a consequence of the mesh regularity assumptions,
		%we have that each  triangle $K\in\CT_h$ admits a sub-triangulation $\CT_h^{K}$
		%obtained by joining each vertex of $K$ with the midpoint of the ball with respect
		%to which $K$ is starred. Let $\CT_h:=\bigcup_{K\in\CT_h}\CT_h^{K}$.
		%Since we are also assuming \textbf{A2}, $\big\{\CT_h\big\}_{0 < h \leq 1}$
		%is a shape-regular family of triangulations of $\O$. 
		
		\subsection{Residual-based a posteriori error estimator}
		In this section, we introduce an a posteriori error estimator based on residuals. We focus on the symmetric Nitsche formulation, that is, $\varepsilon=1$, since this is the variant for which the optimal a priori eigenvalue convergence is obtained. ```latex
We focus on the symmetric Nitsche formulation, that is, $\varepsilon=1$, since this is the variant for which the optimal a priori eigenvalue convergence is obtained. For each element $K\in\mathcal T_h$, we define the element residual
\[
   R_K:=\lambda_h u_h+\Delta u_h \qquad \text{in }K.
\]
The local error indicator is given by
\[
   \eta_K^2 := \eta_{R,K}^2+\eta_{J,K}^2+\eta_{\Gamma,K}^2,
\]
where
\[
   \eta_{R,K}^2 := h_K^2\|R_K\|_{0,K}^2,
\]
\[
   \eta_{J,K}^2 := \sum_{F\subset\partial K\cap\mathcal E_h^0}h_F\|[\![\nabla u_h\cdot n]\!]\|_{0,F}^2,
\]
and
\[
   \eta_{\Gamma,K}^2 := \sum_{F\subset\partial K\cap\mathcal E_h}h_F^{-1}\|u_h\|_{0,F}^2.
\]
%Here, $\alpha>0$ is the Nitsche stabilization parameter introduced in Section~\ref{sec:fem}. 
%In the symmetric case considered here, $\alpha$ is assumed to be sufficiently large so that the bilinear form $A_h(\cdot,\cdot)$ is coercive in the mesh-dependent norm. 
The global estimator is defined by
\begin{equation}\label{errest1}
   \eta_h := \left(\sum_{K\in\mathcal{T}_h}\eta_K^2\right)^{1/2}.
\end{equation}
The stabilization parameter $\alpha$ is not included in the definition of the estimator. Throughout this section, $\alpha$ is assumed to be fixed and sufficiently large to guarantee the coercivity of the symmetric Nitsche bilinear form. Hence, the constants in the reliability estimate may depend on $\alpha$, but they are independent of the mesh size.

%		and the data oscillation term is given by 
%		$$
%		\Theta :=\left(\sum_{K\in\mathcal{T}_h}\Theta_K^2\right)^{1/2},
%		$$
%		where $\Theta_K := \|(2\mu_h)^{-1/2} (\cblue{\mu}-\mu_h)\beps_h (\bu_h) \|_{0,K}$.

We introduce the following saturation assumption \cite{braess1996posteriori}, which is used for the reliability estimate. Let $(\lambda_{h/2},u_{h/2})\in\mathbb{R}\times V_{h/2}$ be the discrete eigenpair associated with the same simple eigenvalue branch, computed on the uniformly refined mesh $\mathcal{T}_{h/2}$. We assume that $u_{h/2}$ is normalized in $L^2(\Omega)$ and oriented consistently with $u_h$. There exists a constant $0<\gamma<1$, independent of $h$, such that
\begin{equation}\label{satassum11}
   \|u-u_{h/2}\|_{h/2}\le \gamma\|u-u_h\|_{h}.
\end{equation}

\subsection{Reliability}

The reliability estimate is obtained by comparing the discrete eigenfunction on the mesh
$\mathcal T_h$ with the one computed on the uniformly refined mesh $\mathcal T_{h/2}$.
The proof follows a two-level argument from \cite{juntunen2009nitsche}, adapted here to the spectral setting. Let
$v\in V_{h/2}$ and let $\overline{v}\in V_h$ be its Lagrange interpolant on the coarse
mesh. We set
\[
   \chi:=v-\overline{v}.
\]
By standard scaling arguments, there holds
\begin{align}
&\sum_{K\in\mathcal T_{h/2}}h_K^{-2}\|\chi\|_{0,K}^2 + \sum_{F\in\mathcal E_{h/2}^0}h_F^{-1}\|\chi\|_{0,F}^2 +
\sum_{F\in\mathcal E_{h/2}} \left(h_F^{-1}\|\chi\|_{0,F}^2 + h_F\|\partial_n \chi\|_{0,F}^2 \right)\leq C \vertiii{v}_{h/2}^2.
\label{eq:lagrange-two-xi-level}
\end{align}
Since $v\in V_{h/2}$, the inverse trace inequality yields
\[
   \vertiii{v}_{h/2}\leq C\|v\|_{h/2}.
\]
Hence, for the function $v$ chosen below with $\|v\|_{h/2}=1$, the right-hand side
of \eqref{eq:lagrange-two-xi-level} is bounded by a constant independent of $h$.

We also introduce the higher-order term
\[
   \theta_h := \|\lambda u-\lambda_hu_h\|_{0,\Omega} + \|\lambda u-\lambda_{h/2}u_{h/2}\|_{0,\Omega}.
\]
For normalized eigenfunctions associated with a simple eigenvalue, $\theta_h$ is of
higher order with respect to the energy error.

\begin{theorem}[Reliability]
Assume that the saturation property \eqref{satassum11} holds for the discrete
eigenfunction associated with the same simple eigenvalue branch. Then there exists a
constant $C_\alpha>0$, independent of $h$ but possibly depending on the fixed Nitsche
stabilization parameter $\alpha$, such that
\[
   \|u-u_h\|_h \leq C_\alpha\big(\eta_h(u_h)+\theta_h\big).
\]
\end{theorem}

\begin{proof}
By the triangle inequality and the saturation assumption \eqref{satassum11}, we have
\[
   \|u-u_h\|_h \leq \frac{1}{1-\gamma}\|u_{h/2}-u_h\|_{h/2}.
\]
Therefore, it is enough to bound $\|u_{h/2}-u_h\|_{h/2}$. Since the symmetric Nitsche
bilinear form is coercive on the refined mesh, there exists $v\in V_{h/2}$, with
$\|v\|_{h/2}=1$, such that
\[
   \|u_{h/2}-u_h\|_{h/2} \leq C_\alpha A_{h/2}(u_{h/2}-u_h,v).
\]
%Let $\overline{v}\in V_h$ be the Lagrange interpolant of $v$ on the coarse mesh and
%set $\chi:=v-\overline{v}$. 
Then
\begin{equation}\label{releq11}
A_{h/2}(u_{h/2}-u_h,v) =A_{h/2}(u_{h/2}-u_h,\chi) + A_{h/2}(u_{h/2}-u_h,\overline{v})=: W_1+W_2.
\end{equation}

We first estimate $W_1$. Since $u_{h/2}$ solves the discrete eigenvalue problem on
the refined mesh,
\[
   A_{h/2}(u_{h/2},\chi) = \lambda_{h/2}(u_{h/2},\chi)_{0,\Omega}.
\]
Thus,
\[
   W_1 = \lambda_{h/2}(u_{h/2},\chi)_{0,\Omega} - A_{h/2}(u_h,\chi).
\]
Adding and subtracting $\lambda_h(u_h,\chi)_{0,\Omega}$ gives
\[
   W_1 = (\lambda_{h/2}u_{h/2}-\lambda_hu_h,\chi)_{0,\Omega} + \left[\lambda_h(u_h,\chi)_{0,\Omega} -  A_{h/2}(u_h,\chi) \right].
\]
The first term is bounded by the higher-order contribution. Indeed, using the
definition of $\theta_h$, the Poincar\'e inequality and the interpolation estimate
\eqref{eq:lagrange-two-xi-level}, we get
\[
   |(\lambda_{h/2}u_{h/2}-\lambda_hu_h,\chi)_{0,\Omega}| \leq C\theta_h\|\chi\|_{0,\Omega} \leq  C\theta_h.
\]

For the second term, an elementwise integration by parts yields
\begin{align}
\lambda_h(u_h,\chi)_{0,\Omega}-A_{h/2}(u_h,\chi) =& \sum_{K\in\mathcal T_h}(R_K,\chi)_{0,K} -\sum_{F\in\mathcal E_h^0} \big([\![\nabla u_h\cdot n]\!],\chi\big)_{0,F} \nonumber\\
&+ \sum_{F\in\mathcal E_h} \Big( (\partial_n\chi,u_h)_{0,F} -\alpha h_F^{-1}(u_h,\chi)_{0,F} \Big),
\label{eq:residual-identity-reliability}
\end{align}
where $R_K:=\lambda_hu_h+\Delta u_h$. We now estimate each term in \eqref{eq:residual-identity-reliability}. By the Cauchy--Schwarz inequality and \eqref{eq:lagrange-two-xi-level},
\[
\left|\sum_{K\in\mathcal T_h}(R_K,\chi)_{0,K}\right| \leq \left( \sum_{K\in\mathcal T_h}h_K^2\|R_K\|_{0,K}^2 \right)^{1/2} \left(
\sum_{K\in\mathcal T_h}h_K^{-2}\|\chi\|_{0,K}^2 \right)^{1/2} \leq C\eta_h(u_h).\]
Similarly,
\[
\left| \sum_{F\in\mathcal E_h^0} \big([\![\nabla u_h\cdot n]\!],\chi\big)_{0,F} \right| \leq \left(\sum_{F\in\mathcal E_h^0}h_F
\|[\![\nabla u_h\cdot n]\!]\|_{0,F}^2 \right)^{1/2} \left( \sum_{F\in\mathcal E_h^0}h_F^{-1}\|\chi\|_{0,F}^2 \right)^{1/2} \leq C\eta_h(u_h).
\]
The boundary contribution is estimated by using the last two terms in
\eqref{eq:lagrange-two-xi-level}. Indeed,
\[
\begin{aligned}
&\left|\sum_{F\in\mathcal E_h} \Big( (\partial_n\chi,u_h)_{0,F} -\alpha h_F^{-1}(u_h,\chi)_{0,F} \Big) \right|\\
&\quad \le\left(\sum_{F\in\mathcal E_h}h_F^{-1}\|u_h\|_{0,F}^2\right)^{1/2}\left[\left(\sum_{F\in\mathcal E_h}h_F\|\partial_n\chi\|_{0,F}^2 \right)^{1/2} + \alpha \left( \sum_{F\in\mathcal E_h} h_F^{-1}\|\chi\|_{0,F}^2 \right)^{1/2} \right].
\end{aligned}
\]
By \eqref{eq:lagrange-two-xi-level} and the fact that
\(\|v\|_{h/2}=1\), the bracketed factor is bounded by a constant depending at most
on the fixed parameter \(\alpha\). Therefore,
\[
\left| \sum_{F\in\mathcal E_h} \Big( (\partial_n\chi,u_h)_{0,F} -\alpha h_F^{-1}(u_h,\chi)_{0,F} \Big) \right|\le C_\alpha \eta_\Gamma(u_h) \le C_\alpha \eta_h(u_h).
\]
Therefore,
\[
   |W_1| \leq C_\alpha\big(\eta_h(u_h)+\theta_h\big).
\]

We now estimate $W_2$. Since $\overline{v}\in V_h$, the discrete eigenvalue equation
on the coarse mesh gives
\[
   A_h(u_h,\overline{v}) = \lambda_h(u_h,\overline{v})_{0,\Omega}.
\]
Hence,
\begin{align*}
W_2 &= A_{h/2}(u_{h/2},\overline{v}) - A_{h/2}(u_h,\overline{v}) \\
&=\lambda_{h/2}(u_{h/2},\overline{v})_{0,\Omega} - \lambda_h(u_h,\overline{v})_{0,\Omega} + \big(A_h-A_{h/2}\big)(u_h,\overline{v}).
\end{align*}
The first two terms are bounded by $C\theta_h$, using the stability of the interpolant
and the normalization of $v$. The last term comes from the change of the Nitsche boundary
weights when passing from $\mathcal T_h$ to $\mathcal T_{h/2}$. Since $\alpha$ is fixed,
this contribution is bounded by
\[
   \left|\big(A_h-A_{h/2}\big)(u_h,\overline{v})\right| \leq  C_\alpha \eta_h(u_h).
\]
Consequently,
\[
   |W_2| \leq C_\alpha\big(\eta_h(u_h)+\theta_h\big).
\]
Combining the bounds for $W_1$ and $W_2$ in \eqref{releq11}, we obtain
\[
   \|u_{h/2}-u_h\|_{h/2} \leq C_\alpha\big(\eta_h(u_h)+\theta_h\big).
\]
The assertion follows from the saturation argument.
\end{proof}
		\subsection{Efficiency}
The following standard estimates for element and facet bubble functions will be used in the proof of the local efficiency bounds; see, for instance, \cite{MR1885308,MR3059294}.

\begin{lemma}[Interior bubble functions]
\label{burbujainterior}
For any $K\in\mathcal T_h$, let $\psi_K$ be the corresponding interior bubble function. Then, there exists a constant $C>0$, independent of $h_K$, such that
\[
   C^{-1}\|q\|_{0,K}^2 \le \int_K \psi_K q^2 \le C\|q\|_{0,K}^2\qquad\forall q\in\mathbb P_k(K),
\]
and
\[
   \|\psi_K q\|_{0,K} + h_K\|\nabla(\psi_K q)\|_{0,K}\le C\|q\|_{0,K} \qquad \forall q\in\mathbb P_k(K).
\]
\end{lemma}
\begin{lemma}[Facet bubble functions]
\label{burbuja}
For any facet $F\subset\partial K$, let $\psi_F$ be the corresponding facet bubble function. Then, there exists a constant $C>0$, independent of $h_F$, such that
\[
   C^{-1}\|q\|_{0,F}^2 \le \int_F \psi_F q^2 \le C\|q\|_{0,F}^2 \qquad \forall q\in\mathbb P_k(F).
\]
Moreover, for all $q\in\mathbb P_k(F)$, there exists an extension, still denoted by $q$, to the element patch $\omega_F$ such that
\[
   h_F^{-1/2}\|\psi_F q\|_{0,\omega_F} + h_F^{1/2}\|\nabla(\psi_F q)\|_{0,\omega_F} \le C\|q\|_{0,F}.
\]
\end{lemma}
\begin{theorem}[Local efficiency]
Let $(\lambda,u)$ be a simple eigenpair and let $(\lambda_h,u_h)$ be the corresponding discrete eigenpair. Then, for every $K\in\mathcal T_h$, there exists a constant $C>0$, independent of $h$, such that
\[
   \eta_{R,K} \le C\left(\|u-u_h\|_{h,K} + h_K\|\lambda u-\lambda_hu_h\|_{0,K} \right),
\]
\[
   \eta_{J,K} \le C\left(\|u-u_h\|_{h,\omega_K} + \left(\sum_{K'\subset\omega_K}h_{K'}^2\|\lambda u-\lambda_hu_h\|_{0,K'}^2\right)^{1/2}\right),
\]
and
\[
   \eta_{\Gamma,K} \le C\|u-u_h\|_{h,K}.
\]
Here $\omega_K$ denotes the union of $K$ and the elements sharing at least one facet with $K$.
\end{theorem}
\begin{proof}
We prove the three estimates separately.

First, we consider the element residual. Let
\[
   R_K:=\lambda_hu_h+\Delta u_h \qquad \text{in } K,
\]
and define the local test function
\[
   \chi_K:=h_K^2\psi_K R_K,
\]
where $\psi_K$ is the interior bubble function associated with $K$. Since $\chi_K$ vanishes on $\partial K$, integration by parts and the identity $-\Delta u=\lambda u$ give
\[
\begin{aligned}
(R_K,\chi_K)_{0,K} &=(\lambda_hu_h+\Delta u_h,\chi_K)_{0,K} \\
&=(\lambda_hu_h-\lambda u,\chi_K)_{0,K} + (\Delta(u_h-u),\chi_K)_{0,K} \\
&= (\lambda_hu_h-\lambda u,\chi_K)_{0,K} - (\nabla(u_h-u),\nabla\chi_K)_{0,K}.
\end{aligned}
\]
Using Lemma~\ref{burbujainterior}, we have
\[
   h_K^2\|R_K\|_{0,K}^2 \le C(R_K,\chi_K)_{0,K}.
\]
Therefore, by the Cauchy--Schwarz inequality and the bubble estimates,
\[
\begin{aligned}
\eta_{R,K}^2 &= h_K^2\|R_K\|_{0,K}^2 \\
&\le C\left( \|\nabla(u-u_h)\|_{0,K} + h_K\|\lambda u-\lambda_hu_h\|_{0,K}\right) \eta_{R,K}.
\end{aligned}
\]
Dividing by $\eta_{R,K}$ gives
\[
   \eta_{R,K} \le C\left(\|u-u_h\|_{h,K} + h_K\|\lambda u-\lambda_hu_h\|_{0,K}\right).
\]

We now estimate the jump contribution. Let $F\in\mathcal E_h^0$ be an interior facet, and set
\[
   J_F:=[\![\nabla u_h\cdot n]\!].
\]
Since the exact solution satisfies flux continuity across interior facets, we have
\[
   J_F=[\![\nabla(u_h-u)\cdot n]\!].
\]
Let $v_F:=h_F\psi_F J_F,$ where $\psi_F$ is the facet bubble function extended to the patch $\omega_F$. Using the facet bubble estimates, we obtain
\[
   h_F\|J_F\|_{0,F}^2 \le C(J_F,v_F)_{0,F}.
\]
Integrating by parts over the two elements in $\omega_F$ sharing $F$, and using $-\Delta u=\lambda u$, we get
\[
\begin{aligned}
(J_F,v_F)_{0,F} &= \sum_{K'\subset\omega_F} \left[ (\nabla(u_h-u),\nabla v_F)_{0,K'} + (\Delta(u_h-u),v_F)_{0,K'}\right] \\
&=
\sum_{K'\subset\omega_F} \left[ (\nabla(u_h-u),\nabla v_F)_{0,K'} + (R_{K'}+\lambda u-\lambda_hu_h,v_F)_{0,K'} \right].
\end{aligned}
\]
Hence, by the Cauchy--Schwarz inequality and Lemma~\ref{burbuja},
$$
h_F\|J_F\|_{0,F}^2 \le C\left[ \|u-u_h\|_{h,\omega_F} \hspace{-0.1cm}+ \left( \sum_{K'\subset\omega_F} h_{K'}^2\|\lambda u-\lambda_hu_h\|_{0,K'}^2 \right)^{1/2} \hspace{-0.2cm}+\hspace{-0.1cm} \left( \sum_{K'\subset\omega_F} \eta_{R,K'}^2 \right)^{1/2} \right] h_F^{1/2}\|J_F\|_{0,F}.
$$
Using the already proved efficiency estimate for $\eta_{R,K'}$, we absorb the residual contribution into the first two terms and obtain
\[
   h_F^{1/2}\|J_F\|_{0,F}\le C\left[\|u-u_h\|_{h,\omega_F}+\left(\sum_{K'\subset\omega_F}h_{K'}^2\|\lambda u-\lambda_hu_h\|_{0,K'}^2\right)^{1/2}\right].
\]
Summing over all interior facets contained in $\partial K$ gives
\[
   \eta_{J,K} \le C\left(\|u-u_h\|_{h,\omega_K} + \left(\sum_{K'\subset\omega_K}h_{K'}^2\|\lambda u-\lambda_hu_h\|_{0,K'}^2
   \right)^{1/2} \right].
\]

Finally, we estimate the boundary contribution. Since the exact eigenfunction satisfies $u=0$ on $\Gamma$, we have, for every boundary facet $F\subset\partial K\cap\mathcal E_h$,
\[
   h_F^{-1}\|u_h\|_{0,F}^2 = h_F^{-1}\|u_h-u\|_{0,F}^2.
\]
Therefore,
\[
   \eta_{\Gamma,K}^2  = \sum_{F\subset\partial K\cap\mathcal E_h}h_F^{-1}\|u_h\|_{0,F}^2 \le \|u-u_h\|_{h,K}^2,
\]
which yields
\[
   \eta_{\Gamma,K}  \le C\|u-u_h\|_{h,K}.
\]
The proof is complete.
\end{proof}
	
	\section{Numerical results}
	\label{sec:numerical-experiments}
	In this section we present a series of numerical tests to assess the performance of the proposed finite element method. In order to implement the scheme, we resort to the DOLFINx library \cite{barrata2023dolfinx,scroggs2022basix}. The meshes have been constructed using GMSH \cite{geuzaine2009gmsh}. The eigenvalue problem is solved using the Scalable Library for Eigenvalue Problem Computations (SLEPc) \cite{hernandez2005slepc}. The General Non-Hermitian Eigenvalue Problem configuration (GNHEP) with TARGET\_MAGNITUDE is considered to deal with the incomplete ($\varepsilon=0$) and skew-symmetric ($\varepsilon=-1$) variants of the Nitsche method, as well as the computation of possible spurious eigenvalues. The convergence rates of the eigenvalues have been  obtained with a standard least square fitting.

	We denote by $N$ the mesh refinement level, whereas $\texttt{dof}$ denotes the number of degrees of freedom.  We denote by $\lambda_{h,i}$ the $i$-th discrete eigenvalue.
	
	The absolute error on the $i$-th eigenvalue in the Nitsche method is denoted by $e(\lambda_i)$ with 
	$$
	e(\lambda_i):=\vert \lambda_{h,i}-\lambda_{i}\vert.
	$$
	Similarly, the error on the $i$-th eigenvalue for the corresponding usual standard finite element method (imposing strongly the boundary condition on $V_h$) is denoted by $e_{std}(\lambda_{i})$.
	
	With the aim of assessing the performance of our estimator,  we consider domains   with singularities in two and three dimensions in order to observe the improvement of the convergence rate.  On each adaptive iteration, we use the blue-green marking strategy to refine each $T'\in \CT_{h}$ whose indicator $\eta_{T'}$ satisfies
		$$
		\eta_{T'}\geq 0.5\max\{\eta_{T}\,:\,T\in\CT_{h} \}.
		$$
		We define the effectivity indexes with respect to $\eta$ and the eigenvalue $\lambda_i$ by 
		$$\eff(\lambda_i):=e(\lambda_i)/\eta^2.$$
	
	Regarding the usual scheme, it is well-known that computing the spectrum yields a finite number of eigenpairs, that we refer to as \textit{physical}. On the other hand, as we will see below, Nitsche's technique tends to overestimate the exact (extrapolated) eigenvalues. Also, from the theory we note that, if we use the same mesh on both methods, a badly chosen Nitsche’s parameter will yield to additional eigenpairs with no physical meaning (spurious). This is of course for the case when the coercivity of $A_h(\cdot,\cdot)$ is broken.

	\subsection{Spectrum study on a rectangle domain}\label{subsec:rectangle-domain}
	In this section we study the incidence of the Nitsche parameter on the correct approximation of the spectrum. To this end we consider the rectangle domain $\Omega:=(0,L)\times(0,2L)$, with $L=1$. The domain boundary is such that $\Gamma:=(0,y)\cup(L,y)$, $y\in[0,2L]$, and $\Sigma:=\partial\Omega\backslash\Gamma$. More precisely, Dirichlet boundary conditions are specified on the short sides of the rectangle, while Neumann boundary are considered on the long ones. An example of the domain is given in Figure \ref{fig:rectangle-domain}.
	
	The exact eigenvalues and eigenfunctions for this configuration are given by
	$$
	\lambda = \frac{\pi}{L}\left(n^2+ \frac{m^2}{4}\right), \quad u=\cos\left(\frac{n\pi x}{L}\right)\sin\left(\frac{m\pi y}{2L}\right).
	$$
	where $n\geq0$ and $m\geq1$. 
	\begin{figure}[!hbt]\centering
		\includegraphics[scale=0.32]{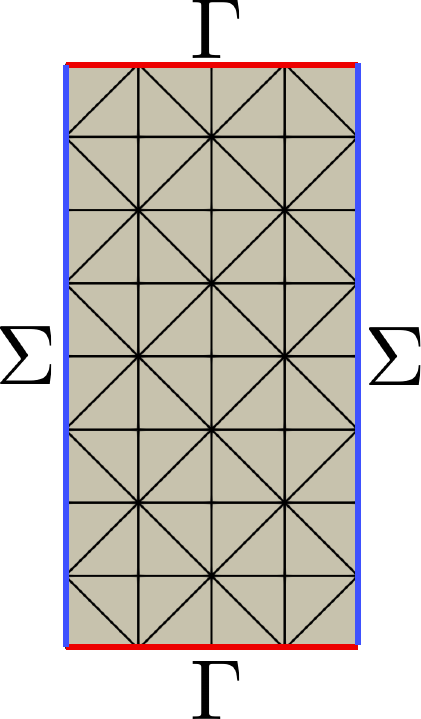}
		\caption{Test \ref{subsec:rectangle-domain}. Rectangle domain $\Omega=(0,L)\times(0,2L)$, $L=1$, with mesh level $N=4$, and prescribed Dirichlet and Neuman boundary conditions on $\Gamma$ and $\Sigma$, respectively.}
		\label{fig:rectangle-domain}
	\end{figure}

	\begin{table}[!hbt]
		\centering
		{\setlength{\tabcolsep}{3.8pt}\footnotesize
			\caption{Test \ref{subsec:rectangle-domain}. First tenth computed eigenvalues $\lambda_{h,i}$ for $k=1$, mesh level $N=8$ and different values of Nitsche parameter $\alpha$.}
			\label{tabla:square-spurious}
			\begin{tabular}{|ccc|ccc|ccc|c|}
				\hline
				\multicolumn{3}{|c|}{$\alpha=0$}                      &           \multicolumn{3}{c}{$\alpha=0.5$}            & \multicolumn{3}{|c|}{$\alpha=1$}                     &         \\ \hline\hline
				$\varepsilon=1$  & $\varepsilon=0$ & $\varepsilon=-1$ & $\varepsilon=1$  & $\varepsilon=0$ & $\varepsilon=-1$ & $\varepsilon=1$ & $\varepsilon=0$ & $\varepsilon=-1$ & Exacts  \\ \hline
				\boxed{-217.3650}& \boxed{0.0000}  &      2.4702      & \boxed{-89.7777} &     2.4713      &      2.4717      &     2.4740      &     2.4726      &      2.4724      & 2.4674  \\
				\boxed{-217.3650} & \boxed{0.0000}  &      9.9119      & \boxed{-89.7777} &     9.9293      &      9.9384      &     9.9747      &     9.9531      &      9.9493      & 9.8696  \\
				\boxed{-90.9595}  &     2.6507      &     12.5202      & \boxed{-61.6008} &     12.5339     &     12.5391      &     12.5643     &     12.5500     &     12.5470      & 12.3370 \\
				\boxed{-90.9595} &     10.6929     &     20.1748      & \boxed{-61.6008} &     20.2453     &     20.2772      &     20.4036     &     20.3329     &     20.3189      & 19.7392 \\
				\boxed{0.0000}    &     10.8748     &     22.4022      &      2.4740      &     22.4901     &     22.5507      &     22.7326     &     22.6273     &     22.6100      & 22.2066 \\
				\boxed{0.0000}     &     12.2053     &     32.9889      &      9.9778      &     33.2048     &     33.3325      &     33.7091     &     33.5053     &     33.4677      & 32.0762 \\
				2.4740     &     15.2092     &     40.0111      &     12.5714      &     40.2759     &     40.5483      &     41.0986     &     40.7965     &     40.7530      & 39.4784 \\
				12.5676         &     23.3702     &     43.9436      &     15.9869      &     44.0228     &     44.0400      &     44.1304     &     44.0882     &     44.0773      & 41.9458 \\
				9.9765     &     24.4006     &     50.9856      &     15.9877      &     51.4935     &     51.9293      &     52.7928     &     52.3566     &     52.2857      & 49.3480 \\
				20.4260    &     37.6534     &     52.1185      &     20.4469      &     52.4612     &     52.5501      &     52.9342     &     52.7591     &     52.7151      & 49.3480 \\ \hline\hline
		\end{tabular}}
		
	\end{table}
	
	Table \ref{tabla:square-spurious} presents the results when $N=8$, $\varepsilon\in\{-1,0,1\}$ with different choices of $\alpha$. We note that two spurious eigenvalues with multiplicity 2 are found in the spectrum of the problem in the symmetric Nitsche method when $\alpha=0$ and $\alpha=0.5$. Although for $\alpha=0$ the ellipticity of the bilinear form $A_h$ is not guaranteed according to Lemma \ref{eliptD}, we note that the skew-symmetric variant of the method yield the correct spectrum.
	This is in line with the penalty-free scheme discussed in Remark~\ref{rem:burman-penalty-free}.
	
	An example of the spurious eigenvalues found are depicted in Figure \ref{fig:spurious-eigenvalues}. Boundary layers across $\Gamma$ are observed for the symmetric case, similar to the results from \cite[Section 3]{harari2018spectral}. The first six computed eigenvalues for $\varepsilon=-1$ are presented in Figure \ref{fig:eigenvalues-nonsymetric}, where we observe an excellent agreement with the analytical eigenmodes. It is important to note that, although not observed if this experiment, the eigensolver may report complex eigenvalues for non-symmetric methods ($\varepsilon=0$ or $\varepsilon=-1$), as we will see below.
	
	We also studied the accuracy of the methods when computing a large number of eigenvalues. On Figure \ref{fig:comparison-spurious} we have plotted the spectrum for different choices of $\alpha$ and the exact eigenvalues in each case. For $\alpha=0$ we observe the spurious eigenvalues on $x=0$. An overprediction for $\varepsilon=-1$ is observed for all the choices of $\alpha$ and $N=8$, while the same is observed for all the methods when increasing the value of $\alpha$. For $\alpha=1$ we have a clean spectrum in all the methods, with considerable overprediction of the incomplete and skew-symmetric variants over the symmetric scheme. Note that for $\alpha=10$, there is no difference between the selected method with respect to their accuracy. 
	
	With respect to the above, we made a study of the dependence of the eigenvalues on the stabilization parameter. The results are described in Figure \ref{fig:alfa-dependence-rectangle}. We can observe that from $\alpha=2$ there is no significant difference in the computed eigenvalue. It is worth noting that the precision decimals can be important when convergence rates are required. We also note that the Nitsche method predicts with the same accuracy as the standard method when $\alpha>0$ is big enough. The spectrum cleanliness observed in Table \ref{tabla:square-spurious} for small eigenvalues is also evident in the relative accuracy on Figure \ref{fig:alfa-dependence-rectangle} when $\alpha=1$. Moreover, the accuracy behaves like $\log(1)=0$ for $\alpha>1$.
	
	Finally, we study the convergence of the schemes in Tables \ref{tabla:square-convergence-k1}--\ref{tabla:square-convergence-k3}. We note that the convergence rates behaves like the one predicted in Theorem \ref{gapr}. For $k=1$ the regularity result allows to obtain $\mathcal{O}(h^2)$ for all the methods, while for $k>1$ we observe that the skew-symmetric and incomplete schemes behave roughly like $\mathcal{O}(h^{2k-0.5})$.

	\begin{table}[!hbt]
		\centering
		{\setlength{\tabcolsep}{3.8pt}\footnotesize
			\caption{Test \ref{subsec:rectangle-domain}. Relative error and convergence behavior for the first six lowest computed eigenvalues on the three variants of the Nitsche method with $k=1$. The stabilization parameter is set to be $\alpha=10$.}
			\label{tabla:square-convergence-k1}
			\begin{tabular}{|c|c|cccc|}
				\hline
				$\varepsilon$& Exacts & \multicolumn{4}{c|}{Relative error (rate)} \\
				\hline
				\multirow{6}{0.02\linewidth}{1}&    2.4674 &   6.88e-03  &   1.72e-03 (2.22) &   7.62e-04 (2.12) &   4.28e-04 (2.09)   \\
				&    9.8696 &   2.79e-02  &   6.88e-03 (2.23) &   3.05e-03 (2.13) &   1.72e-03 (2.09)   \\
				&   12.3370 &   4.81e-02  &   1.19e-02 (2.23) &   5.28e-03 (2.13) &   2.97e-03 (2.09)   \\
				&   19.7392 &   8.92e-02  &   2.21e-02 (2.23) &   9.77e-03 (2.13) &   5.49e-03 (2.09)   \\
				&   22.2066 &   6.40e-02  &   1.56e-02 (2.26) &   6.88e-03 (2.14) &   3.86e-03 (2.09)   \\
				&   32.0762 &   1.38e-01  &   3.40e-02 (2.24) &   1.50e-02 (2.13) &   8.44e-03 (2.09)   \\
				\hline
				\multirow{6}{0.02\linewidth}{0}&    2.4674 &   6.82e-03  &   1.71e-03 (2.21) &   7.60e-04 (2.12) &   4.28e-04 (2.08)   \\
				&    9.8696 &   2.76e-02  &   6.85e-03 (2.23) &   3.04e-03 (2.12) &   1.71e-03 (2.09)   \\
				&   12.3370 &   4.80e-02  &   1.19e-02 (2.23) &   5.27e-03 (2.13) &   2.96e-03 (2.09)   \\
				&   19.7392 &   8.88e-02  &   2.20e-02 (2.23) &   9.75e-03 (2.13) &   5.48e-03 (2.09)   \\
				&   22.2066 &   6.35e-02  &   1.55e-02 (2.25) &   6.86e-03 (2.13) &   3.86e-03 (2.09)   \\
				&   32.0762 &   1.37e-01  &   3.39e-02 (2.23) &   1.50e-02 (2.13) &   8.42e-03 (2.09)   \\
				\hline
				\multirow{6}{0.02\linewidth}{-1}&    2.4674 &   6.76e-03  &   1.70e-03 (2.20) &   7.57e-04 (2.12) &   4.27e-04 (2.08)   \\
				&    9.8696 &   2.74e-02  &   6.82e-03 (2.22) &   3.03e-03 (2.12) &   1.71e-03 (2.08)   \\
				&   12.3370 &   4.78e-02  &   1.19e-02 (2.23) &   5.27e-03 (2.13) &   2.96e-03 (2.09)   \\
				&   19.7392 &   8.84e-02  &   2.19e-02 (2.22) &   9.74e-03 (2.13) &   5.48e-03 (2.09)   \\
				&   22.2066 &   6.29e-02  &   1.54e-02 (2.24) &   6.84e-03 (2.13) &   3.85e-03 (2.09)   \\
				&   32.0762 &   1.36e-01  &   3.38e-02 (2.23) &   1.50e-02 (2.13) &   8.41e-03 (2.09)   \\
				\hline
				&$N$&5&10&15 &20\\
				\hline
				\hline
		\end{tabular}}
		
	\end{table}
	
	\begin{table}[!hbt]
		\centering
		{\setlength{\tabcolsep}{3.8pt}\footnotesize
			\caption{Test \ref{subsec:rectangle-domain}. Relative error and convergence behavior for the first six lowest computed eigenvalues on the three variants of the Nitsche method with $k=2$. The stabilization parameter is set to be $\alpha=10$.}
			\label{tabla:square-convergence-k2}
			\begin{tabular}{|c|c|cccc|}
				\hline
				$\varepsilon$& Exacts & \multicolumn{4}{c|}{Relative error (rate)} \\
				\hline
				\multirow{6}{0.02\linewidth}{1}&    2.4674 &   1.13e-05  &   7.13e-07 (4.20) &   1.41e-07 (4.12) &   4.48e-08 (4.08)   \\
				&    9.8696 &   1.76e-04  &   1.13e-05 (4.17) &   2.26e-06 (4.11) &   7.15e-07 (4.08)   \\
				&   12.3370 &   4.27e-04  &   2.77e-05 (4.16) &   5.51e-06 (4.10) &   1.75e-06 (4.08)   \\
				&   19.7392 &   1.28e-03  &   8.45e-05 (4.13) &   1.69e-05 (4.09) &   5.37e-06 (4.07)   \\
				&   22.2066 &   8.62e-04  &   5.69e-05 (4.13) &   1.14e-05 (4.09) &   3.61e-06 (4.07)   \\
				&   32.0762 &   3.02e-03  &   2.06e-04 (4.08) &   4.13e-05 (4.08) &   1.32e-05 (4.06)   \\
				\hline
				\multirow{6}{0.02\linewidth}{0}&    2.4674 &   2.67e-05  &   2.65e-06 (3.51) &   7.15e-07 (3.33) &   2.87e-07 (3.24)   \\
				&    9.8696 &   2.37e-04  &   1.90e-05 (3.83) &   4.54e-06 (3.64) &   1.68e-06 (3.53)   \\
				&   12.3370 &   4.34e-04  &   2.87e-05 (4.13) &   5.80e-06 (4.06) &   1.87e-06 (4.02)   \\
				&   19.7392 &   1.31e-03  &   8.97e-05 (4.08) &   1.85e-05 (4.02) &   6.04e-06 (3.97)   \\
				&   22.2066 &   9.93e-04  &   7.41e-05 (3.95) &   1.65e-05 (3.82) &   5.78e-06 (3.72)   \\
				&   32.0762 &   3.11e-03  &   2.19e-04 (4.03) &   4.54e-05 (4.00) &   1.49e-05 (3.96)   \\
				\hline
				\multirow{6}{0.02\linewidth}{-1}&    2.4674 &   4.13e-05  &   4.48e-06 (3.38) &   1.26e-06 (3.23) &   5.16e-07 (3.16)   \\
				&    9.8696 &   2.94e-04  &   2.63e-05 (3.67) &   6.71e-06 (3.47) &   2.60e-06 (3.37)   \\
				&   12.3370 &   4.41e-04  &   2.96e-05 (4.11) &   6.08e-06 (4.02) &   1.99e-06 (3.97)   \\
				&   19.7392 &   1.35e-03  &   9.46e-05 (4.04) &   2.00e-05 (3.96) &   6.67e-06 (3.89)   \\
				&   22.2066 &   1.12e-03  &   9.03e-05 (3.82) &   2.14e-05 (3.67) &   7.84e-06 (3.56)   \\
				&   32.0762 &   3.19e-03  &   2.32e-04 (3.99) &   4.93e-05 (3.93) &   1.66e-05 (3.87)   \\
				\hline
				&$N$&5&10&15 &20\\
				\hline
				\hline
		\end{tabular}}
		
	\end{table}
	
	\begin{table}[!hbt]
		\centering
		{\setlength{\tabcolsep}{3.8pt}\footnotesize
			\caption{Test \ref{subsec:rectangle-domain}. Relative error and convergence behavior for the first six lowest computed eigenvalues on the three variants of the Nitsche method with $k=3$. The stabilization parameter is set to be $\alpha=10$.}
			\label{tabla:square-convergence-k3}
			\begin{tabular}{|c|c|cccc|}
				\hline
				$\varepsilon$& Exacts & \multicolumn{4}{c|}{Relative error (rate)} \\
				\hline
				\multirow{6}{0.02\linewidth}{1}&    2.4674 &   6.80e-09  &   1.06e-10 (6.22) &   8.95e-12 (6.22) &   1.52e-12 (6.24)   \\
				&   12.3370 &   1.73e-06  &   2.73e-08 (6.20) &   2.40e-09 (6.12) &   4.28e-10 (6.08)   \\
				&   19.7392 &   8.83e-06  &   1.40e-07 (6.19) &   1.23e-08 (6.12) &   2.20e-09 (6.08)   \\
				&   22.2066 &   4.90e-06  &   7.74e-08 (6.20) &   6.80e-09 (6.12) &   1.21e-09 (6.08)   \\
				&   32.0762 &   3.43e-05  &   5.51e-07 (6.18) &   4.86e-08 (6.11) &   8.66e-09 (6.08)   \\
				
				\hline
				\multirow{6}{0.02\linewidth}{0}&    2.4674 &   3.56e-08  &   1.00e-09 (5.33) &   1.24e-10 (5.27) &   2.59e-11 (5.52)   \\
				&    9.8696 &   8.92e-07  &   2.12e-08 (5.59) &   2.50e-09 (5.38) &   5.58e-10 (5.29)   \\
				&   12.3370 &   1.86e-06  &   3.14e-08 (6.10) &   2.94e-09 (5.96) &   5.56e-10 (5.87)   \\
				&   19.7392 &   9.95e-06  &   1.76e-07 (6.03) &   1.70e-08 (5.87) &   3.31e-09 (5.77)   \\
				&   22.2066 &   7.20e-06  &   1.50e-07 (5.78) &   1.64e-08 (5.57) &   3.50e-09 (5.46)   \\
				&   32.0762 &   3.84e-05  &   6.83e-07 (6.02) &   6.61e-08 (5.88) &   1.28e-08 (5.78)   \\
				\hline
				\multirow{6}{0.02\linewidth}{-1}&    2.4674 &   6.27e-08  &   1.85e-09 (5.27) &   2.30e-10 (5.25) &   4.65e-11 (5.63)   \\
				&    9.8696 &   1.32e-06  &   3.47e-08 (5.44) &   4.28e-09 (5.27) &   9.81e-10 (5.20)   \\
				&   12.3370 &   1.98e-06  &   3.53e-08 (6.02) &   3.45e-09 (5.85) &   6.77e-10 (5.74)   \\
				&   19.7392 &   1.10e-05  &   2.09e-07 (5.92) &   2.15e-08 (5.73) &   4.37e-09 (5.61)   \\
				&   22.2066 &   9.35e-06  &   2.18e-07 (5.61) &   2.54e-08 (5.41) &   5.64e-09 (5.31)   \\
				&   32.0762 &   4.23e-05  &   8.08e-07 (5.91) &   8.27e-08 (5.74) &   1.68e-08 (5.63)   \\
				\hline
				&$N$&5&10&15 &20\\
				\hline
				\hline
		\end{tabular}}
	\end{table}
	
	\begin{figure}[!hbt]\centering
		\begin{minipage}{0.49\linewidth}\centering
			\includegraphics[scale=0.298, trim=0cm 1cm 1cm 2cm,clip]{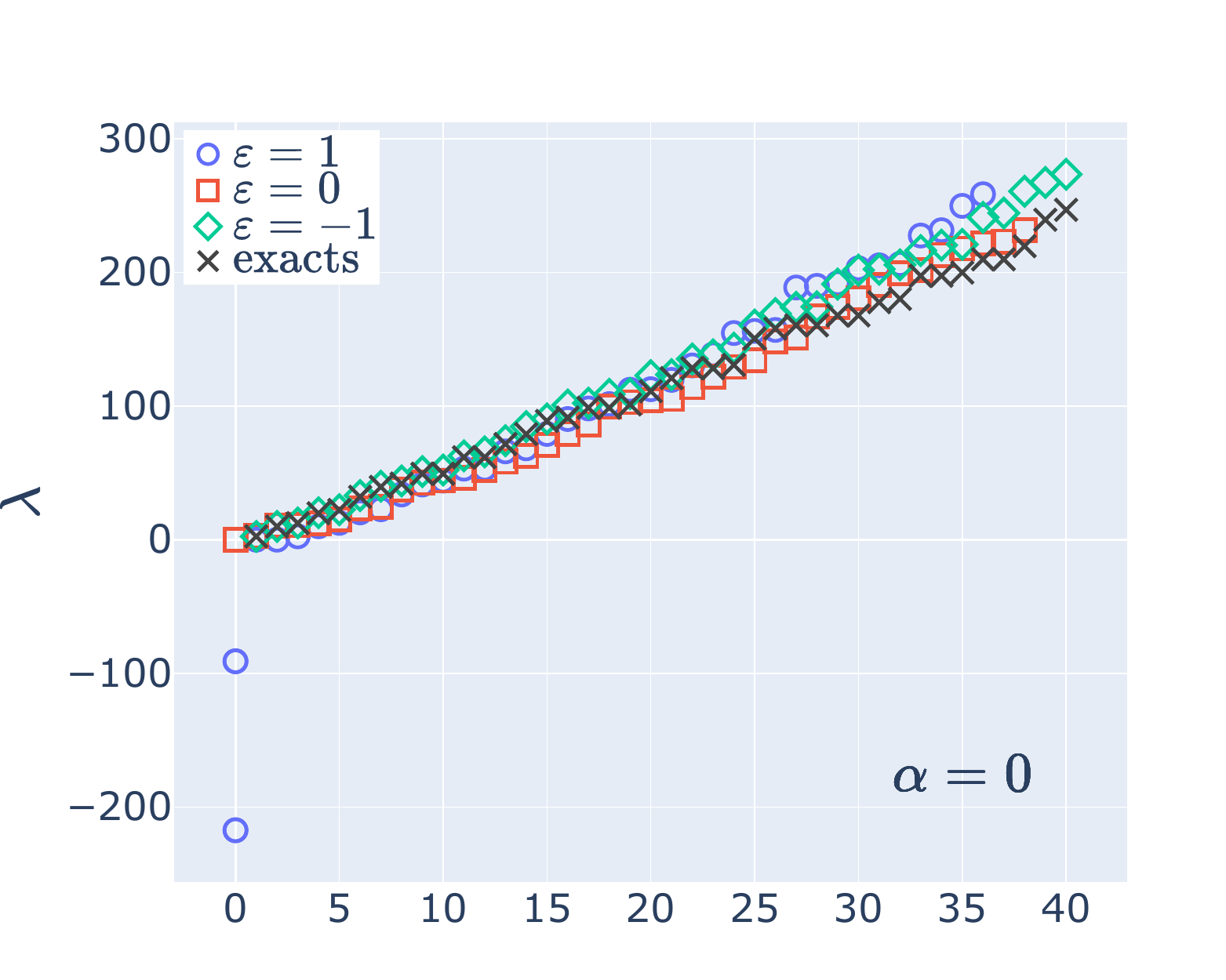}
		\end{minipage}
		\begin{minipage}{0.49\linewidth}\centering
			\includegraphics[scale=0.298, trim=0cm 1cm 1cm 2cm,clip]{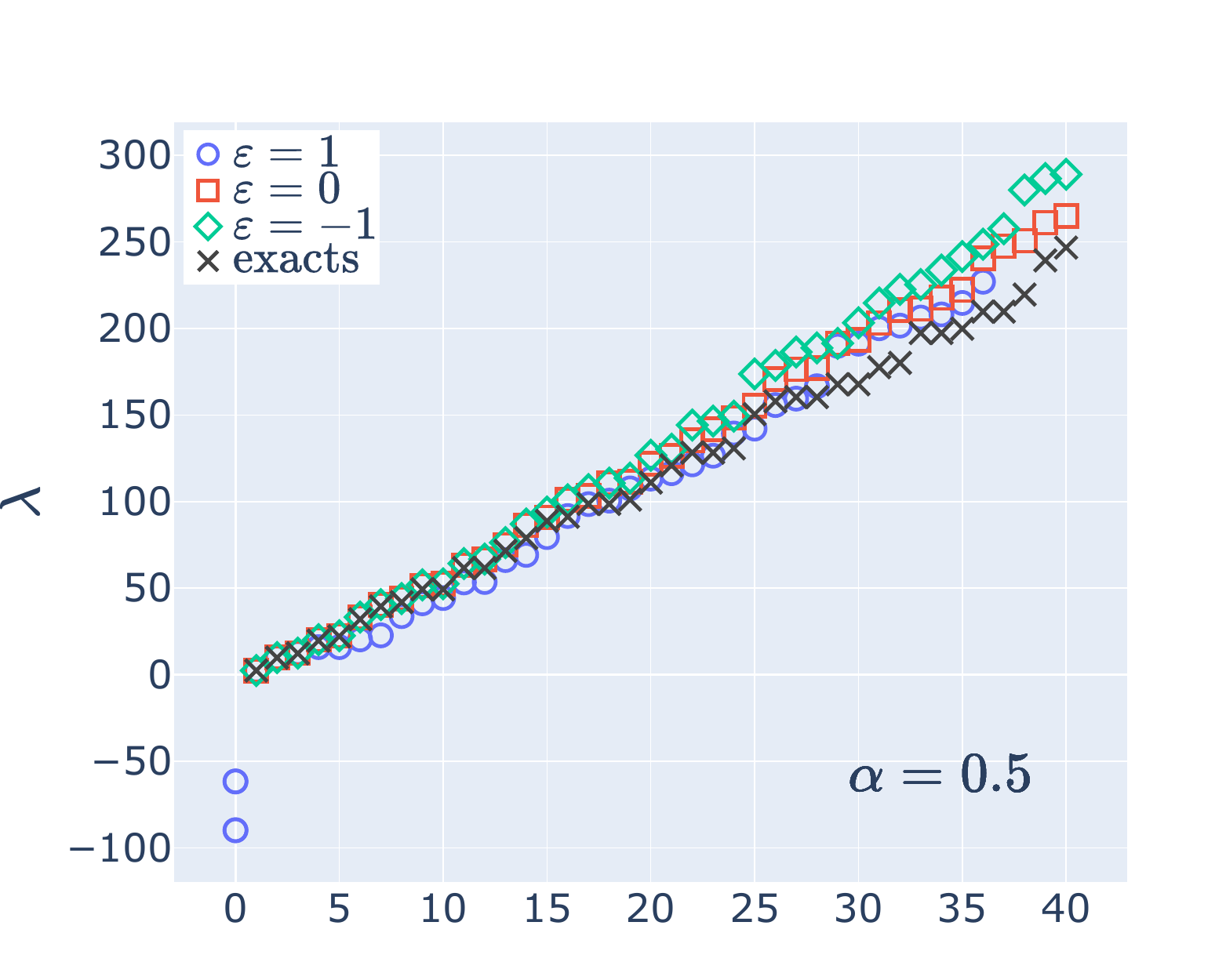}
		\end{minipage}\\
		\begin{minipage}{0.49\linewidth}\centering
			\includegraphics[scale=0.298, trim=0cm 1cm 1cm 2cm,clip]{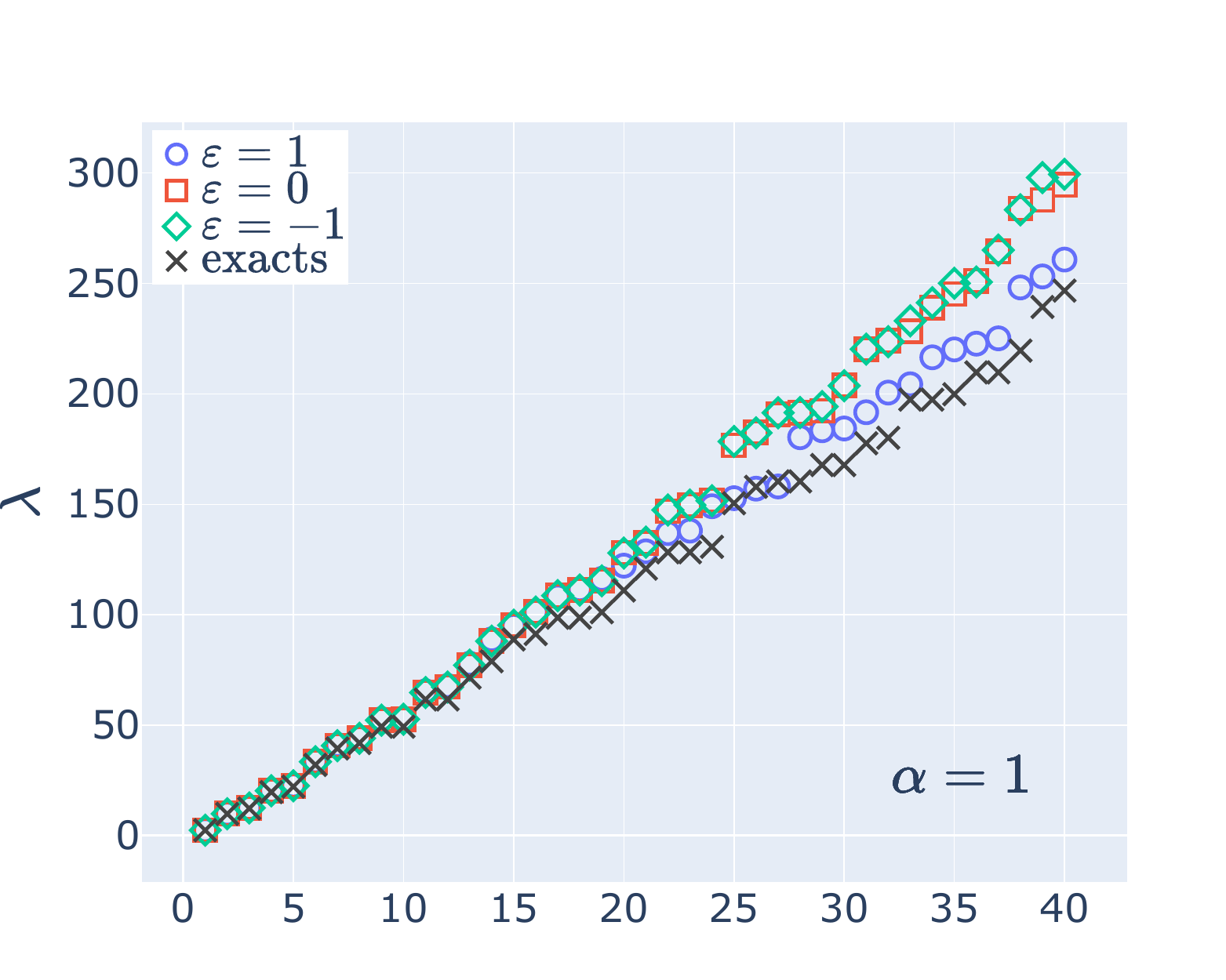}
		\end{minipage}
		\begin{minipage}{0.49\linewidth}\centering
			\includegraphics[scale=0.298, trim=0cm 1cm 1cm 2cm,clip]{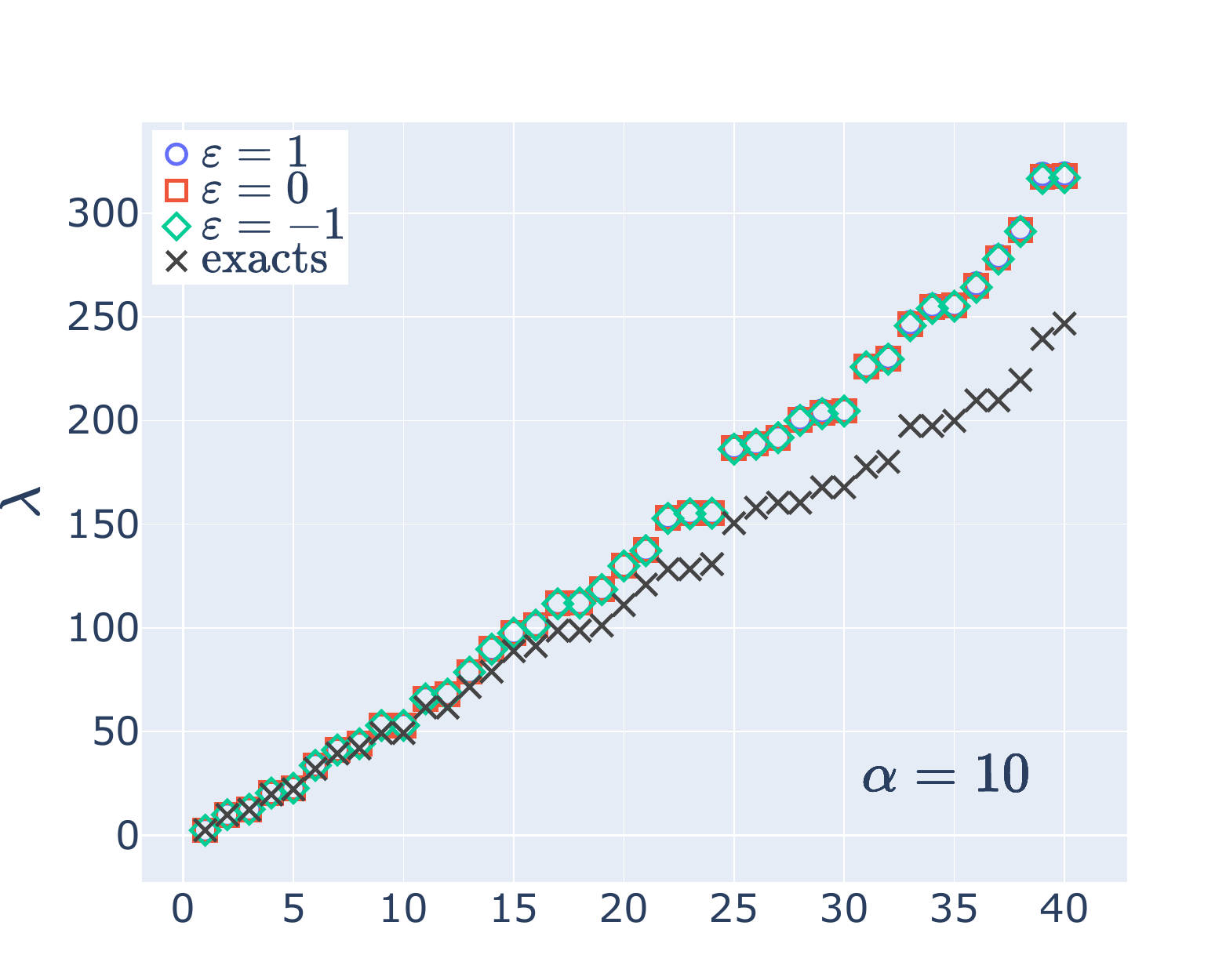}
		\end{minipage}\\
		\caption{Test \ref{subsec:rectangle-domain}. Computed spectrum in the rectangle domain for different choices of $\alpha$ compared with the analytical solutions on the lowest order case $k=1$. Spurious eigenvalues are assigned to x-axis value 0.}
		\label{fig:comparison-spurious}
	\end{figure}
	
	\begin{figure}[!hbt]\centering
		\begin{minipage}{0.49\linewidth}\centering
			\includegraphics[scale=0.298,trim=0cm 0cm 1cm 1cm,clip]{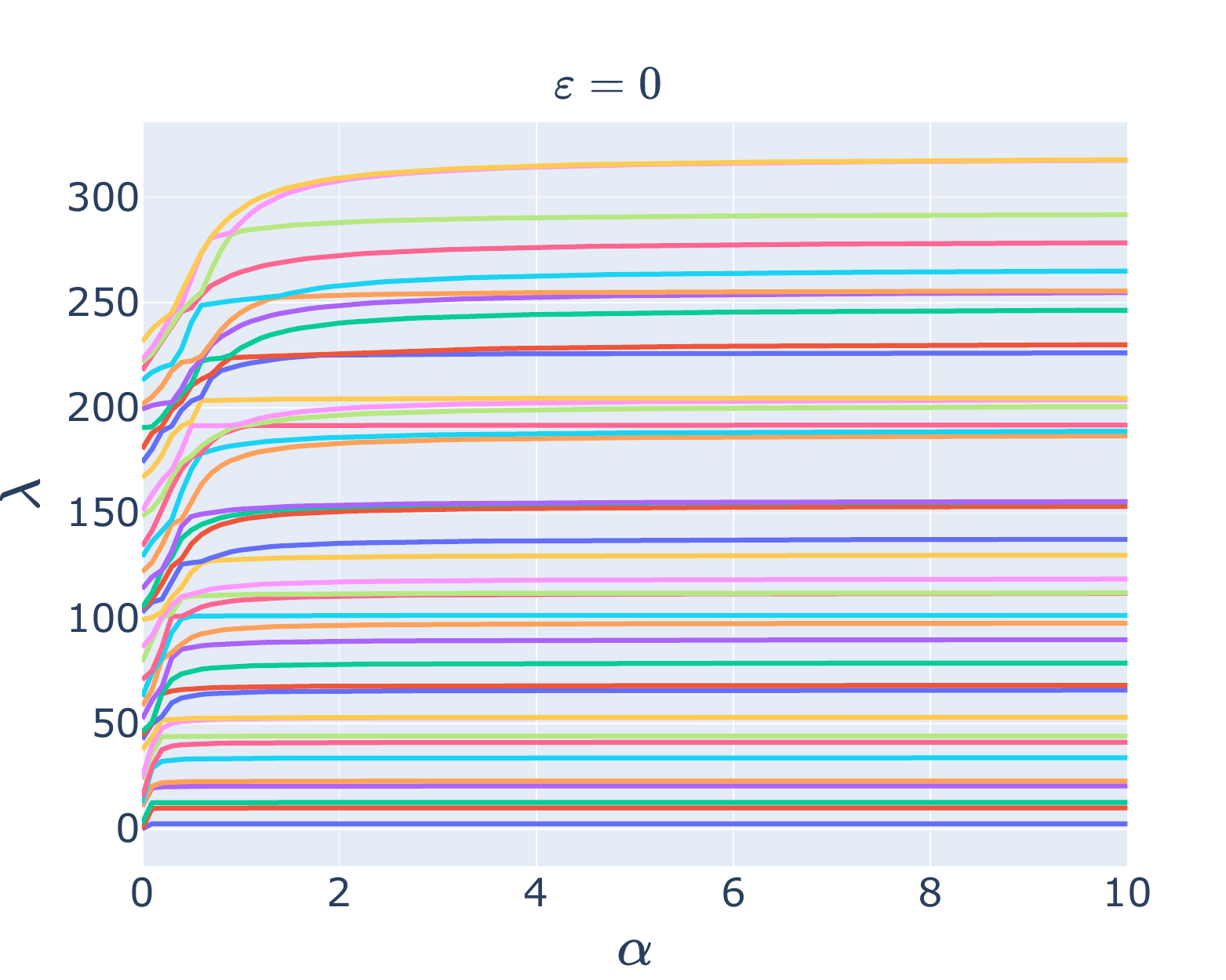}
		\end{minipage}
		\begin{minipage}{0.49\linewidth}\centering
			\includegraphics[scale=0.298,trim=0cm 0cm 1cm 1cm,clip]{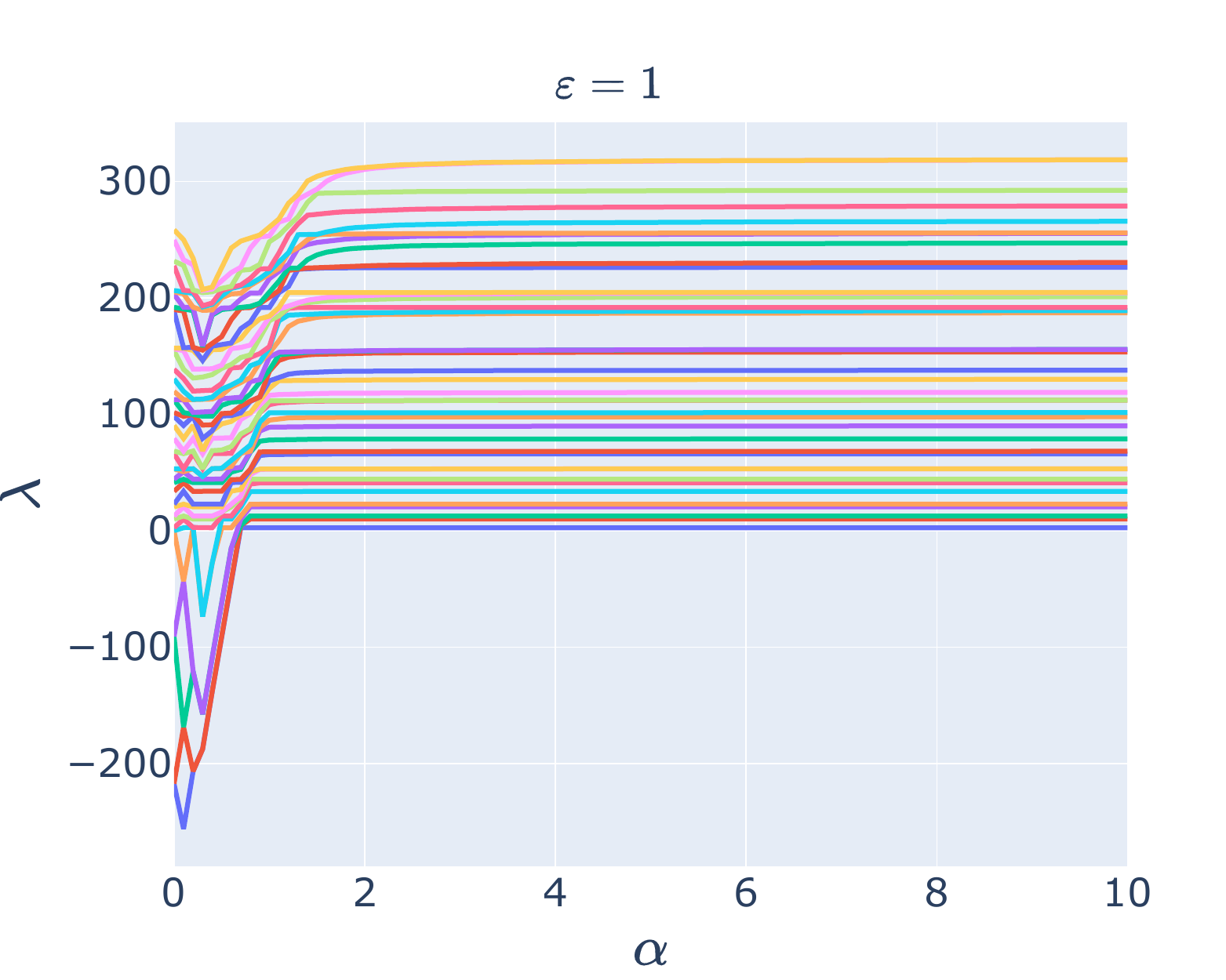}
		\end{minipage}\\
		\begin{minipage}{0.49\linewidth}\centering
			\includegraphics[scale=0.298,trim=0cm 0cm 1cm 1cm,clip]{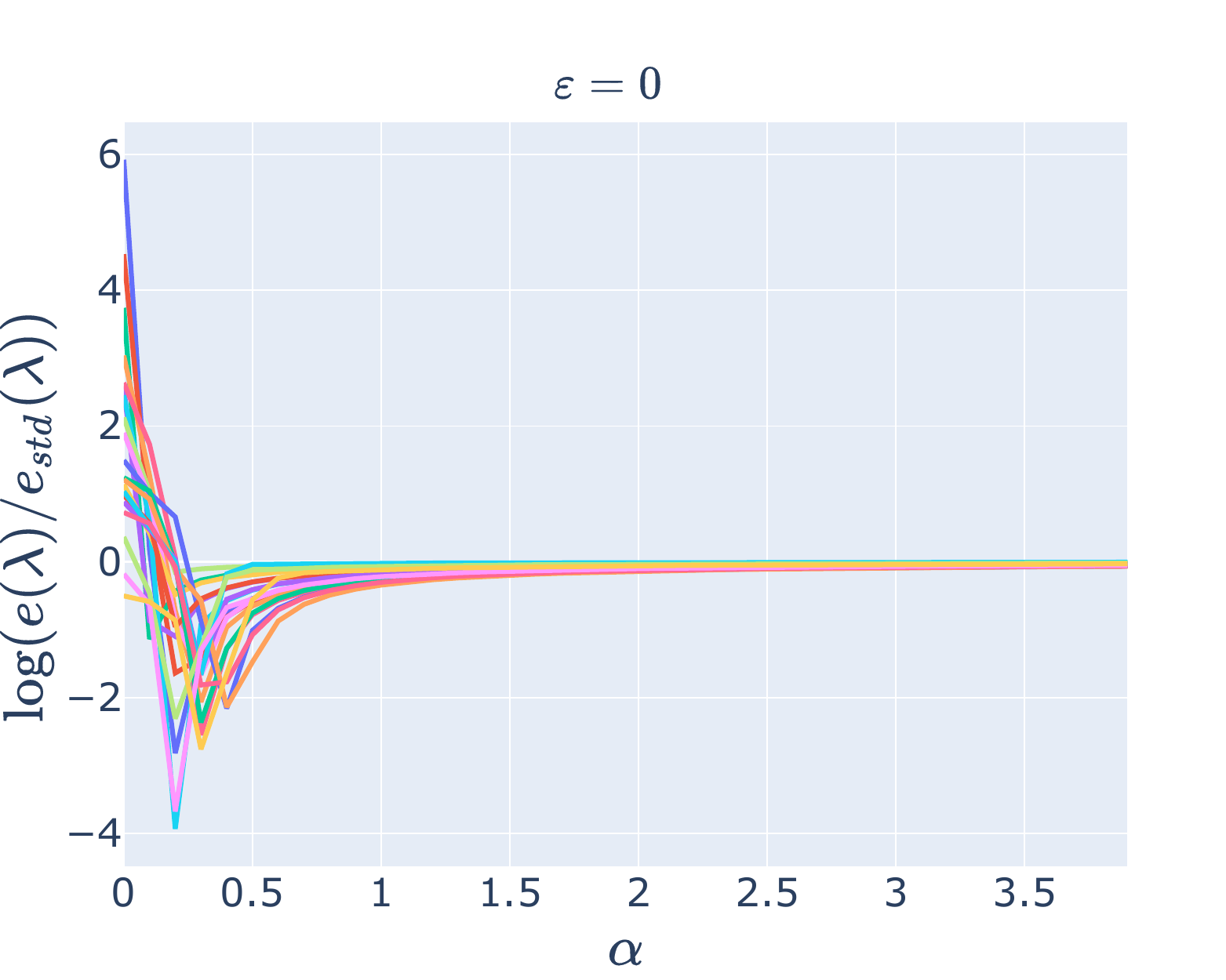}
		\end{minipage}
		\begin{minipage}{0.49\linewidth}\centering
			\includegraphics[scale=0.298,trim=0cm 0cm 1cm 1cm,clip]{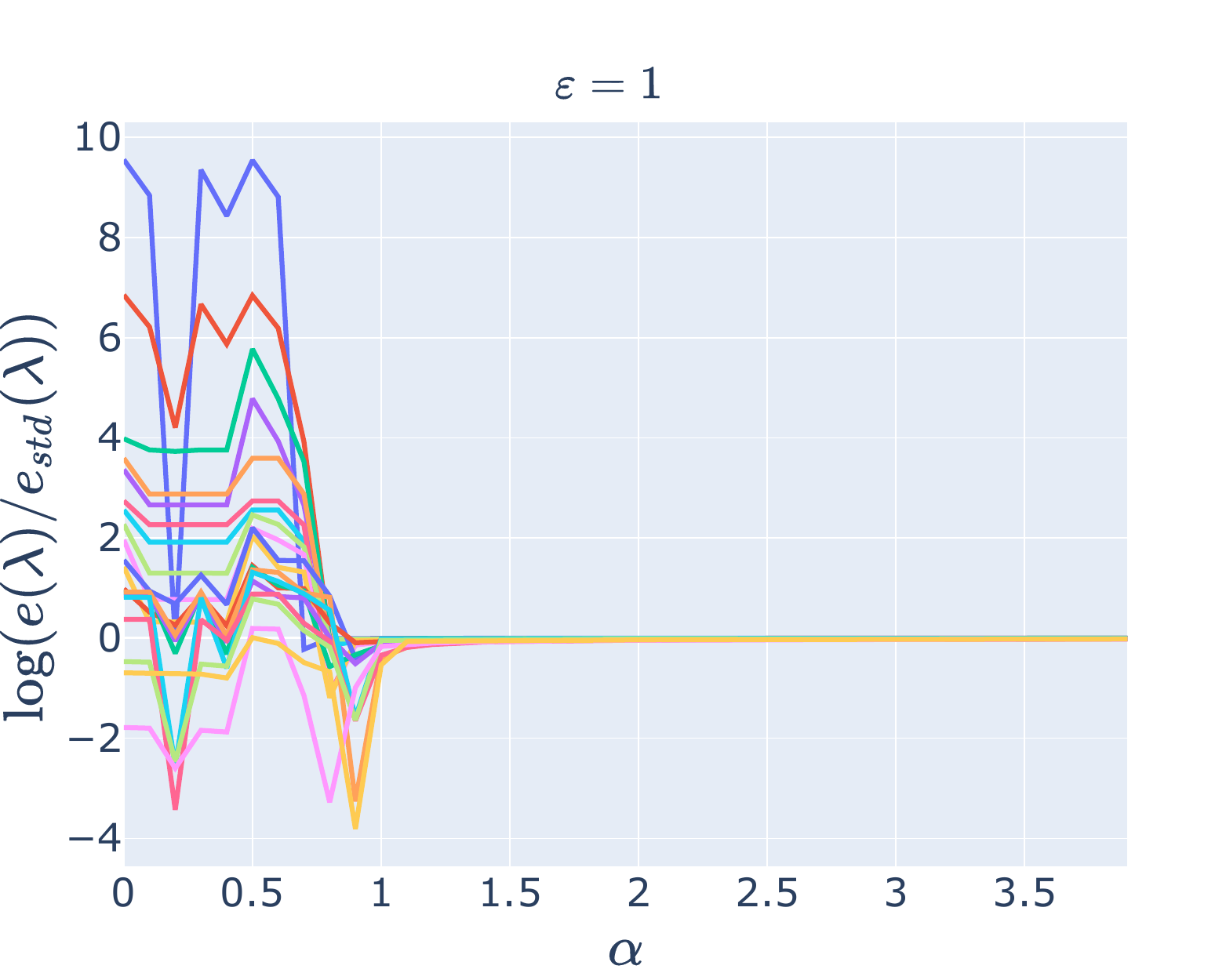}
		\end{minipage}\\
		\caption{Test \ref{subsec:rectangle-domain}. Dependence of the eigenvalues in the rectangle domain when using Nitsche's symmetric and incomplete variants with respect to the stabilization parameter $\alpha$ and $k=1$, $N=5$. Top: computation of the first 40 eigenvalues for each $\alpha$. Bottom: relative accuracy of first 20 computed eigenvalues.}
		\label{fig:alfa-dependence-rectangle}
	\end{figure}
	\begin{figure}[!hbt]\centering
		\begin{minipage}{0.15\linewidth}\centering
			{\footnotesize $\lambda_{h,1},\varepsilon=1$}\\
			\includegraphics[scale=0.12,trim=29cm 1cm 29cm 1cm,clip]{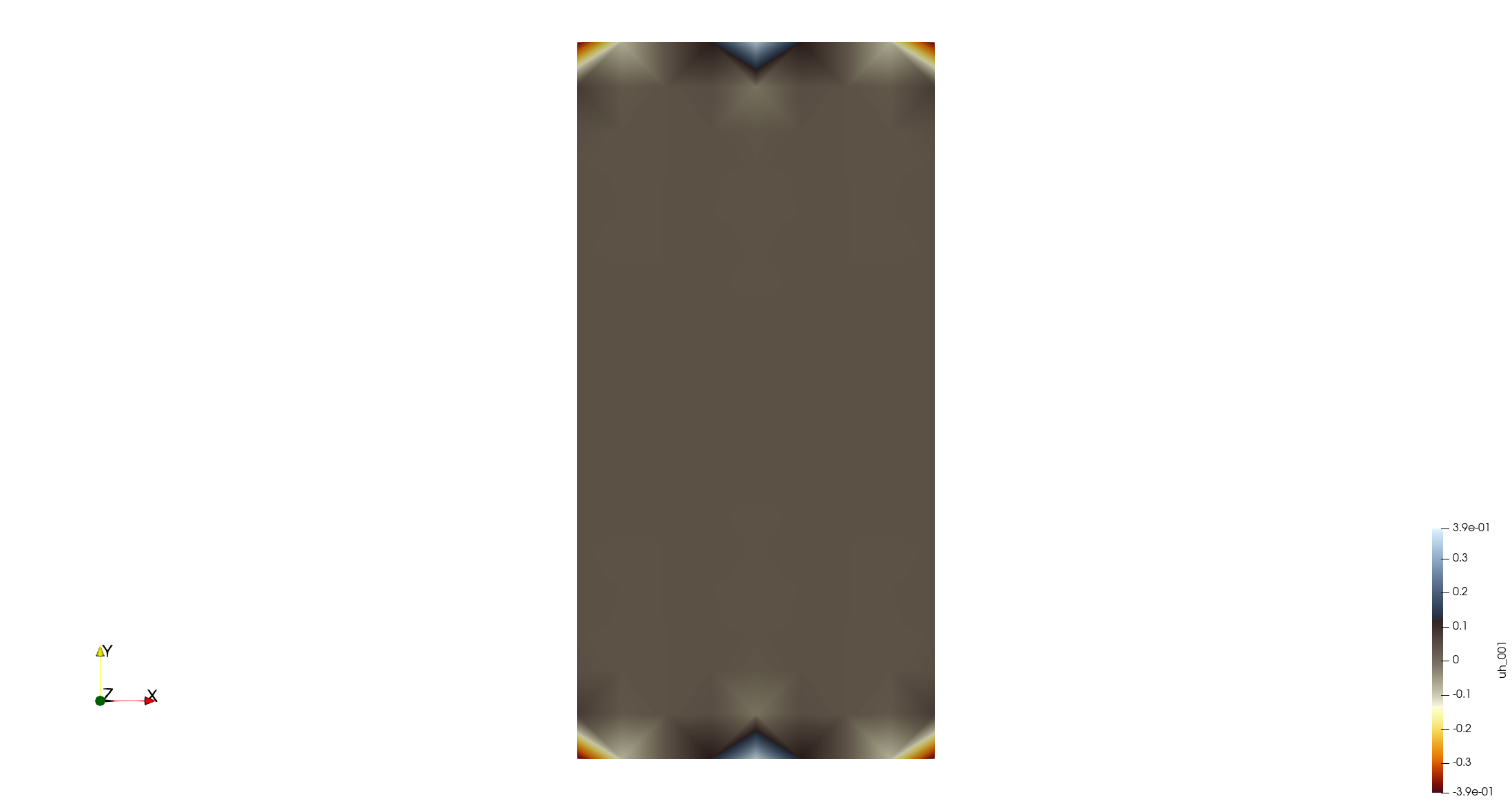}
		\end{minipage}
		\begin{minipage}{0.15\linewidth}\centering
			{\footnotesize $\lambda_{h,3},\varepsilon=1$}\\
			\includegraphics[scale=0.12,trim=29cm 1cm 29cm 1cm,clip]{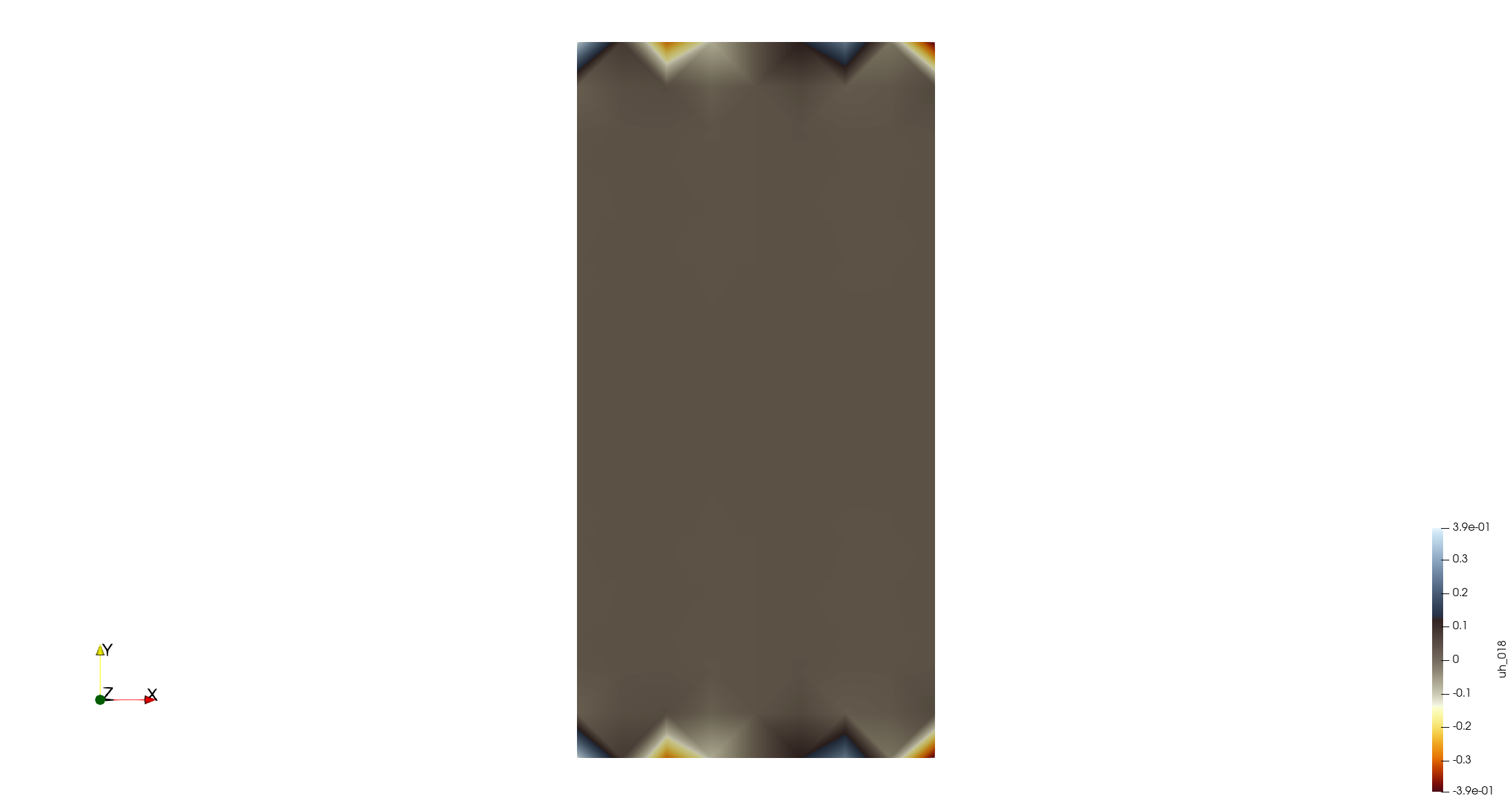}
		\end{minipage}
		\begin{minipage}{0.15\linewidth}\centering
			{\footnotesize $\lambda_{h,5},\varepsilon=1$}\\
			\includegraphics[scale=0.12,trim=29cm 1cm 29cm 1cm,clip]{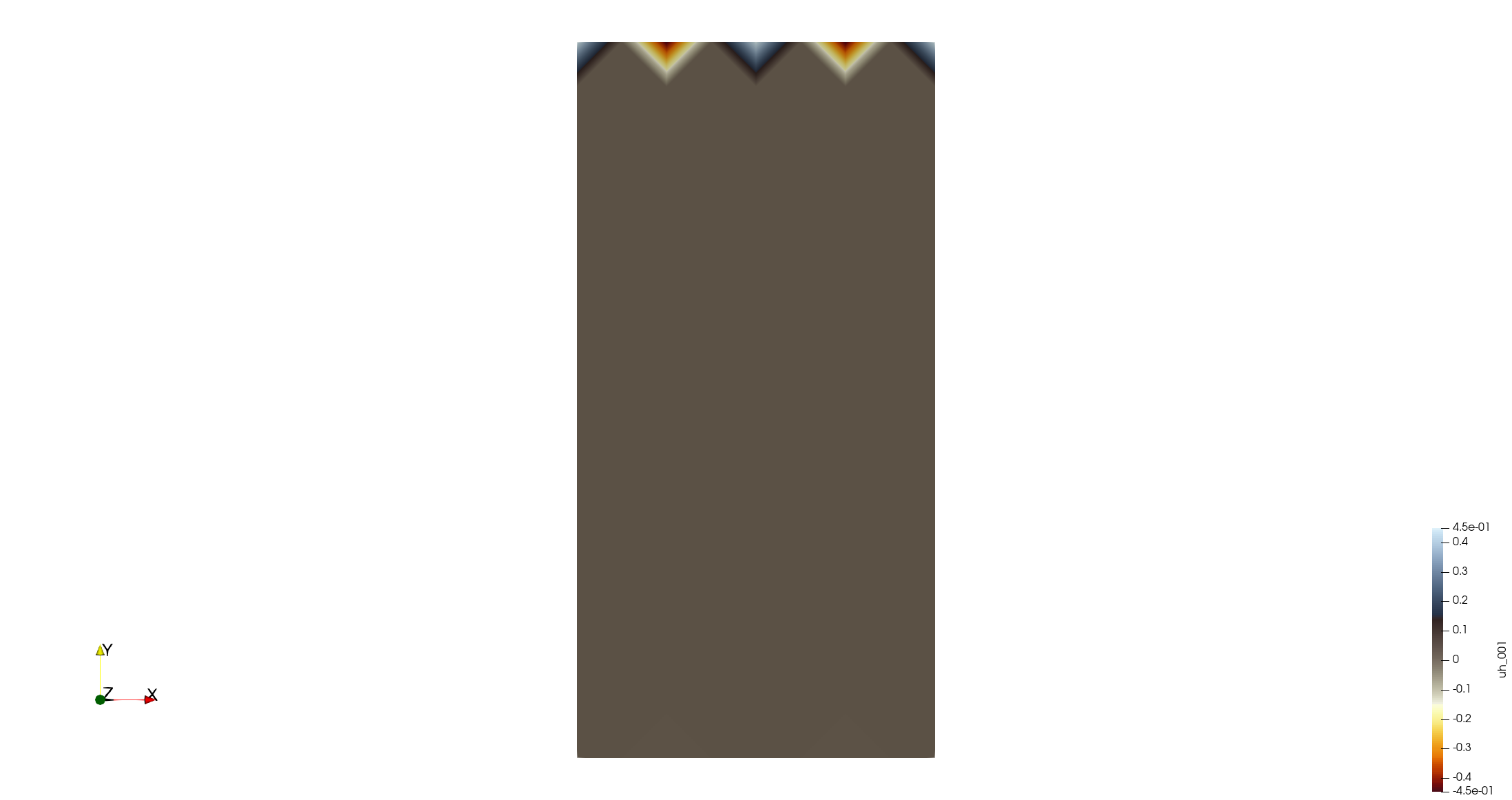}
		\end{minipage}
		\begin{minipage}{0.15\linewidth}\centering
			{\footnotesize $\lambda_{h,6},\varepsilon=1$}\\
			\includegraphics[scale=0.12,trim=29cm 1cm 29cm 1cm,clip]{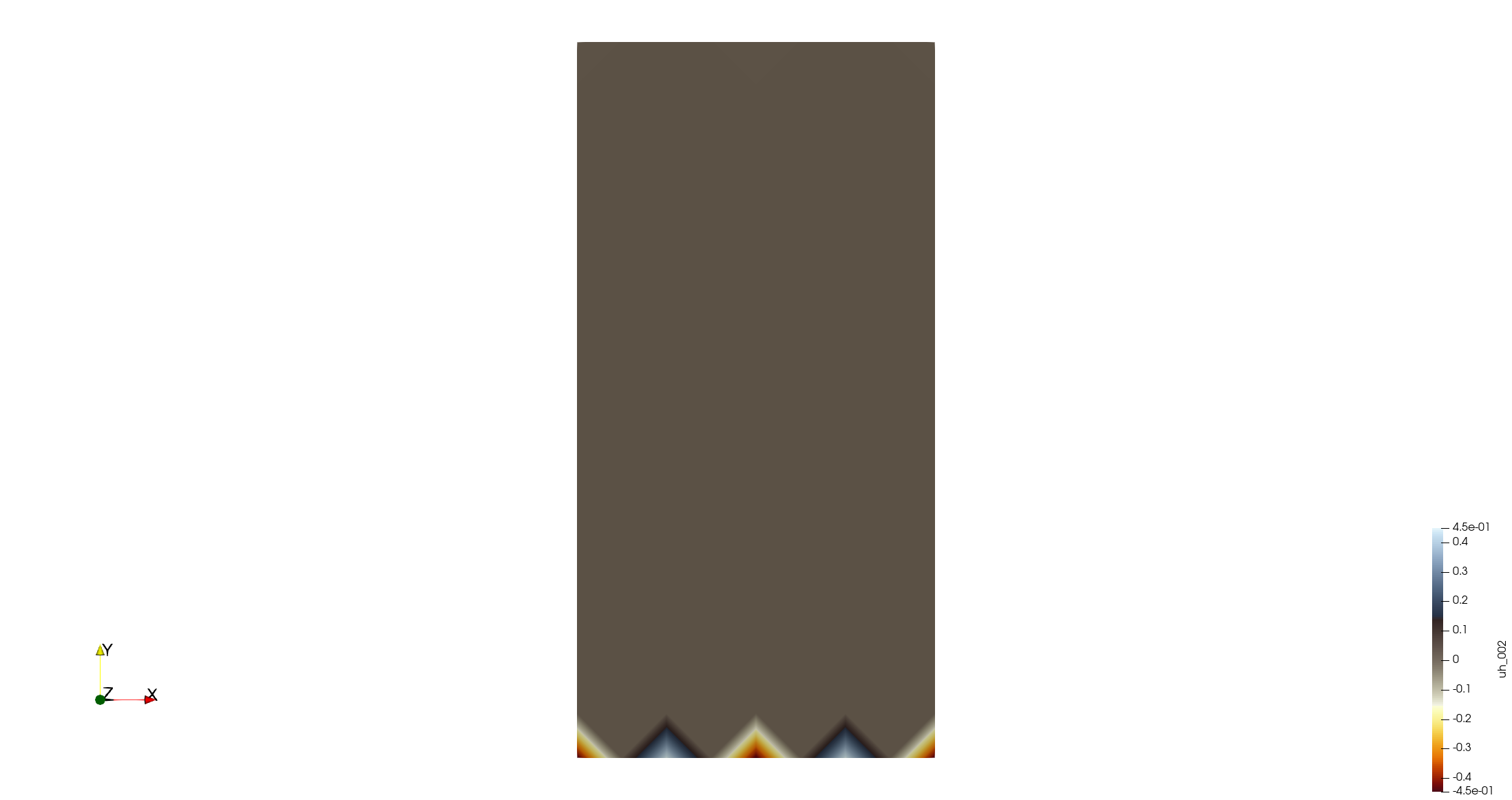}
		\end{minipage}
		\begin{minipage}{0.15\linewidth}\centering
			{\footnotesize $\lambda_{h,1},\varepsilon=0$}\\
			\includegraphics[scale=0.12,trim=29cm 1cm 29cm 1cm,clip]{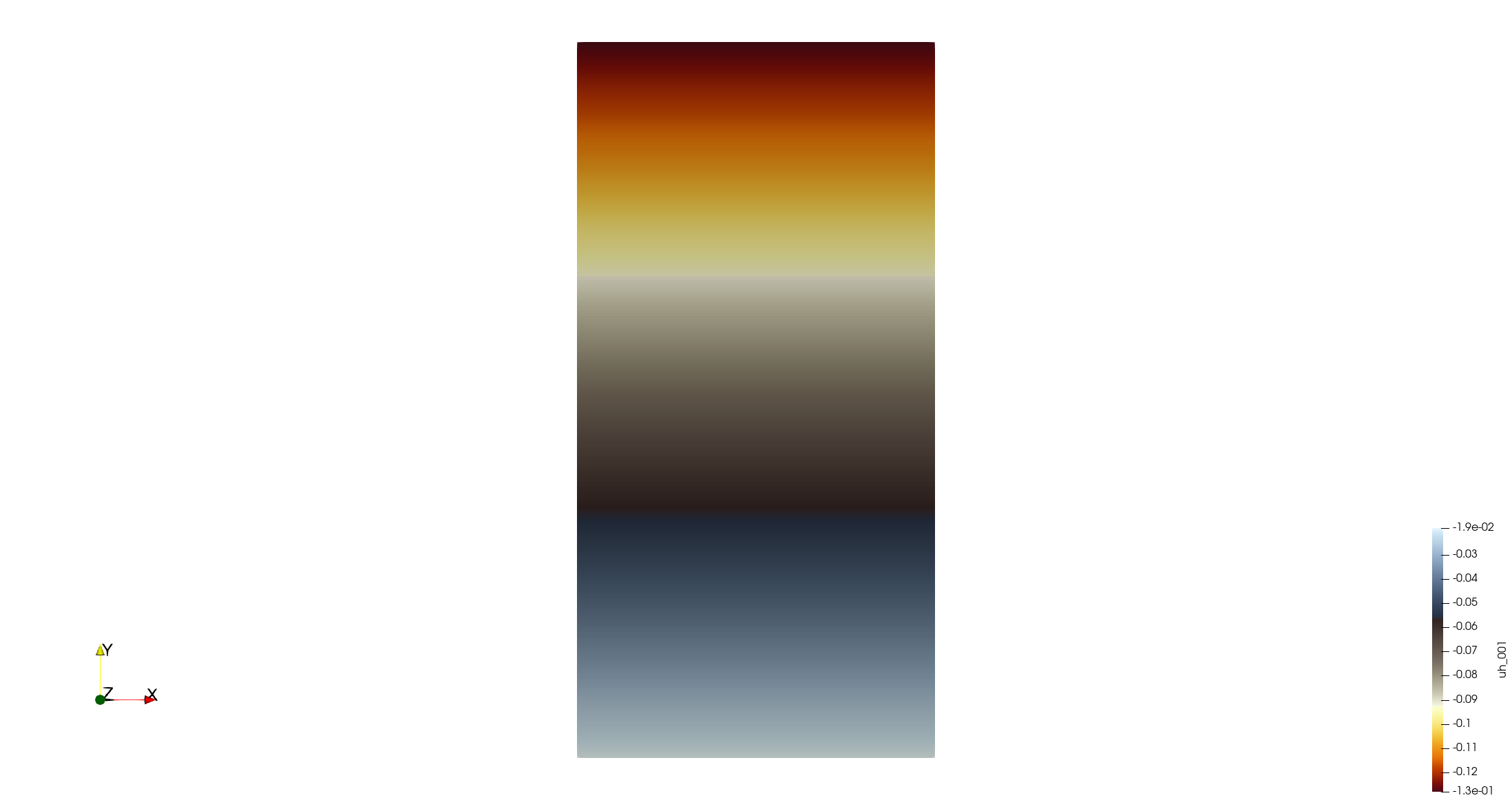}
		\end{minipage}
		\begin{minipage}{0.15\linewidth}\centering
			{\footnotesize $\lambda_{h,2},\varepsilon=0$}\\
			\includegraphics[scale=0.12,trim=29cm 1cm 29cm 1cm,clip]{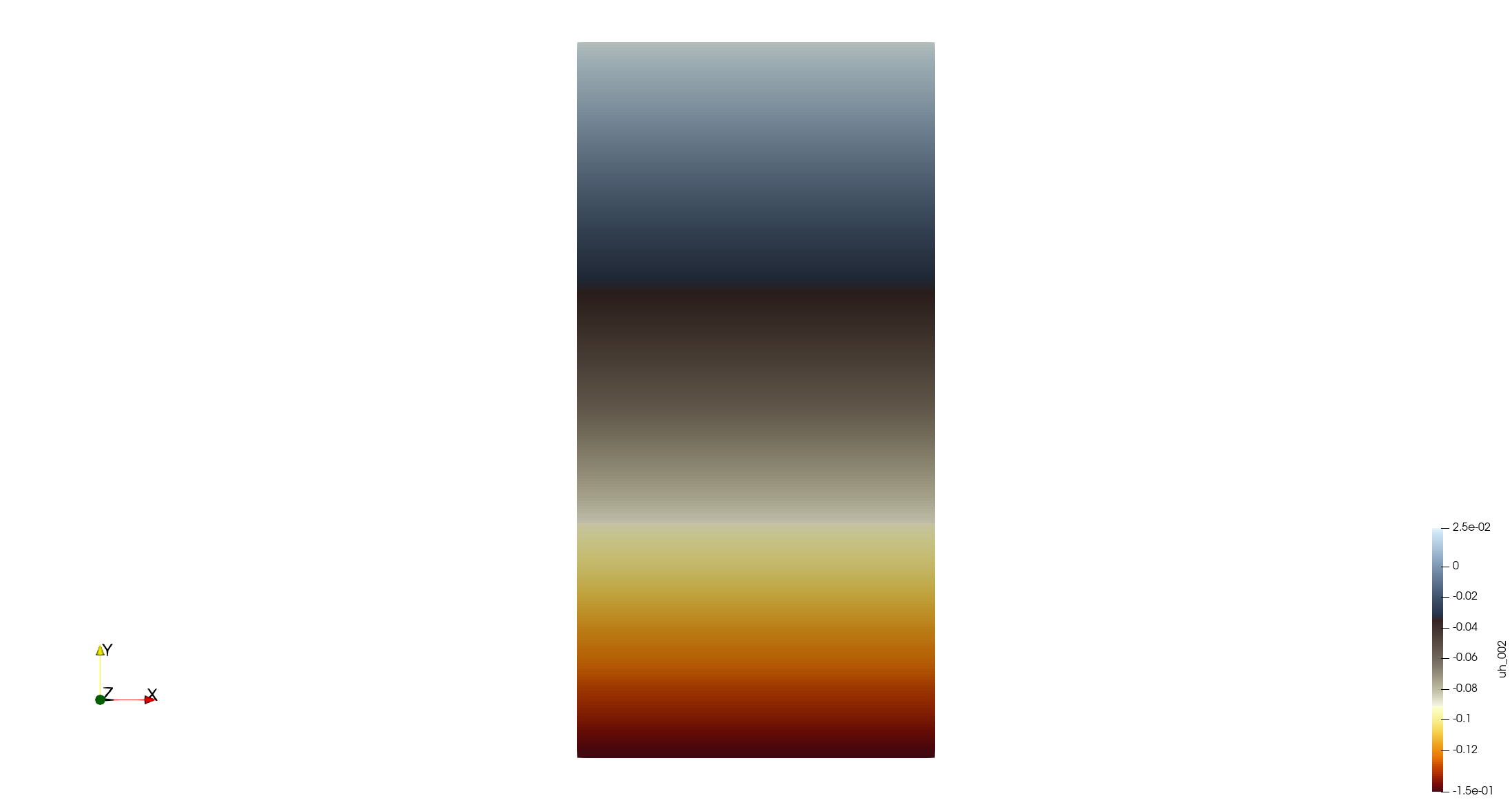}
		\end{minipage}
		\caption{Test \ref{subsec:rectangle-domain}. Comparison between spurious eigenmodes in the rectangle domain when taking $\alpha=0$ on the symmetric ($\varepsilon=1$) and incomplete ($\varepsilon=0$) variants of the Nitsche method for $N=8$.}
		\label{fig:spurious-eigenvalues}
	\end{figure}
	\begin{figure}[!hbt]\centering
		\begin{minipage}{0.15\linewidth}\centering
			{\footnotesize $\lambda_{h,1}$}\\
			\includegraphics[scale=0.12,trim=29cm 1cm 29cm 1cm,clip]{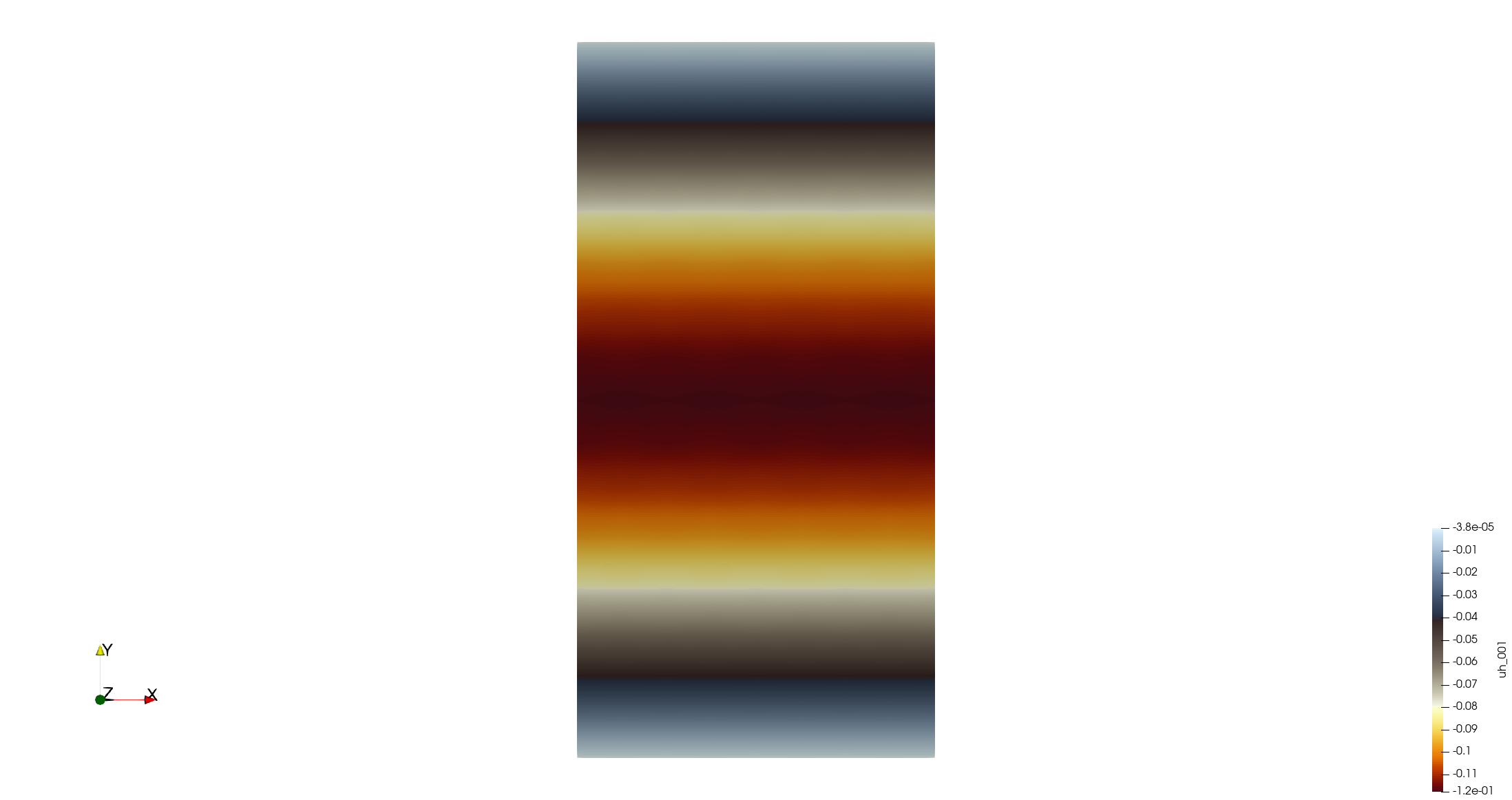}
		\end{minipage}
		\begin{minipage}{0.15\linewidth}\centering
			{\footnotesize $\lambda_{h,2}$}\\
			\includegraphics[scale=0.12,trim=29cm 1cm 29cm 1cm,clip]{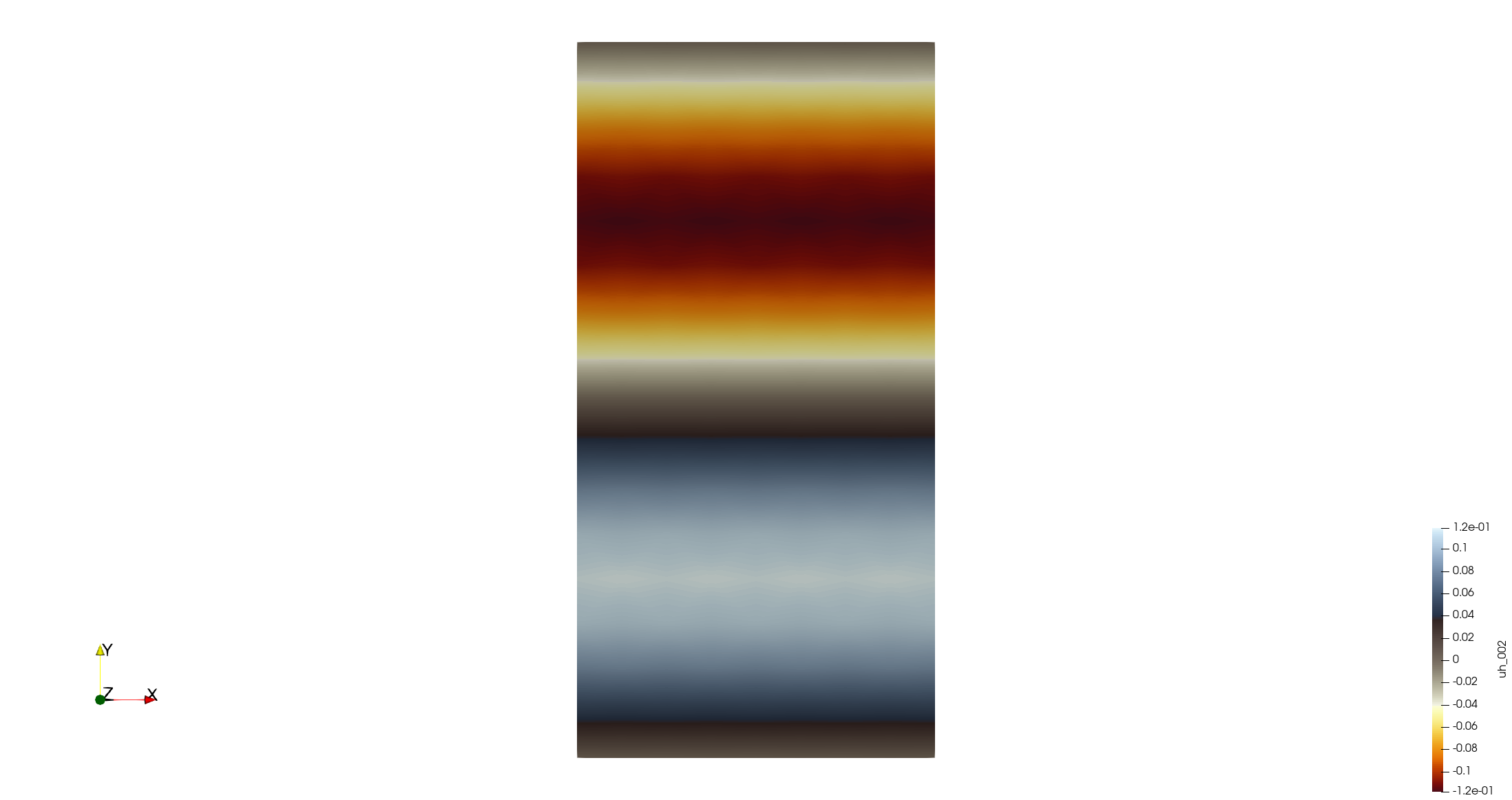}
		\end{minipage}
		\begin{minipage}{0.15\linewidth}\centering
			{\footnotesize $\lambda_{h,3}$}\\
			\includegraphics[scale=0.12,trim=29cm 1cm 29cm 1cm,clip]{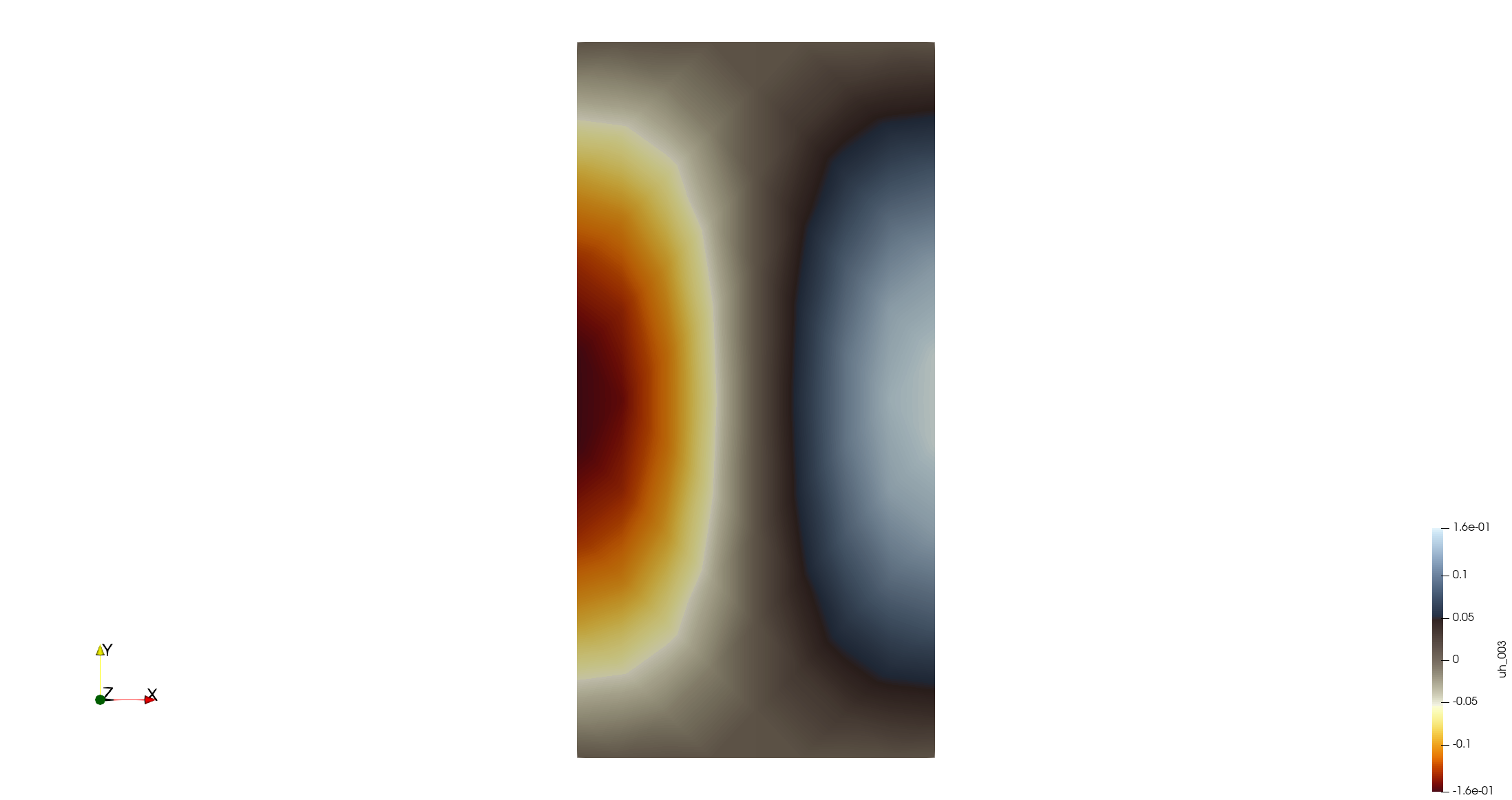}
		\end{minipage}
		\begin{minipage}{0.15\linewidth}\centering
			{\footnotesize $\lambda_{h,4}$}\\
			\includegraphics[scale=0.12,trim=29cm 1cm 29cm 1cm,clip]{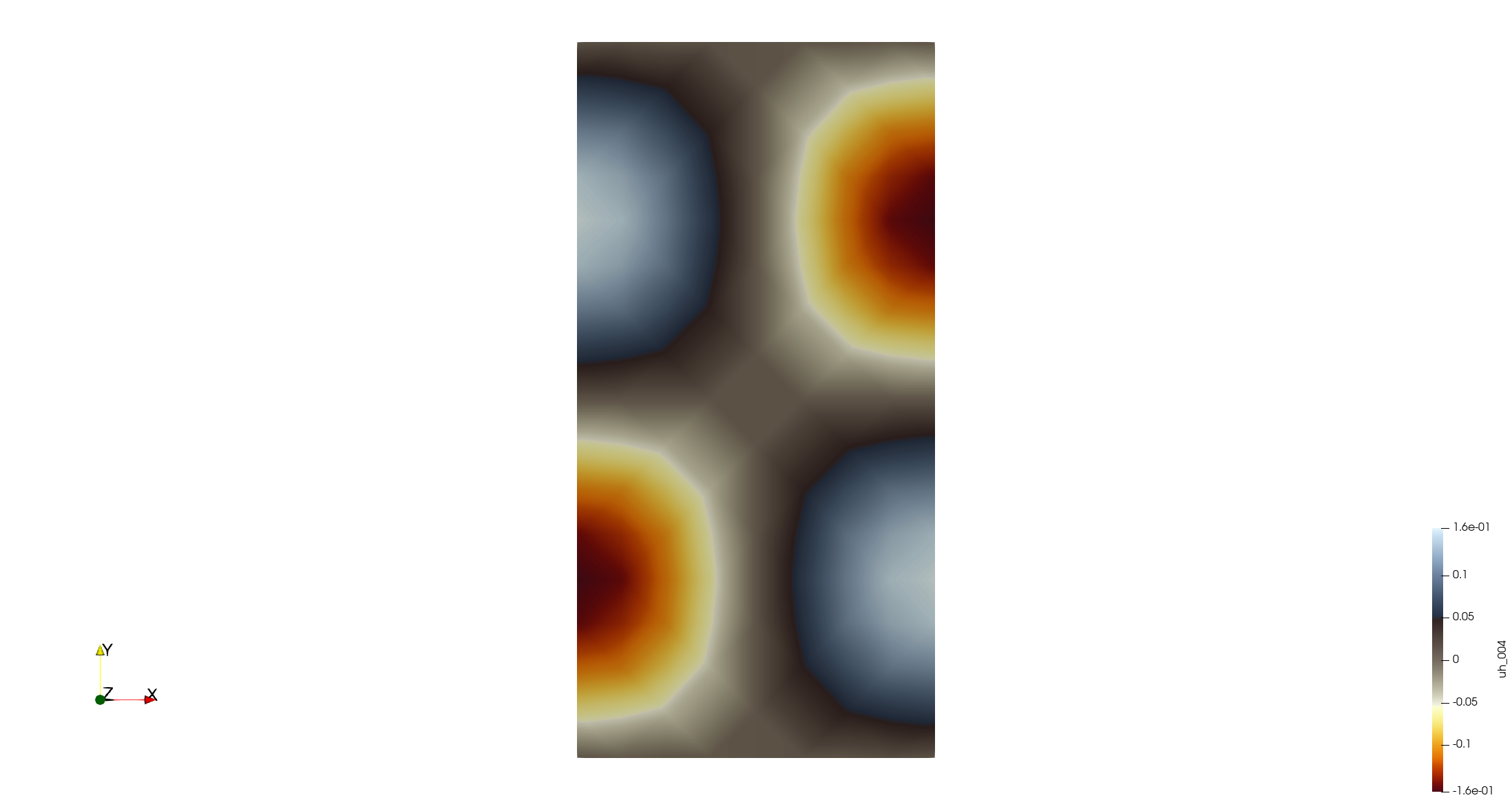}
		\end{minipage}
		\begin{minipage}{0.15\linewidth}\centering
			{\footnotesize $\lambda_{h,5}$}\\
			\includegraphics[scale=0.12,trim=29cm 1cm 29cm 1cm,clip]{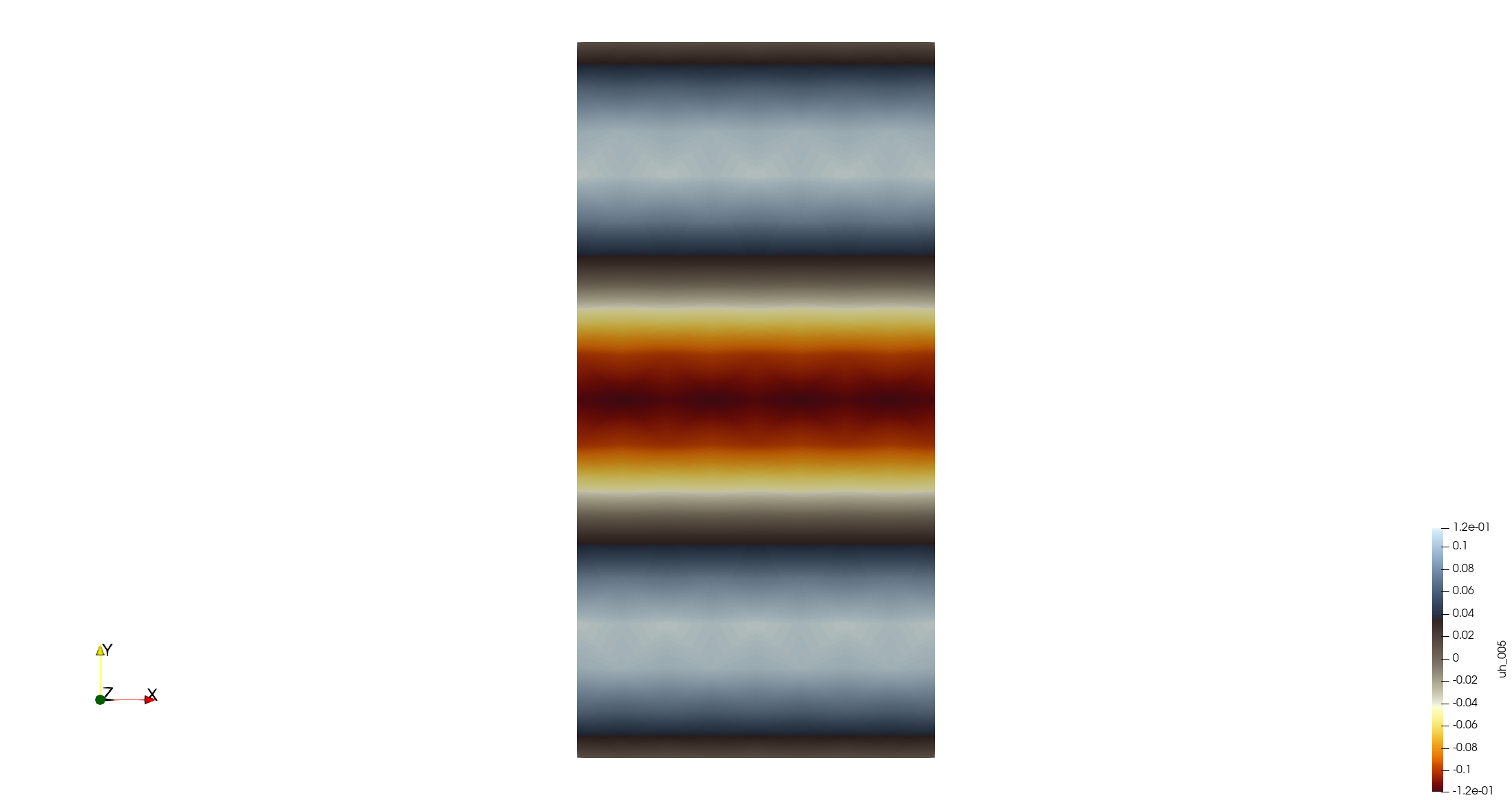}
		\end{minipage}
		\begin{minipage}{0.15\linewidth}\centering
			{\footnotesize $\lambda_{h,6}$}\\
			\includegraphics[scale=0.12,trim=29cm 1cm 29cm 1cm,clip]{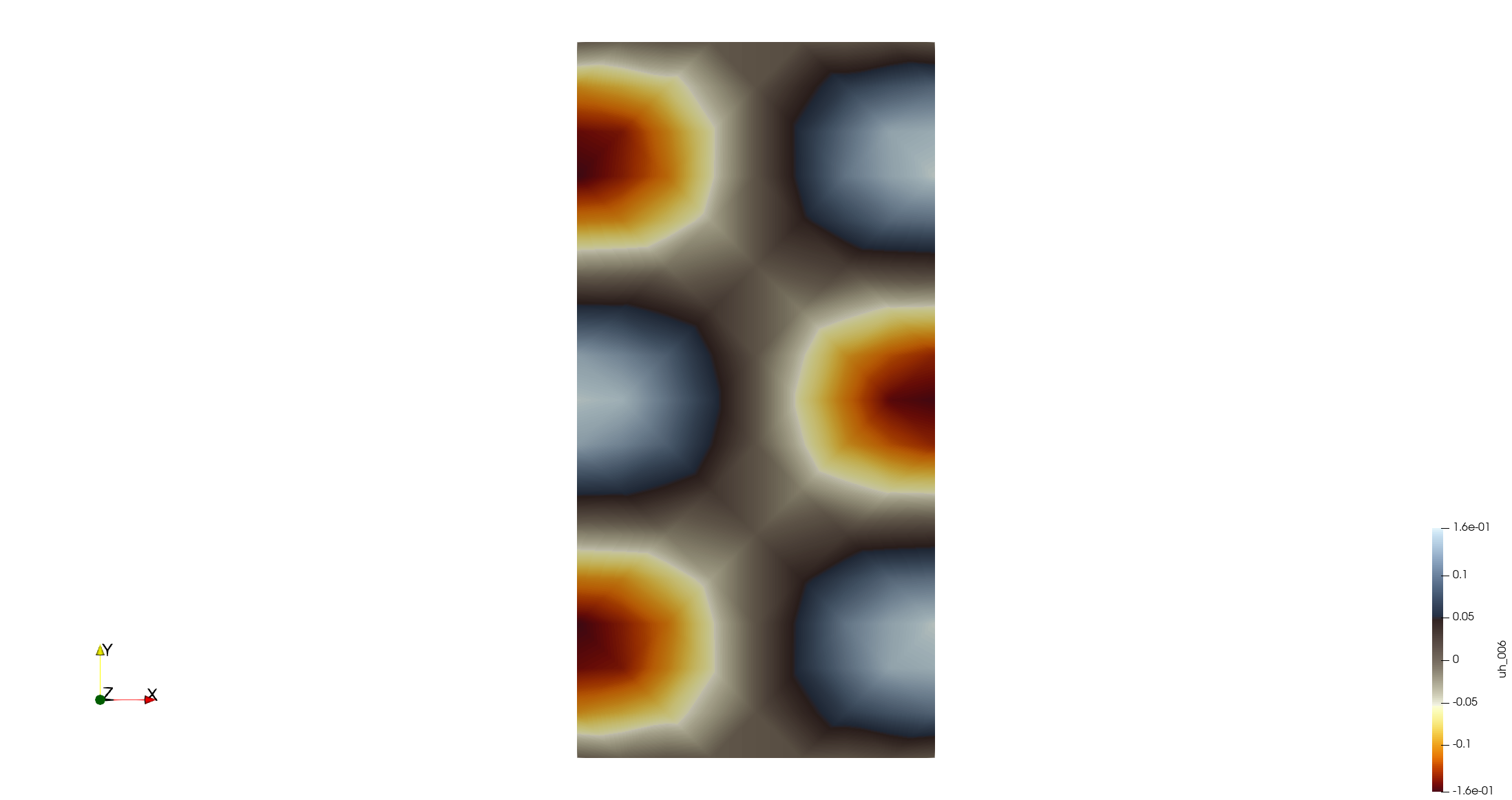}
		\end{minipage}
		\caption{Test \ref{subsec:rectangle-domain}. Surface plot of the first six lowest computed eigenmodes when taking $\alpha=0$ on the skew-symmetric ($\varepsilon=-1$) variant of the Nitsche method for $N=8$.}
		\label{fig:eigenvalues-nonsymetric}
	\end{figure}

	\subsection{The L-shaped domain}\label{subsec:lshape-domain}
	This experiment test the performance of the schemes when a domain with geometrical singularities is considered. To this end, let us define the L-shaped domain $\Omega:=(0,1)^2\backslash\left((1/2,1)\times(0,1/2)\right)$ with homogeneous boundary conditions, i.e, $u=0$ on $\Gamma$ and $\Sigma=\emptyset$. Because of the re-entrant corner at $(x,y)=(0,0)$, it is expected that at least one of the eigenvalues becomes singular. We also study the effects of the stabilization on the spectrum correctness. The mesh levels are such that $h\approx 1/N$. Since in this particular case, there is no analytical expression for the spectrum, the extrapolated eigenvalues have been calculated by least squares fitting.
	
	From the results depicted in Figure \ref{fig:alfa-dependence-lshape2D}, we observe that the symmetric and incomplete schemes may become unstable for $\alpha<2$. The symmetric variant of Nitsche is the one presenting more oscillations near $\alpha=0$, while for $\varepsilon=0$, we observe a computation of near-zero eigenvalues, similar to the rectangle domain. Looking at the relative accuracy results, we note that the Nitsche method is more accurate for the selected mesh size. We can justify this behavior due to the overprediction seen in the first test. It should be noted that the errors and extrapolations presented were calculated in the same mesh levels for the standard and the Nitsche method. Moreover, the errors were computed using the standard method extrapolated values.
	
	Inspired by the above observation, we have computed a convergence history in Table \ref{tabla:lshape-convergence-k1}. Here we note that the first eigenvalue behaves like $\mathcal{O}(h^{1.7})$ in the three Nitsche variants due to the singularity in $(x,y)=(0.5,0.5)$. Instead, the method converges with larger orders for rest of the eigenvalues. We also present Figure \ref{fig:lshape-eigenvalues-spurious}, where we show the first four eigenmodes computed when $\alpha=0.1$. We observed that boundary layers are formed across the domain boundary for the symmetric scheme. Similar results were observed in \cite{harari2018spectral}. A loss of the Dirichlet boundary imposition is observed for $\varepsilon=0$, yielding to spurious eigenvalues. Note also that the skew-symmetric method gives the physical eigenvalues.

	We emphasize that the oscillations and figures shown correspond to the computation of the real part of the eigenvalues in the non-symmetric methods. In fact, according to the theory presented in this work, there is a possibility that the spectrum is complex. For example, for $\alpha=0.1$, we observed the appearance of the eigenvalues $-72.2156 \pm 12.0544i$, $79.1095\pm 13.0344i$, and $86.7849\pm 6.8064i$ for $\varepsilon=0$. For similar $\alpha$ and $\varepsilon=-1$ we obtained $383.876262\pm8.310906i$, $522.893254\pm10.143418i$, $635.700189\pm5.190318i$, $649.005625\pm17.160873i$ and $711.650930\pm10.436213i$ with a target of 40 eigenvalues in the eigensolver, while, as expected, no complex eigenvalues were observed in the symmetric scheme.
	
	\begin{figure}[!hbt]\centering
		\begin{minipage}{0.49\linewidth}\centering
			\includegraphics[scale=0.298,trim=0cm 0cm 1cm 1cm,clip]{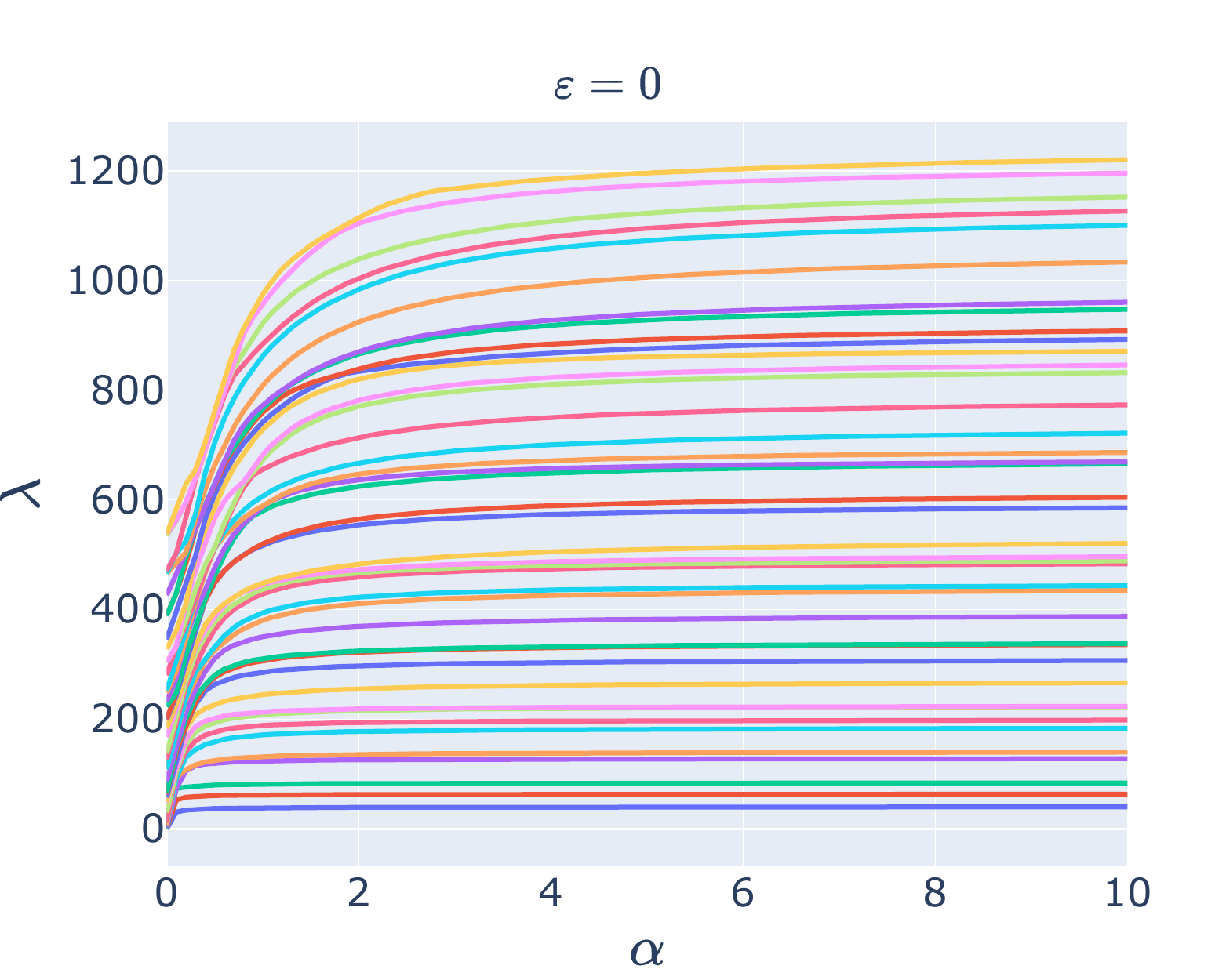}
		\end{minipage}
		\begin{minipage}{0.49\linewidth}\centering
			\includegraphics[scale=0.298,trim=0cm 0cm 1cm 1cm,clip]{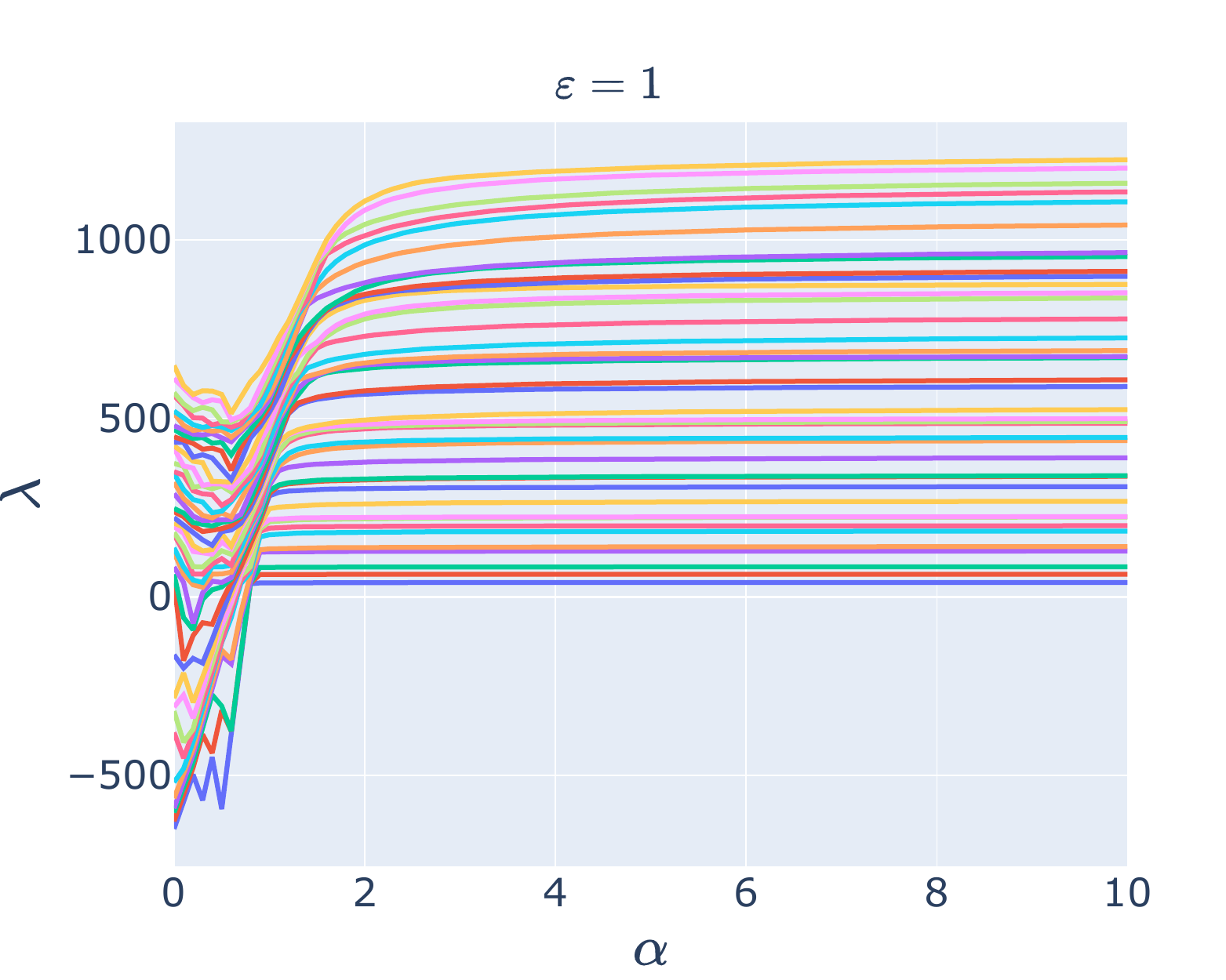}
		\end{minipage}\\
		\begin{minipage}{0.49\linewidth}\centering
			\includegraphics[scale=0.298,trim=0cm 0cm 1cm 1cm,clip]{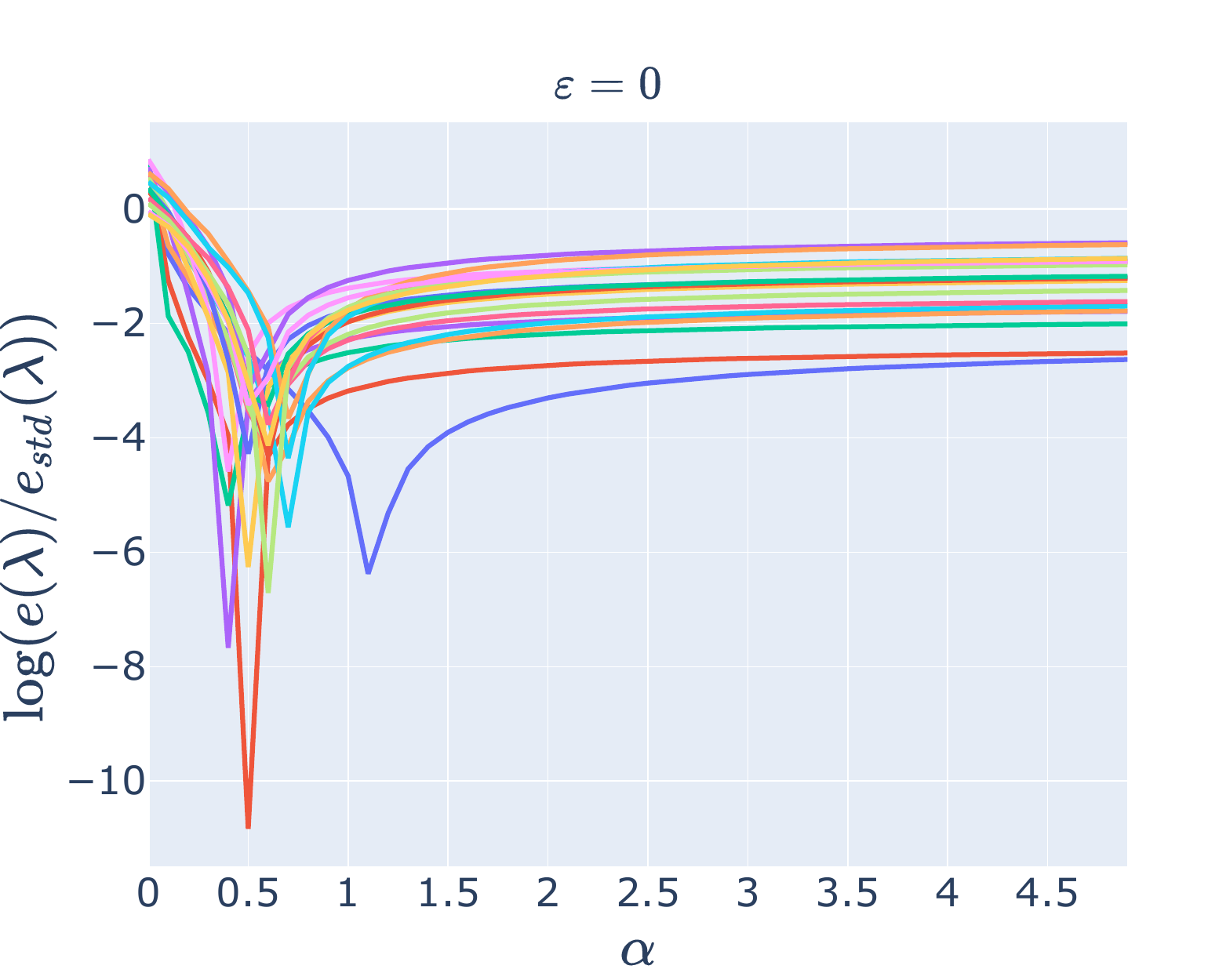}
		\end{minipage}
		\begin{minipage}{0.49\linewidth}\centering
			\includegraphics[scale=0.298,trim=0cm 0cm 1cm 1cm,clip]{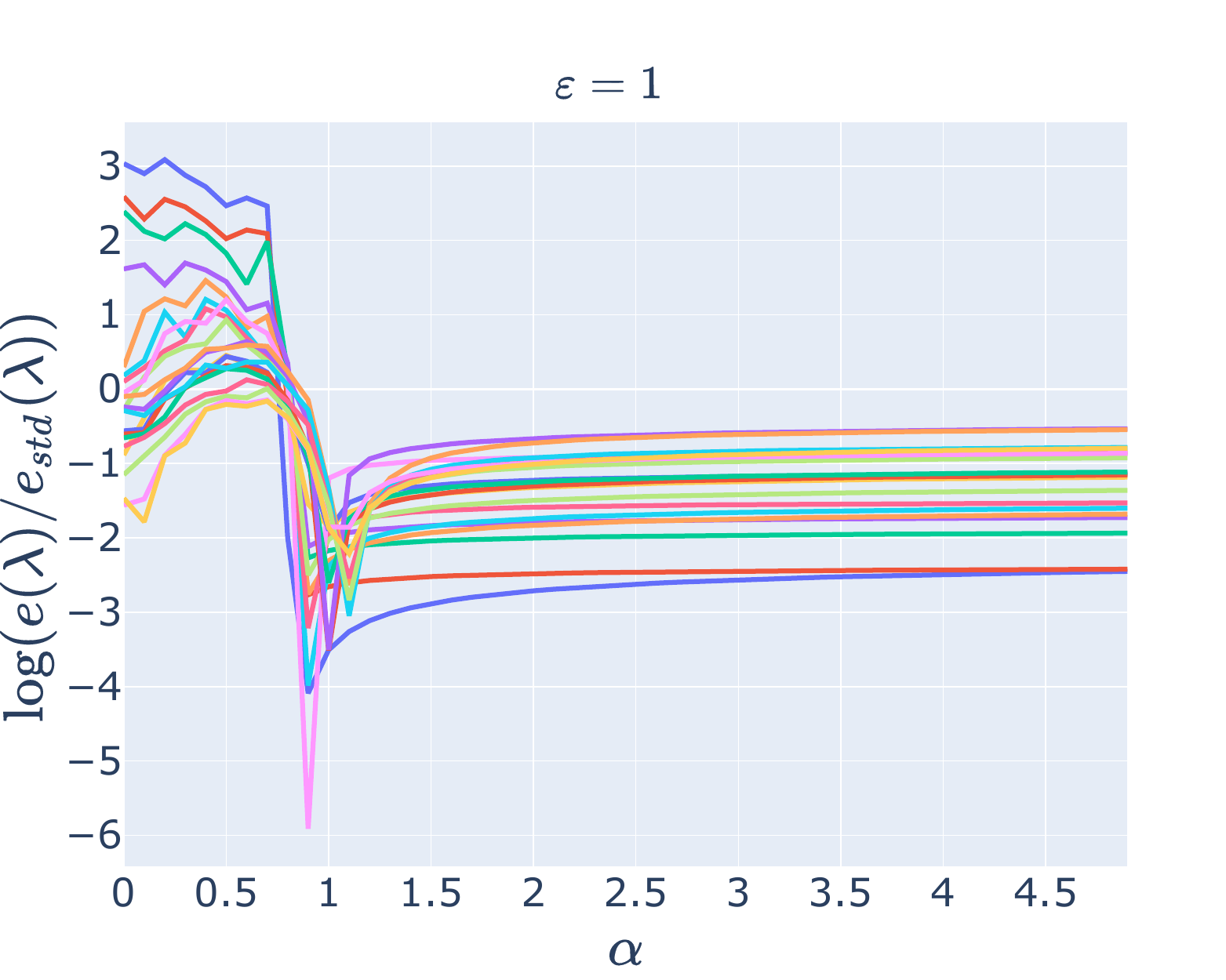}
		\end{minipage}\\
		\caption{Test \ref{subsec:lshape-domain}. Dependence of the eigenvalues in the L-shaped domain when using Nitsche's symmetric and incomplete variants with respect to the stabilization parameter $\alpha$ and $k=1$, $N=10$. Top: computation of the first 40 eigenvalues for each $\alpha$. Bottom: relative accuracy of first 20 computed eigenvalues.}
		\label{fig:alfa-dependence-lshape2D}
	\end{figure}

	\begin{table}[!hbt]
		\centering
		{\setlength{\tabcolsep}{3.8pt}\footnotesize
			\caption{Test \ref{subsec:lshape-domain}. Relative error and convergence behavior in the L-shaped domain for the first fourth lowest computed eigenvalues on the three variants of the Nitsche method with $k=1$. The stabilization parameter is set to be $\alpha=10$.}
			\label{tabla:lshape-convergence-k1}
			\begin{tabular}{|c|c|cccc|}
				\hline
				$\varepsilon$& Extrapolated & \multicolumn{4}{c|}{Relative error (rate)} \\
				\hline
				\multirow{4}{0.02\linewidth}{1}&   38.5635 &   8.93e-03  &   5.64e-03 (1.66) &   3.84e-03 (1.77) &   2.92e-03 (1.55)   \\
				&   60.7912 &   6.51e-03  &   3.65e-03 (2.08) &   2.33e-03 (2.08) &   1.61e-03 (2.08)   \\
				&   78.9601 &   9.69e-03  &   5.45e-03 (2.08) &   3.47e-03 (2.09) &   2.40e-03 (2.07)   \\
				&  118.1004 &   1.30e-02  &   7.26e-03 (2.11) &   4.61e-03 (2.10) &   3.16e-03 (2.13)   \\
				&  127.6937 &   1.73e-02  &   1.00e-02 (1.96) &   6.51e-03 (2.00) &   4.62e-03 (1.93)   \\
				\hline
				\multirow{4}{0.02\linewidth}{0}&   38.5644 &   8.20e-03  &   5.13e-03 (1.69) &   3.49e-03 (1.78) &   2.62e-03 (1.60)   \\
				&   60.7923 &   6.36e-03  &   3.57e-03 (2.08) &   2.28e-03 (2.06) &   1.57e-03 (2.10)   \\
				&   78.9601 &   9.53e-03  &   5.37e-03 (2.07) &   3.43e-03 (2.07) &   2.37e-03 (2.07)   \\
				&  118.1010 &   1.28e-02  &   7.16e-03 (2.10) &   4.56e-03 (2.09) &   3.12e-03 (2.12)   \\
				&  127.6977 &   1.66e-02  &   9.57e-03 (1.99) &   6.20e-03 (2.00) &   4.36e-03 (1.97)   \\
				\hline
				\multirow{4}{0.02\linewidth}{-1}&   38.5649 &   7.54e-03  &   4.68e-03 (1.73) &   3.17e-03 (1.80) &   2.36e-03 (1.66)   \\
				&   60.7933 &   6.24e-03  &   3.50e-03 (2.09) &   2.24e-03 (2.06) &   1.54e-03 (2.12)   \\
				&   78.9620 &   9.37e-03  &   5.28e-03 (2.07) &   3.37e-03 (2.07) &   2.33e-03 (2.08)   \\
				&  118.1012 &   1.26e-02  &   7.08e-03 (2.09) &   4.51e-03 (2.08) &   3.10e-03 (2.11)   \\
				&  127.7012 &   1.59e-02  &   9.14e-03 (2.01) &   5.91e-03 (2.01) &   4.13e-03 (2.02)   \\
				\hline
				&$N$&30&40&50 &60\\
				\hline
				\hline
		\end{tabular}}
	\end{table}

	\begin{figure}[!hbt]
		\centering
		\begin{minipage}{0.24\linewidth}\centering
			{\footnotesize $\lambda_{h,1},\varepsilon=1$}\\
			\includegraphics[scale=0.11,trim=22cm 4cm 22cm 4cm,clip]{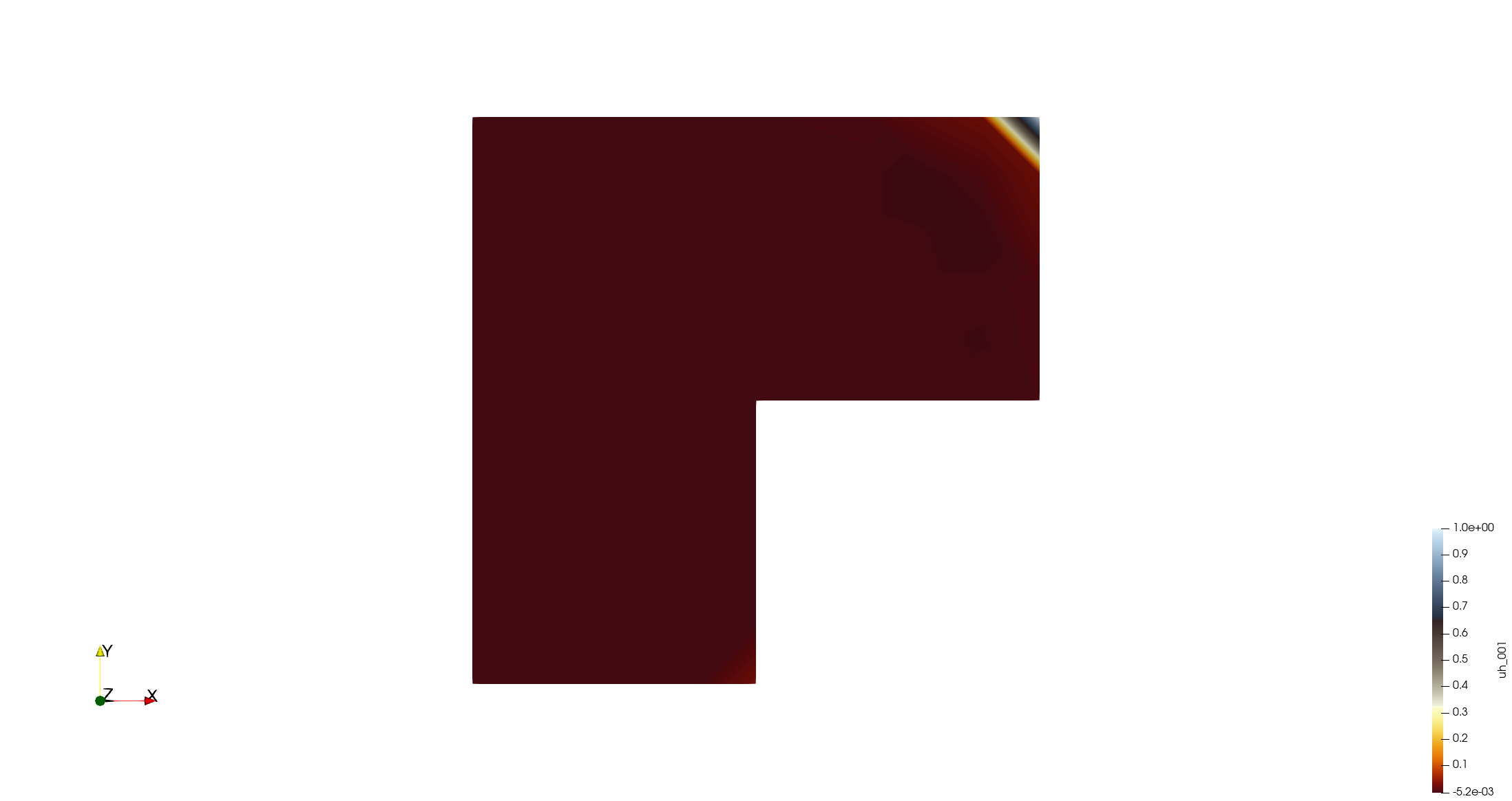}
		\end{minipage}
		\begin{minipage}{0.24\linewidth}\centering
			{\footnotesize $\lambda_{h,2},\varepsilon=1$}\\
			\includegraphics[scale=0.11,trim=22cm 4cm 22cm 4cm,clip]{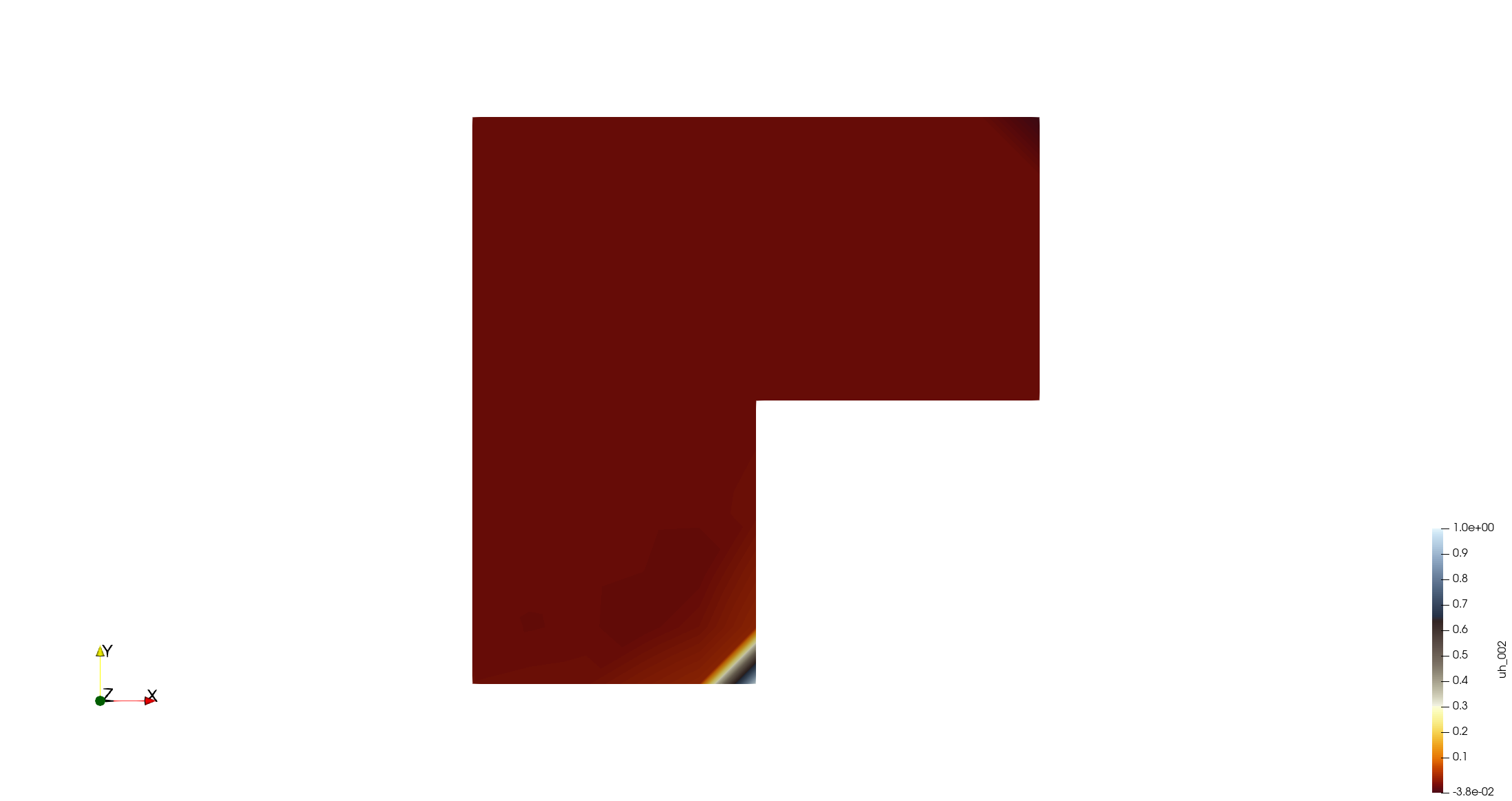}
		\end{minipage}
		\begin{minipage}{0.24\linewidth}\centering
			{\footnotesize $\lambda_{h,3},\varepsilon=1$}\\
			\includegraphics[scale=0.11,trim=22cm 4cm 22cm 4cm,clip]{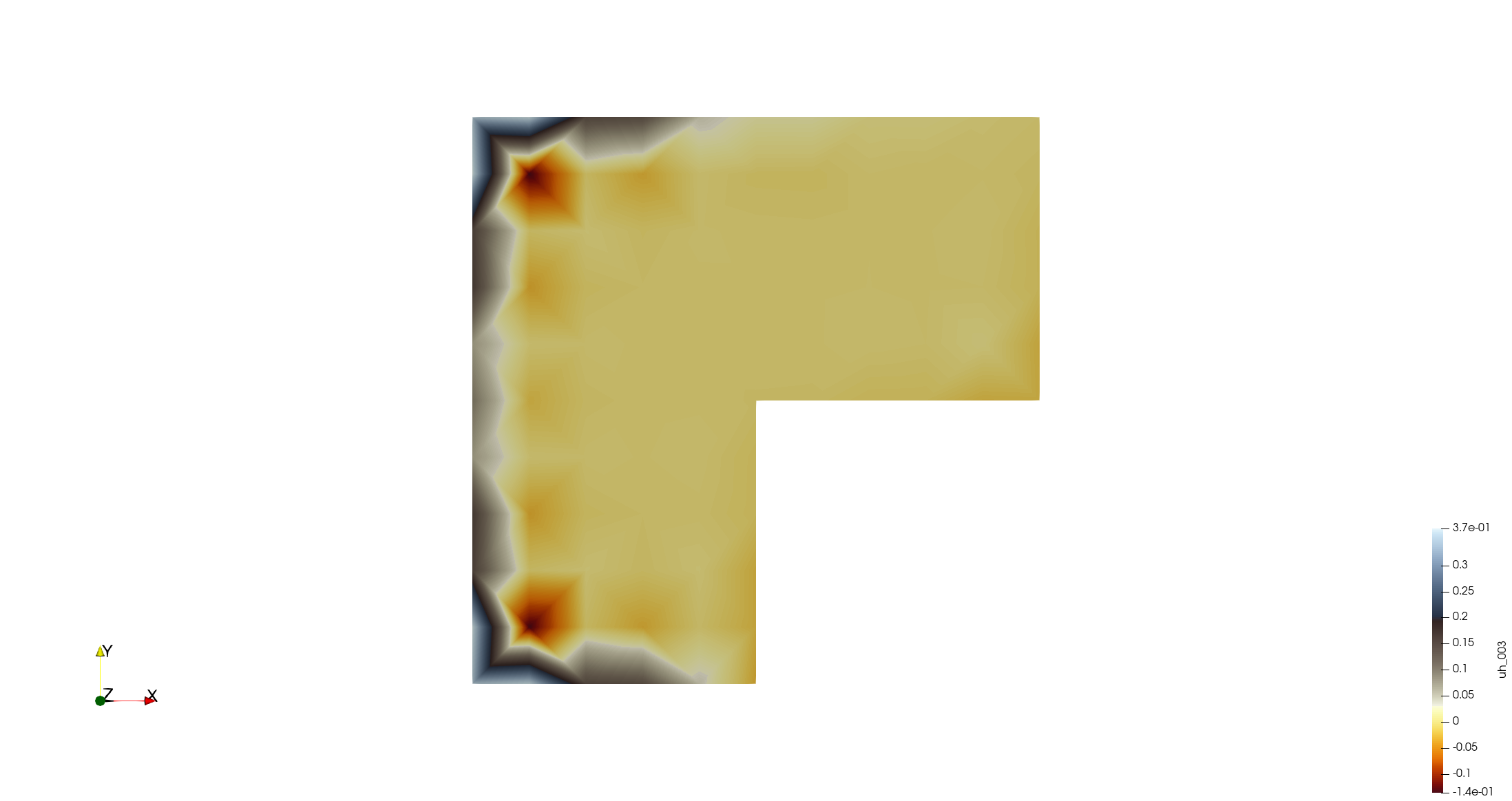}
		\end{minipage}
		\begin{minipage}{0.24\linewidth}\centering
			{\footnotesize $\lambda_{h,4},\varepsilon=1$}\\
			\includegraphics[scale=0.11,trim=22cm 4cm 22cm 4cm,clip]{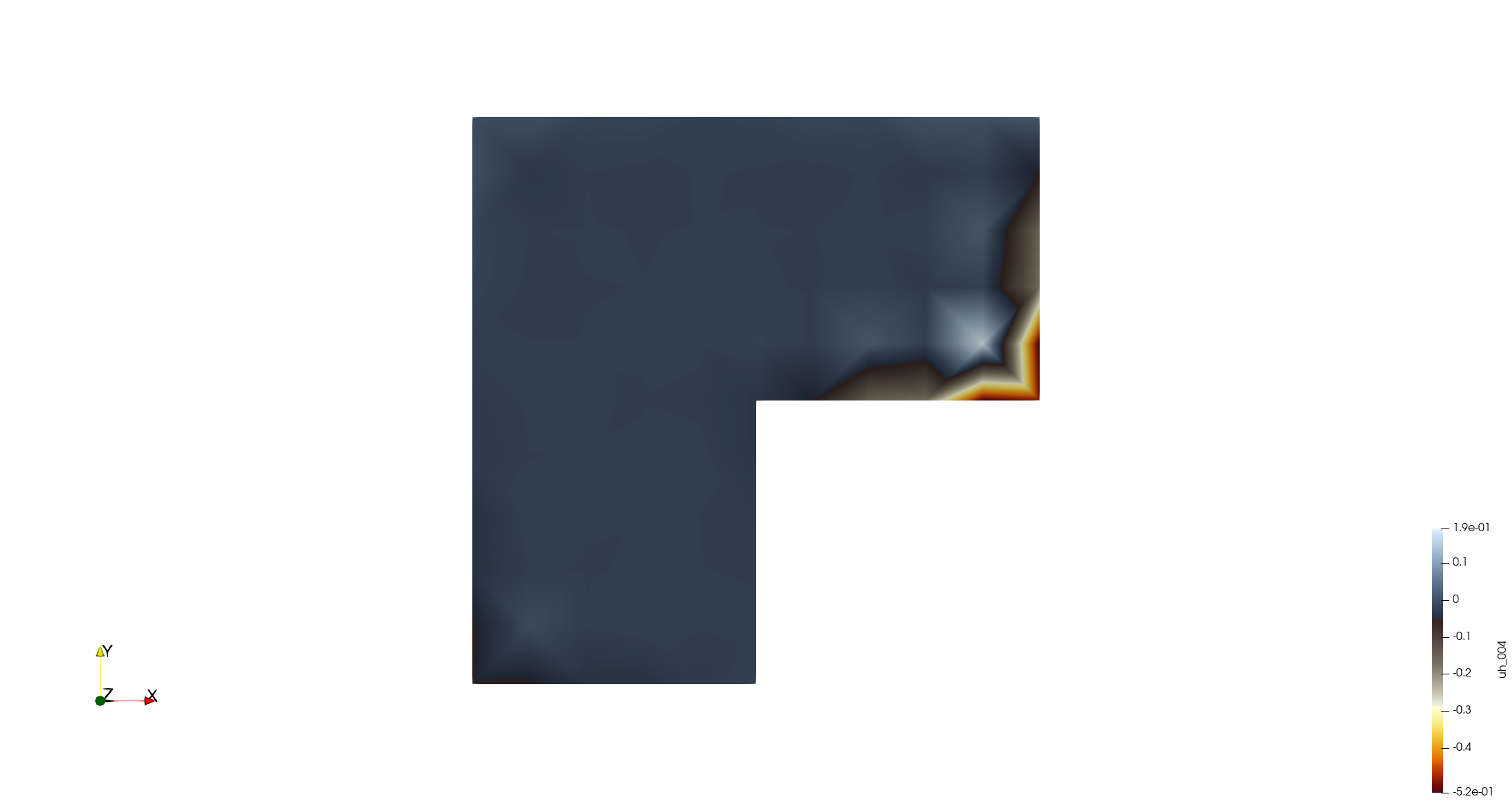}
		\end{minipage}\\
		\begin{minipage}{0.24\linewidth}\centering
			{\footnotesize $\lambda_{h,1},\varepsilon=0$}\\
			\includegraphics[scale=0.11,trim=22cm 4cm 22cm 4cm,clip]{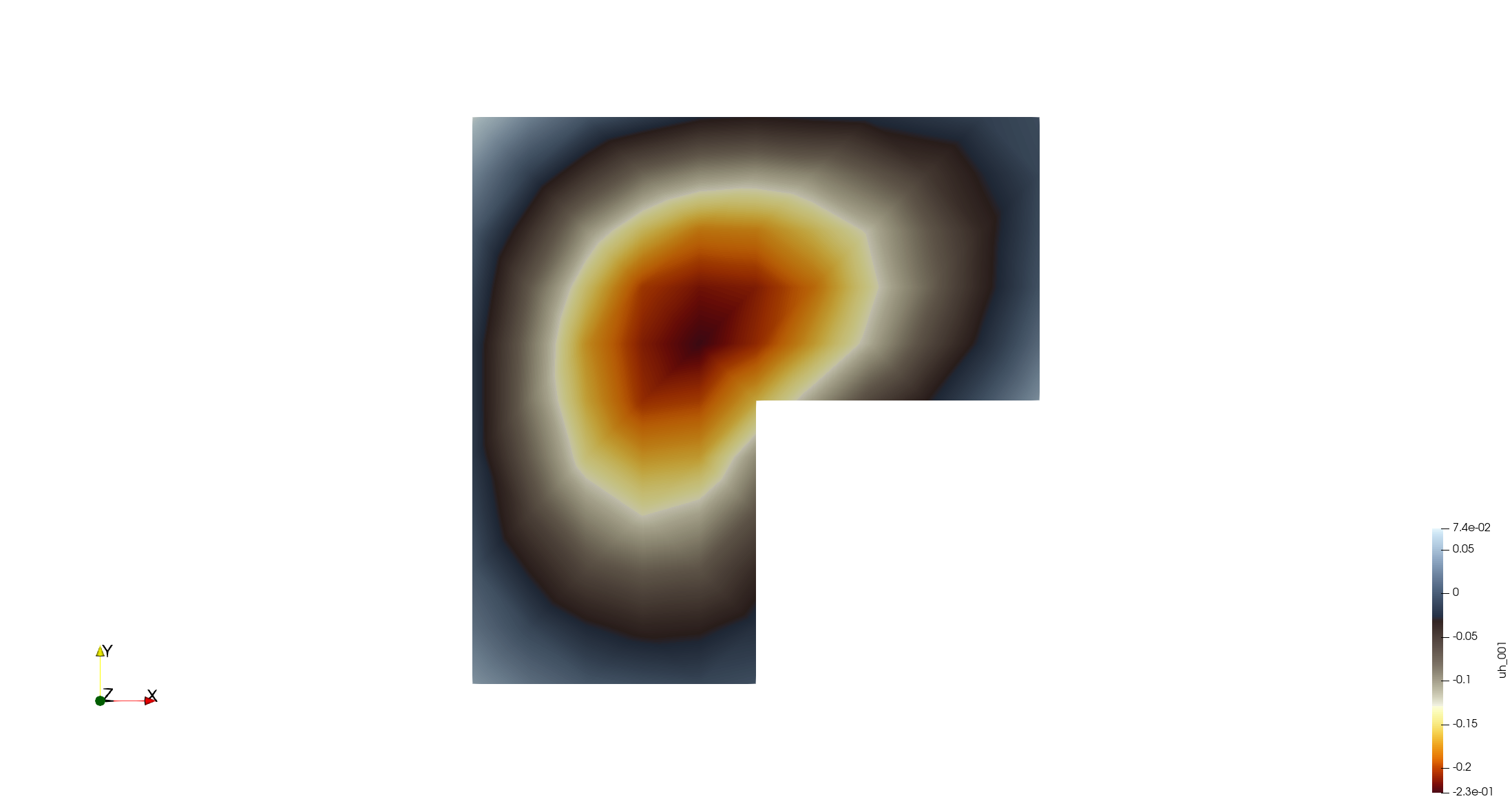}
		\end{minipage}
		\begin{minipage}{0.24\linewidth}\centering
			{\footnotesize $\lambda_{h,2},\varepsilon=0$}\\
			\includegraphics[scale=0.11,trim=22cm 4cm 22cm 4cm,clip]{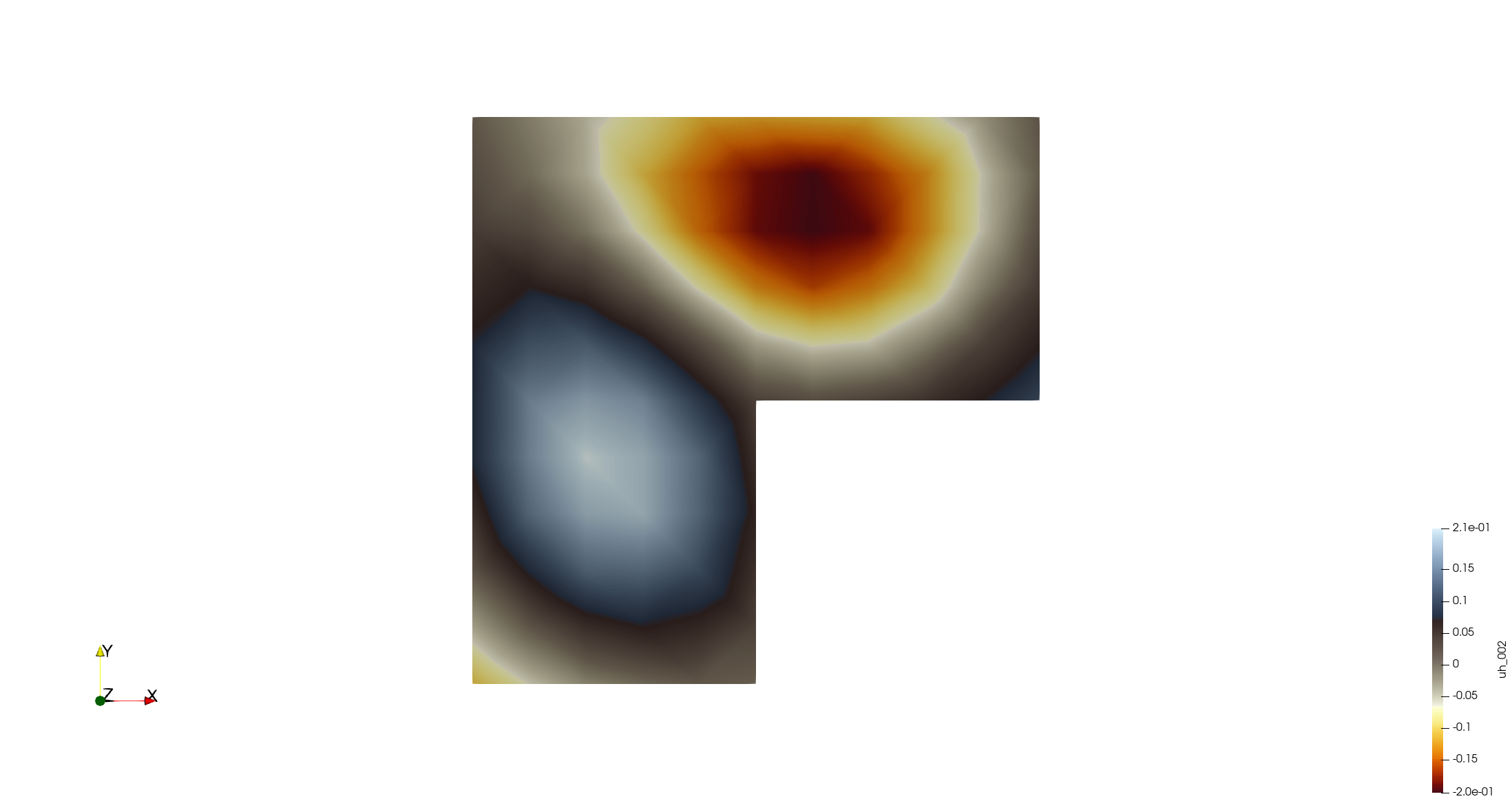}
		\end{minipage}
		\begin{minipage}{0.24\linewidth}\centering
			{\footnotesize $\lambda_{h,3},\varepsilon=0$}\\
			\includegraphics[scale=0.11,trim=22cm 4cm 22cm 4cm,clip]{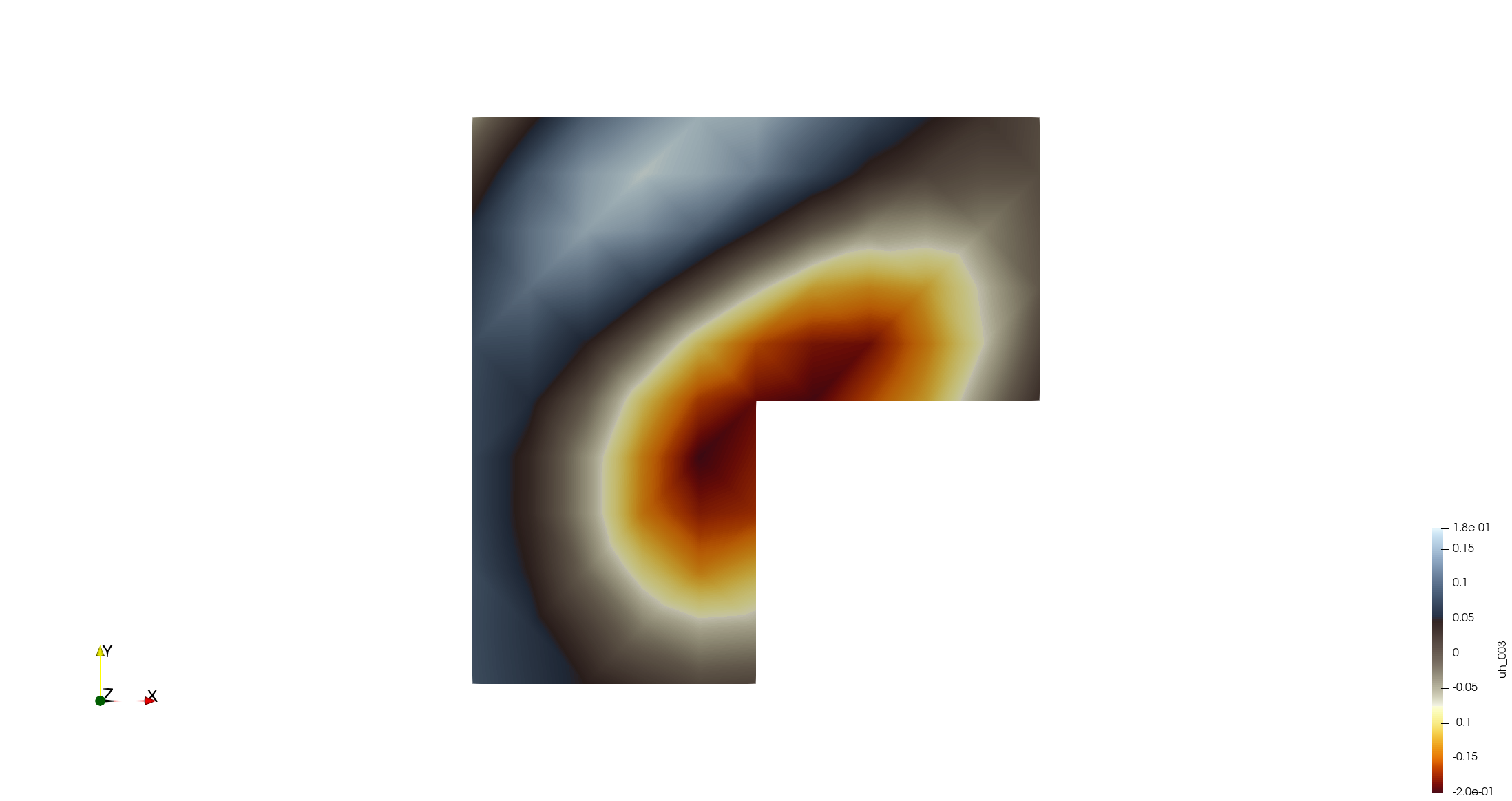}
		\end{minipage}
		\begin{minipage}{0.24\linewidth}\centering
			{\footnotesize $\lambda_{h,4},\varepsilon=0$}\\
			\includegraphics[scale=0.11,trim=22cm 4cm 22cm 4cm,clip]{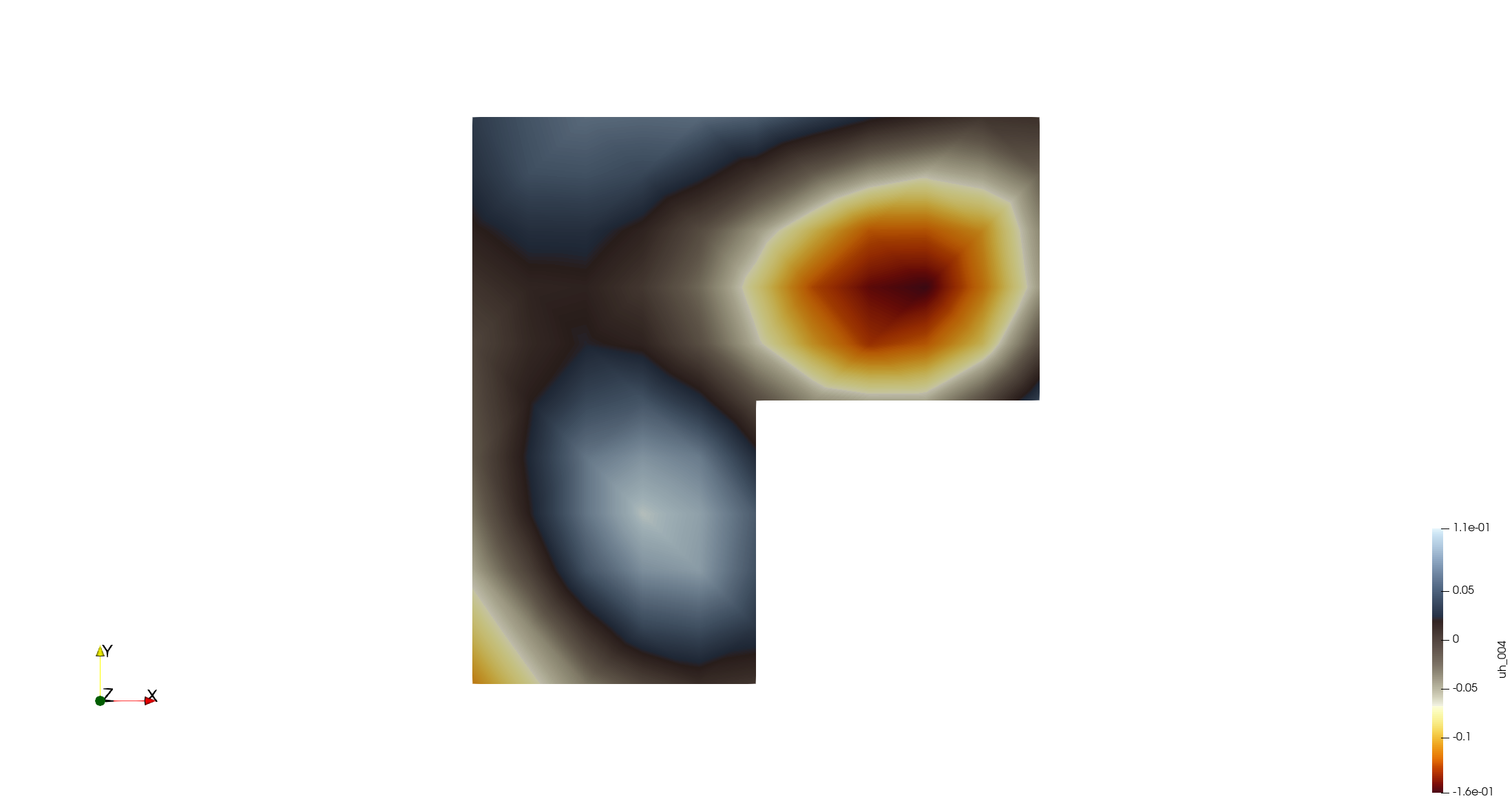}
		\end{minipage}\\
		\begin{minipage}{0.24\linewidth}\centering
			{\footnotesize $\lambda_{h,1},\varepsilon=-1$}\\
			\includegraphics[scale=0.11,trim=22cm 4cm 22cm 4cm,clip]{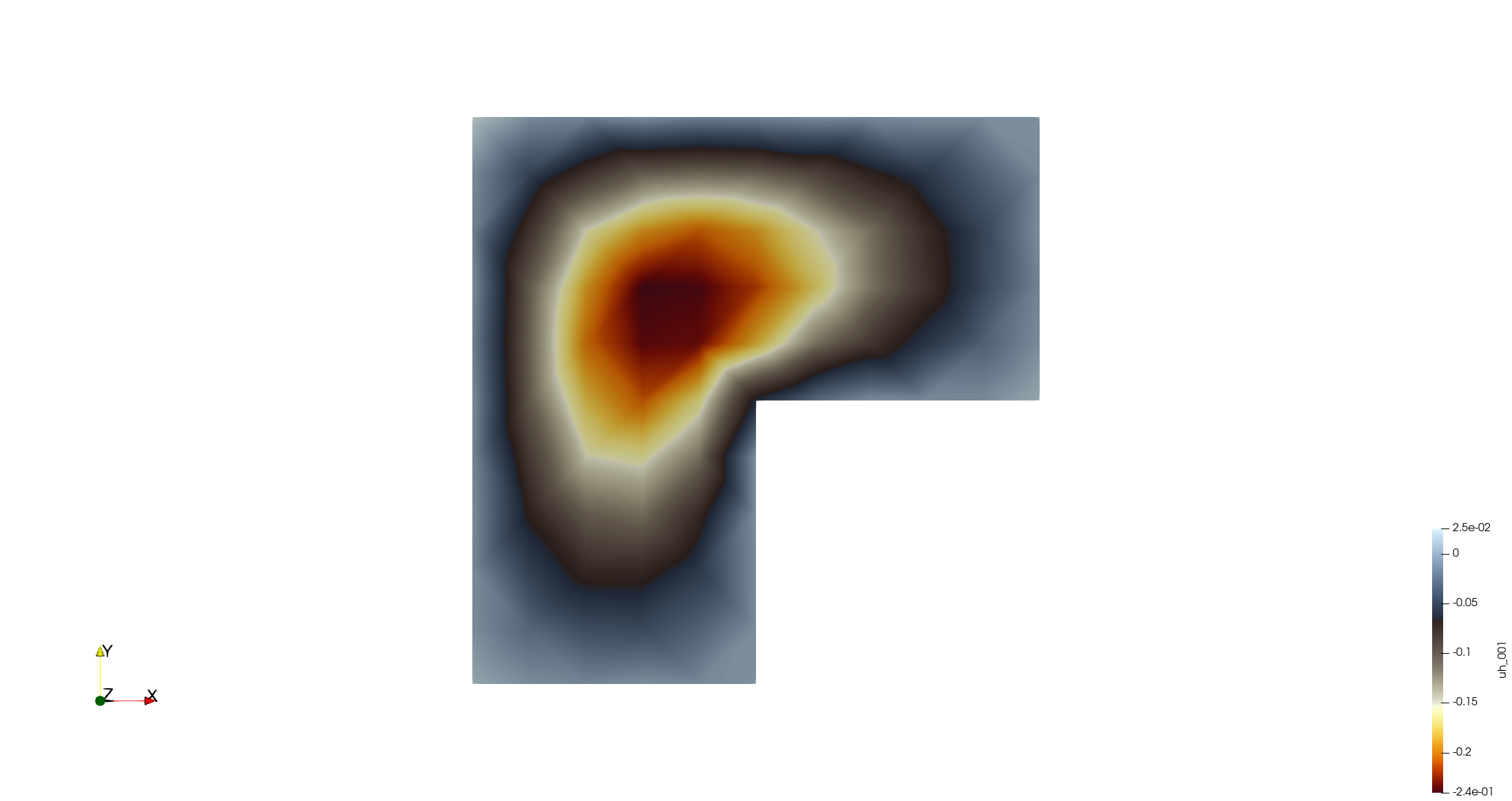}
		\end{minipage}
		\begin{minipage}{0.24\linewidth}\centering
			{\footnotesize $\lambda_{h,2},\varepsilon=-1$}\\
			\includegraphics[scale=0.11,trim=22cm 4cm 22cm 4cm,clip]{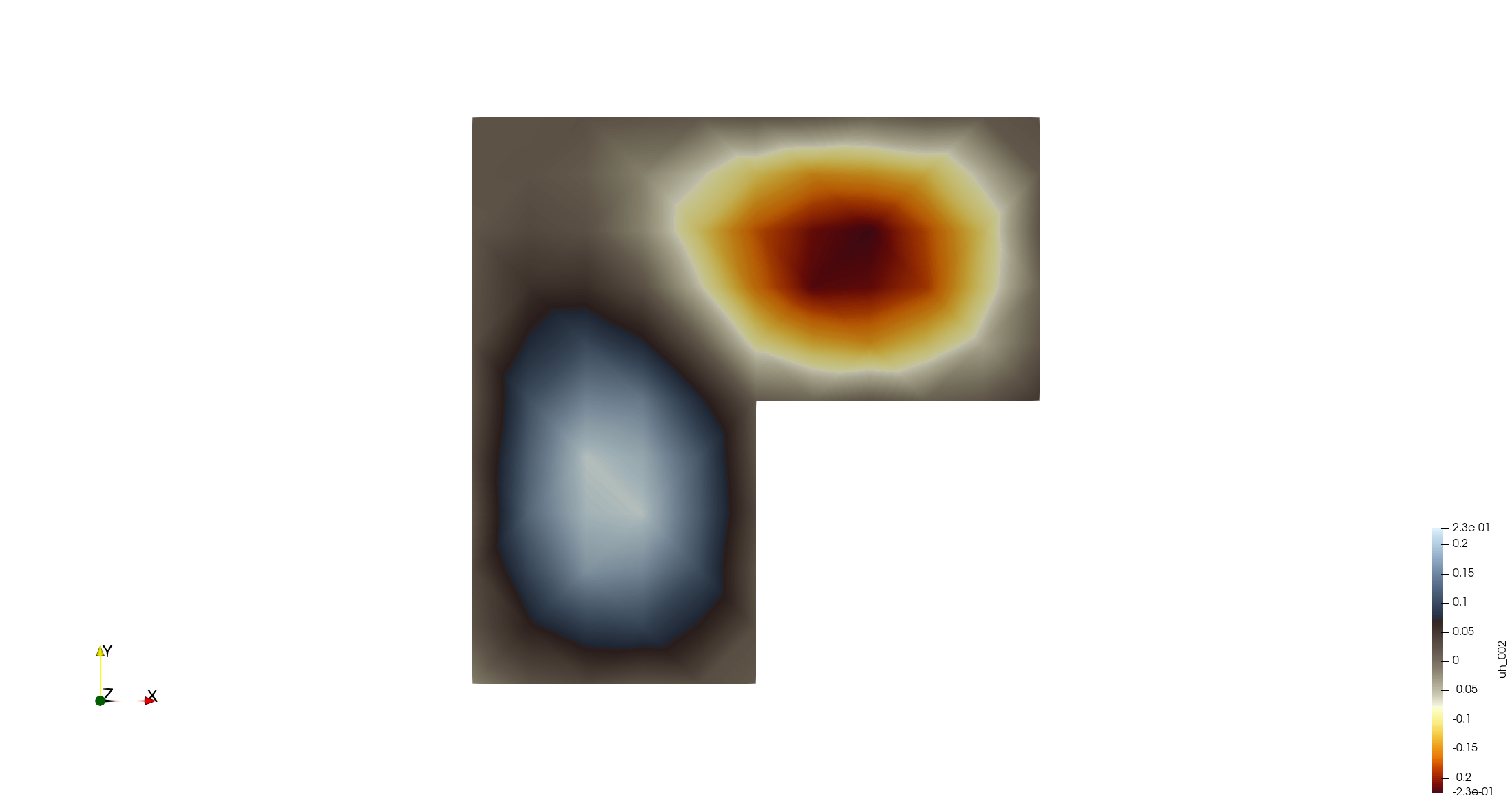}
		\end{minipage}
		\begin{minipage}{0.24\linewidth}\centering
			{\footnotesize $\lambda_{h,3},\varepsilon=-1$}\\
			\includegraphics[scale=0.11,trim=22cm 4cm 22cm 4cm,clip]{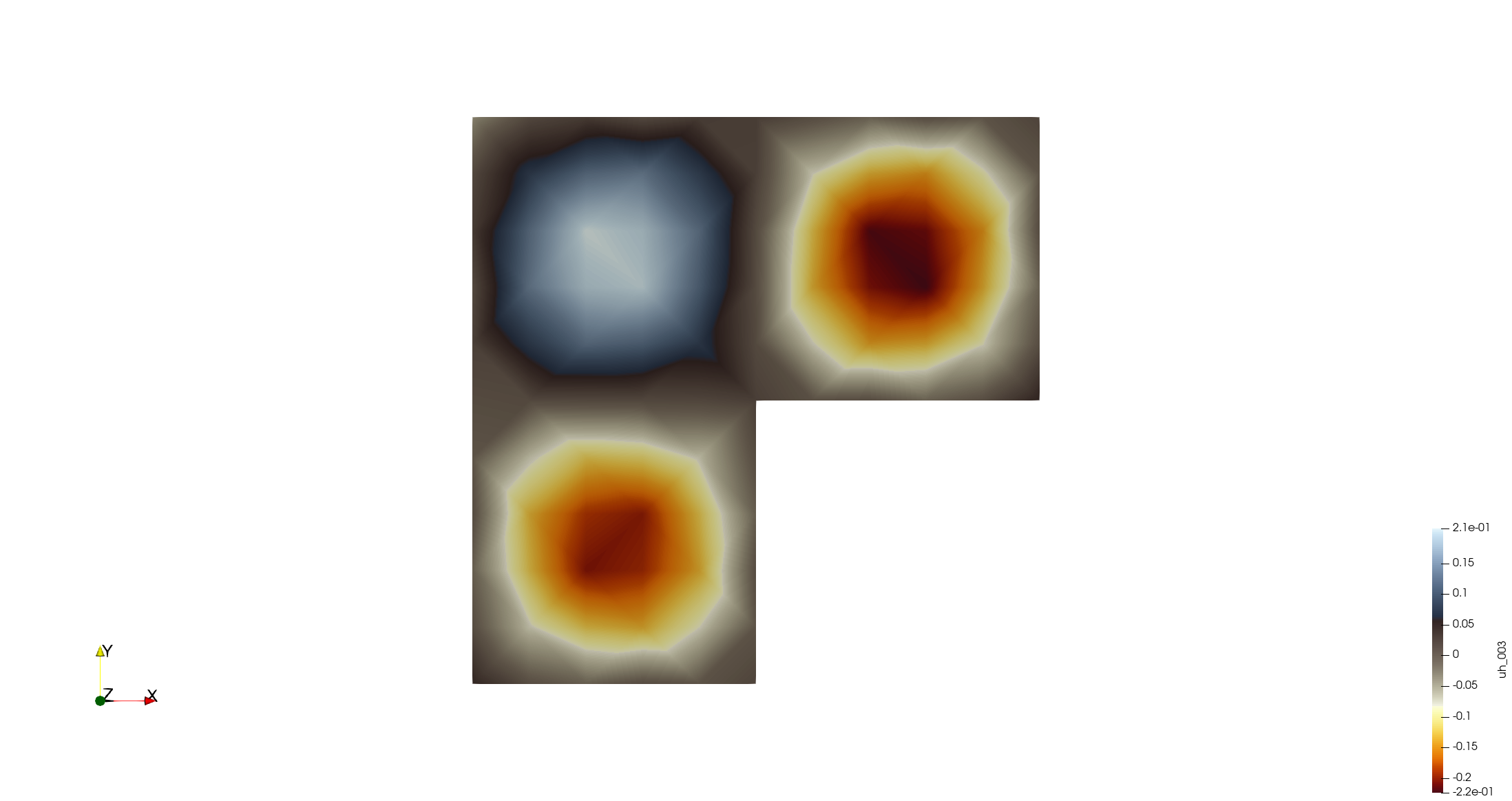}
		\end{minipage}
		\begin{minipage}{0.24\linewidth}\centering
			{\footnotesize $\lambda_{h,4},\varepsilon=-1$}\\
			\includegraphics[scale=0.11,trim=22cm 4cm 22cm 4cm,clip]{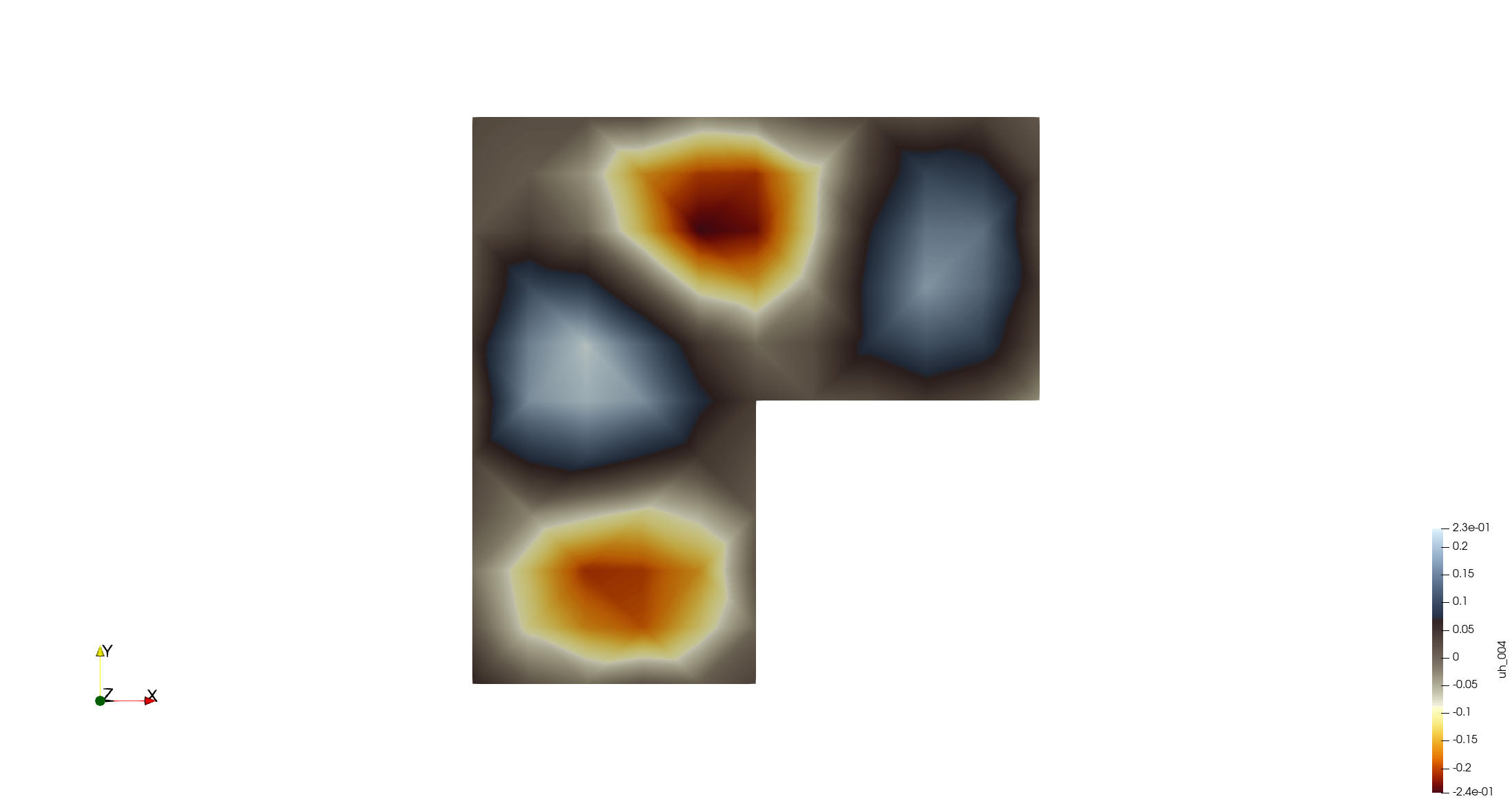}
		\end{minipage}
		\caption{Test \ref{subsec:lshape-domain}. Surface plot of the first 4 lowest computed eigenmodes in the L-shaped domain when taking $\alpha=0.1$ in the three variants of Nitsche method for $k=1$ and $N=8$.}
		\label{fig:lshape-eigenvalues-spurious}
	\end{figure}
	
	\begin{figure}[!hbt]\centering
		\begin{minipage}{0.49\linewidth}\centering
			\includegraphics[scale=0.4,trim=0cm 0cm 2cm 2cm,clip]{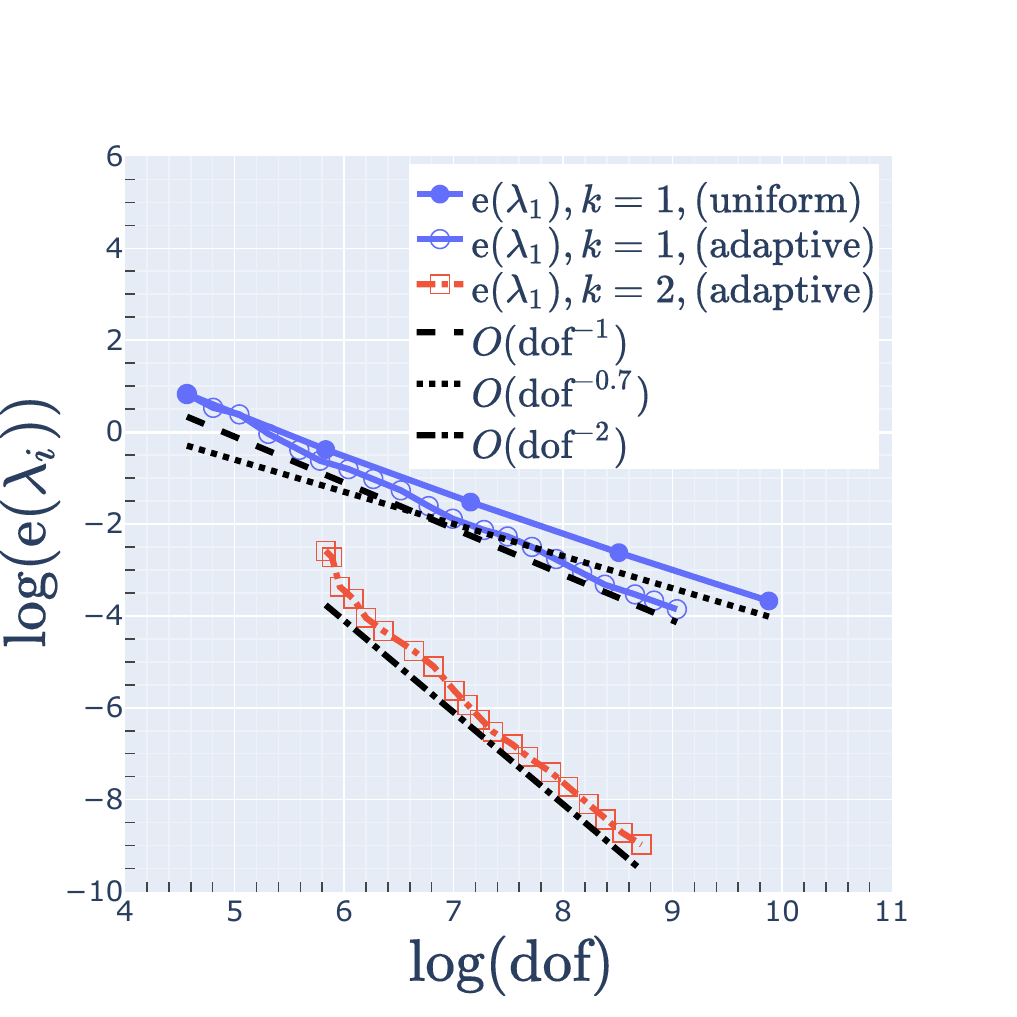}
		\end{minipage}
		\begin{minipage}{0.49\linewidth}\centering
			\includegraphics[scale=0.4,trim=0cm 0cm 2cm 2cm,clip]{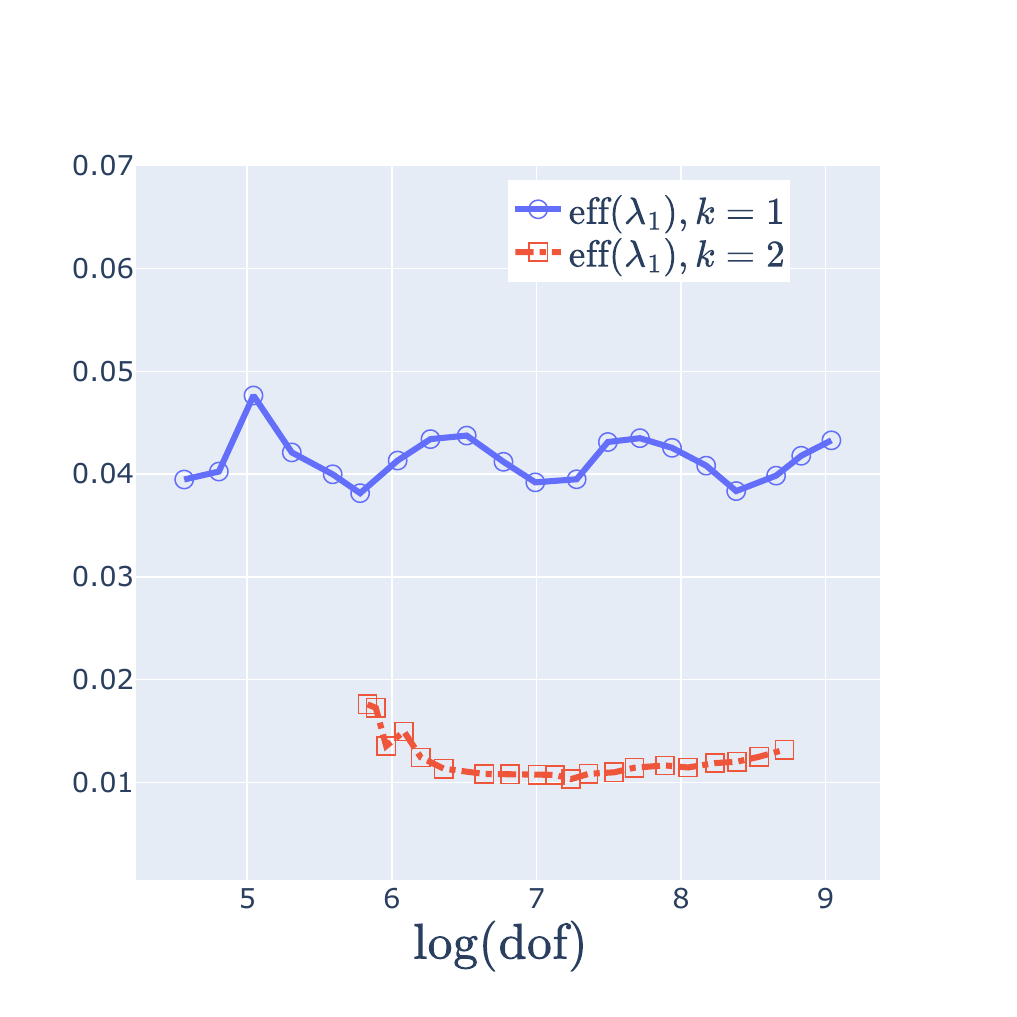}
		\end{minipage}
		\caption{Test \ref{subsec:lshape-domain}. Error history and efficiency curves for the symmetric version of the adaptive Nitsche method in the Lshape domain.}
		\label{fig:lshape-error-adaptive}
	\end{figure}
	
	\begin{figure}[!hbt]\centering
		\begin{minipage}{0.24\linewidth}\centering
			{\footnotesize $\texttt{dof}=96$, iter = 1}
			\includegraphics[scale=0.08,trim=27cm 7cm 28cm 7cm,clip]{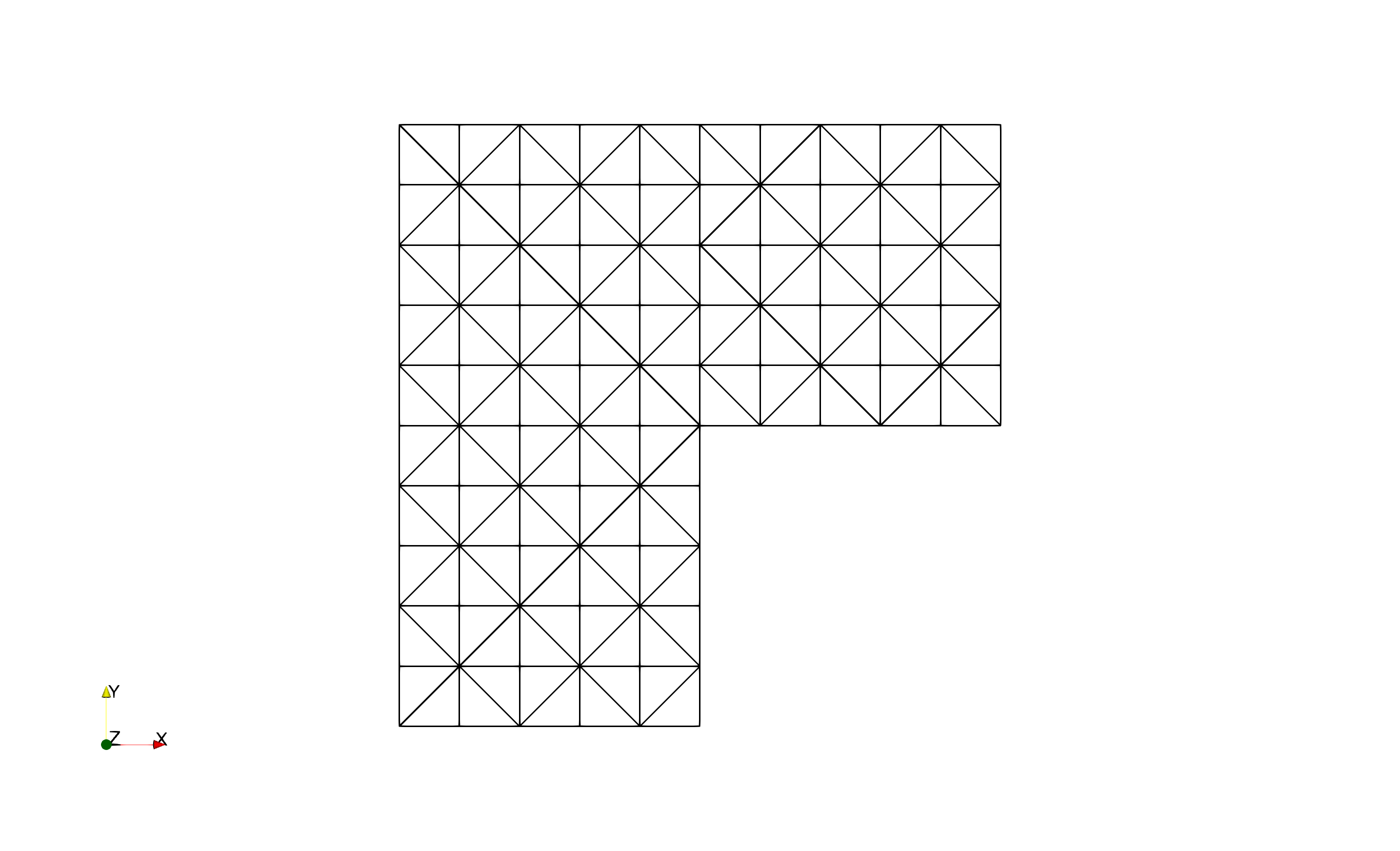}
		\end{minipage}
		\begin{minipage}{0.24\linewidth}\centering
			{\footnotesize $\texttt{dof}=873$, iter = 10}
			\includegraphics[scale=0.08,trim=27cm 7cm 28cm 7cm,clip]{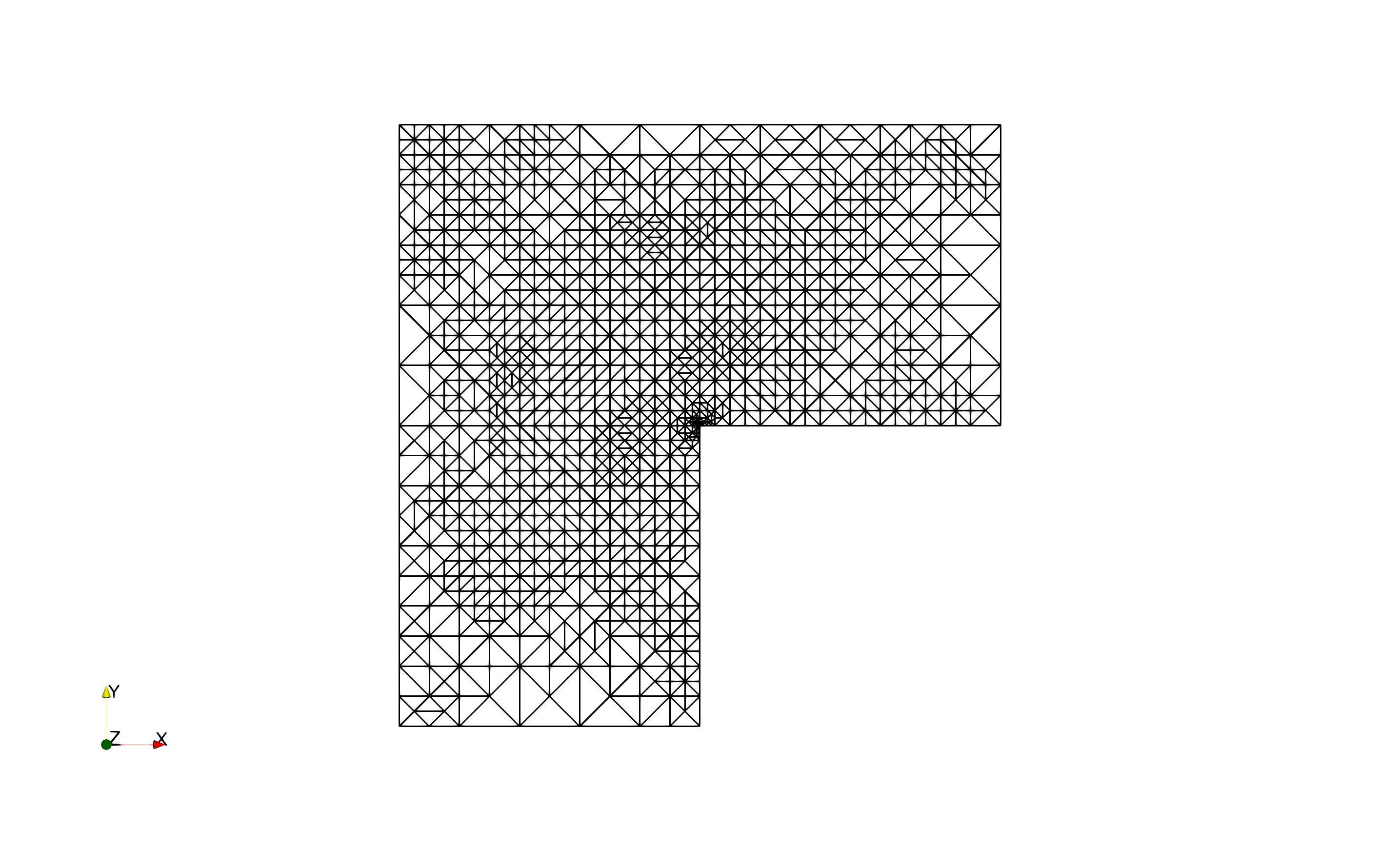}
		\end{minipage}
		\begin{minipage}{0.24\linewidth}\centering
			{\footnotesize $\texttt{dof}=2800$, iter = 15}
			\includegraphics[scale=0.08,trim=27cm 7cm 28cm 7cm,clip]{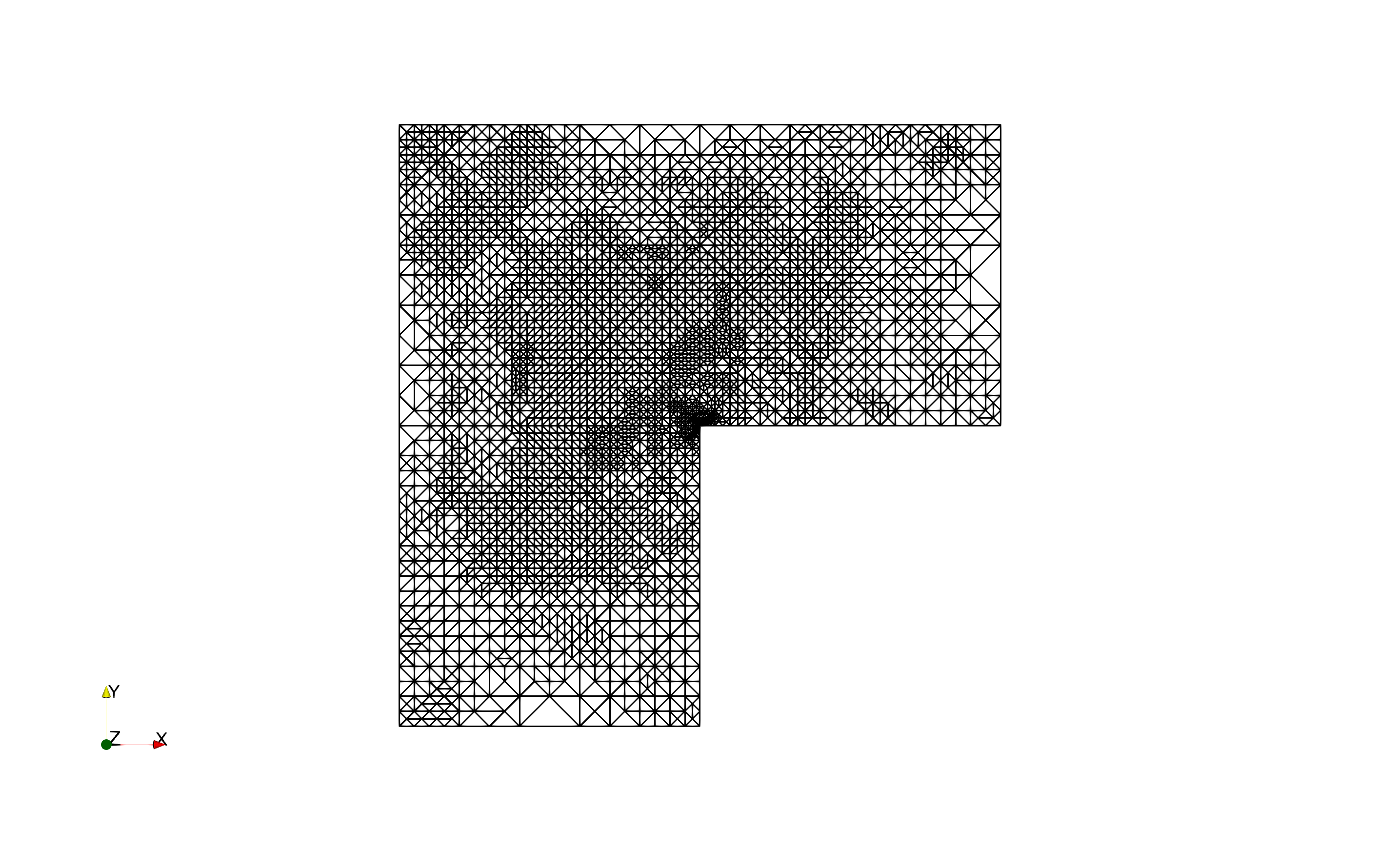}
		\end{minipage}
		\begin{minipage}{0.24\linewidth}\centering
			{\footnotesize $\texttt{dof}=8430$, iter = 20}
			\includegraphics[scale=0.08,trim=27cm 7cm 28cm 7cm,clip]{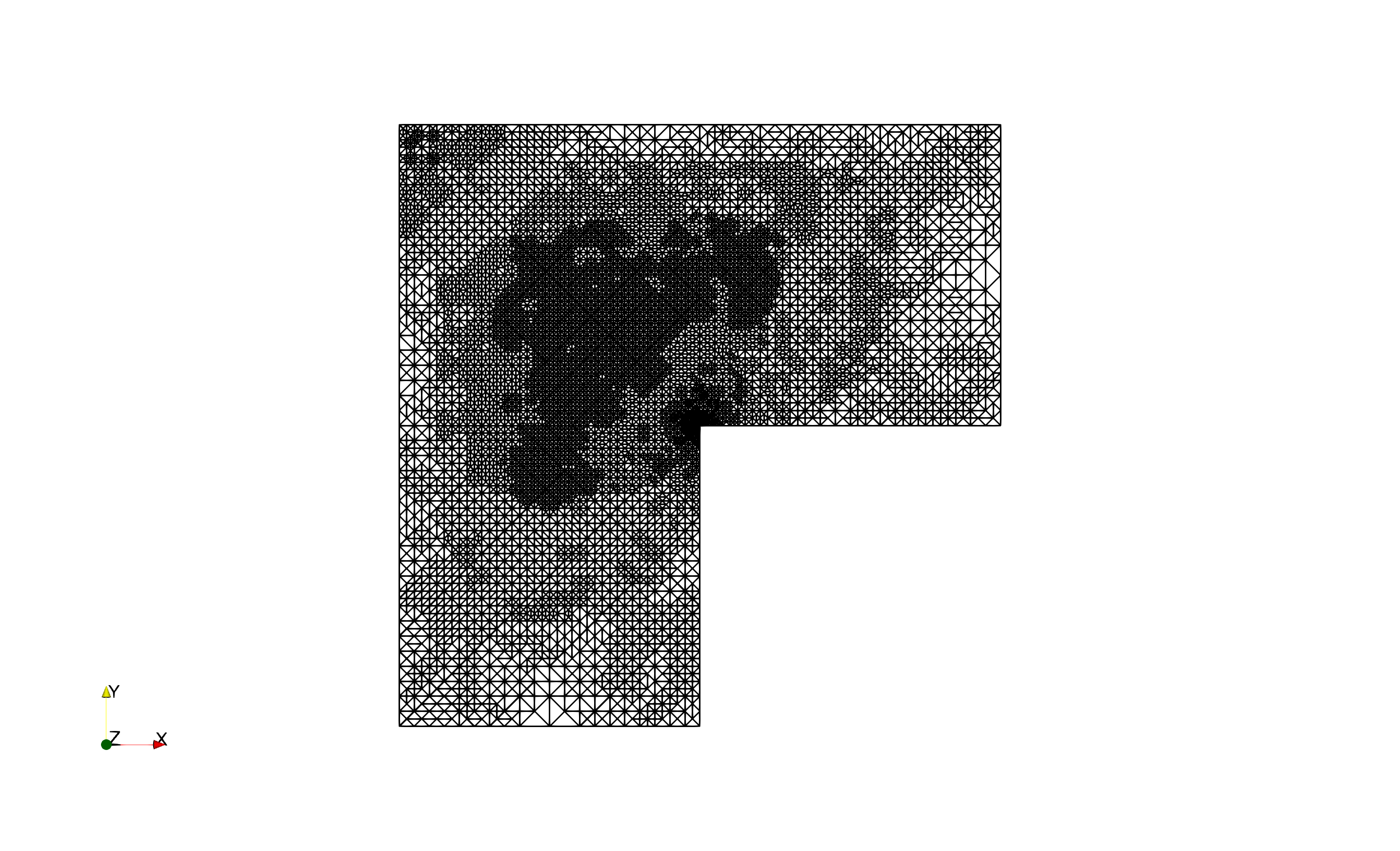}
		\end{minipage}
		\caption{Test \ref{subsec:lshape-domain}. Initial L-shaped meshed domain, followed by intermediate adaptive meshes at different iteration steps.
		}
		\label{fig:lshape-adaptive-meshes}
	\end{figure}
	
	\subsection{Extension to a 3D setting}\label{subsec:3d-extension}

Although the theoretical analysis has been presented in the two-dimensional setting, the same operator-based construction can be extended to three-dimensional polyhedral domains whenever the required elliptic regularity is available. In the present work, we only use this extension in the fully Dirichlet case, for which the regularity needed below is available on convex polyhedra. Let $\Omega\subset\mathbb{R}^3$ be a bounded Lipschitz polyhedron and consider the extension to three dimensions of Problem \eqref{P1}--\eqref{P2} with $|\Gamma_D|>0$. Then we have the continuous solution operator
\[
   T^{3D}:L^2(\Omega)\to L^2(\Omega),\qquad Tf:=u,
\]
where $u\in V$ is the solution of the previous source problem. The compactness of
$T^{3D}$ follows from the compact embedding $H^1(\Omega)\hookrightarrow L^2(\Omega)$.

The key point for extending the convergence theory is the regularity of the solution operator. If $\Omega$ is a convex polyhedron and $\Gamma_D=\partial\Omega$, then we have the standard elliptic regularity estimates and consequently, in the notation of Lemma~\ref{LEM:REG}, the three-dimensional fully
Dirichlet convex case corresponds to (see \cite{grisvard1986problemes,MR961439})
\[
   u\in H^{1+r}(\Omega),\qquad r>0.
\]
Thus, Lemma~\ref{LEM:REG} remains valid in this setting. 
For the Nitsche discretization, the discrete solution operator is given by
\[
   T_h^{3D}:L^2(\Omega)\to V_h,\qquad
   A_h(T_h^{3D} f,v_h)=(f,v_h)_{0,\Omega}\qquad \forall v_h\in V_h,
\]
where $A_h(\cdot,\cdot)$ denotes the corresponding Nitsche bilinear form. For a sufficiently large stabilization parameter, $A_h$ is coercive in the mesh-dependent energy norm
\[
   \|v_h\|_{h}^{2}
   :=
   \|\nabla v_h\|_{0,\Omega}^{2}
   +\sum_{F\subset\Gamma_D}\alpha h_F^{-1}\|v_h\|_{0,F}^{2}.
\]
where $h_F$ is the face diameter. Hence $T_h^{3D}$ is well defined. The consistency, stability and approximation arguments used in the previous sections then yield
\[
   \|T^{3D}-T_h^{3D}\|_{\mathcal{L}(L^2(\Omega),L^2(\Omega))}\to 0
   \qquad \text{as } h\to 0.
\]
Therefore, the Babu\v{s}ka--Osborn spectral approximation framework applies also in this case. If $u$ is an eigenfunction associated with a simple eigenvalue and $u\in H^{1+r}(\Omega)$, one can follow the arguments from Section \ref{SEC:approximation} to obtain
\[
   \|u-u_h\|_{h}\lesssim h^{\min\{k,r\}},
   \qquad
   |\lambda-\lambda_h|\lesssim h^{2\min\{k,r\}}
\]
for the symmetric Nitsche method with full Dirichlet boundary conditions. 

In what follows, we present two three-dimensional experiments with different purposes. The first one is the unit cube with full Dirichlet boundary conditions. This test lies within the regularity regime described above and is used to verify the convergence
rate with the Nitsche method. The second one is a Fichera-type non-convex domain. For this geometry, the re-entrant corner and re-entrant edges induce singular eigenfunctions. We consider both a fully Dirichlet configuration and a mixed Dirichlet--Neumann configuration. The fully Dirichlet case is used to assess the adaptive strategy in a non-convex three-dimensional domain, whereas the mixed case is used as a numerical experiment beyond the scope of the regularity theory used in this work. We show that the proposed estimator and adaptive refinement strategy remain effective in a configuration where geometric singularities interact with changes of boundary condition.

	\subsubsection{The unit cube domain}\label{subsec:cube-domain}
	This test aims to assess the robustness of the method for three-dimensional problems. The model problem, numerical scheme and error indicators follow the same strategy as in the two-dimensional case. However, the theoretical estimates proved in the previous sections rely on the regularity statement in Lemma~\ref{LEM:REG}, which was stated for the two-dimensional setting. Therefore, the present experiment should be understood as going beyond the scope of the available theory. The test considers the study of the spectrum on the domain $\Omega:=(0,1)^3$. We take $u=0$ and $\partial\Omega=\Gamma$. In this case, the exact eigenvalues are given by
	$$
	\lambda_{mnp}=\pi^2(m^2+n^2+p^2),\qquad m,n,p\geq 1.
	$$
	with $m,n,p\geq 1$.  Here, $N$ scales as the number of cells such that the number of tetrahedrons is $6(N +1)^3$ and $h\approx 1/N$. 
	
	The convergence results for a sufficiently large value of $\alpha$ are given in Tables \ref{tabla:cube-convergence-k1}--\ref{tabla:cube-convergence-k2}. From these tables, we note that all the schemes exhibit the same convergence orders predicted by the two-dimensional theory, namely the expected order $\mathcal{O}(h^{2k})$ for the symmetric variant and the corresponding suboptimal behavior for the nonsymmetric variants. The choice of this stabilization parameter is justified by the results depicted in Figure \ref{fig:alfa-dependence-cube}. Here, we observe that small values of $\alpha$ may introduce instabilities, together with eigenvalue crossings and veering in the system. For $\alpha > 4$, we observe a good separation of the spectrum with considerable overprediction for some eigenvalues. Although not presented in this paper, we have analyzed the error history of these overpredicted eigenvalues and obtained the optimal convergence rates according to the scheme ($\mathcal{O}(h^{2k})$ for $\varepsilon=1$ and $\mathcal{O}(h^{2k-0.5})$ for the rest). Similar to the rectangle example, we observe that the relative accuracy converges asymptotically to $\log(1)=0$ from below, which indicates a slightly smaller error when using Nitsche's method.

	Similar to the previous example, we also observed the appearance of complex eigenvalues in this case for small values of $\alpha$. For example, for $\varepsilon=-1$ and $\alpha=0.1$ we obtained $135.811760\pm17.700768 i$, $196.729753+13.754241i$, $204.619429\pm22.832995i$, and $221.582774\pm25.628950i$,  while $29.7811 \pm 7.6192 i$, $37.0800 \pm 5.0612 i$ and $75.237506\pm5.257980i$ were observed when $\varepsilon=0$. Higher values of $\alpha$ provided a clean spectrum for all the variants.
	
	We finish this test by presenting contour plots of some of the lowest computed eigenmodes in Figure \ref{fig:cube-eigenvalues-spurious} with a small value of $\alpha$. We observe boundary surface layers for $\varepsilon=1$ accumulated along the edges and corner points of the domain. For $\varepsilon=0$, we note that the imposition of fixed boundary conditions is lost, while for $\varepsilon=-1$ we have an excellent agreement with the exact eigenmodes. We recall that, from the results in Figure \ref{fig:alfa-dependence-cube}, similar outcomes are achieved for the symmetric and incomplete variants if we choose, for example, $\alpha > 5$.
	
	\begin{table}[!hbt]
		\centering
		{\setlength{\tabcolsep}{3.8pt}\footnotesize
			\caption{Test \ref{subsec:cube-domain}. Relative error and convergence behavior for the first six lowest computed eigenvalues in the unit cube domain for the three variants of the Nitsche method with $k=1$. The stabilization parameter is set to be $\alpha=10$.}
			\label{tabla:cube-convergence-k1}
			\begin{tabular}{|c|c|cccc|}
				\hline
				$\varepsilon$& Exacts & \multicolumn{4}{c|}{Relative error (rate)} \\
				\hline
				\multirow{6}{0.02\linewidth}{1}&   29.6088 &   1.65e-01  &   4.11e-02 (2.30) &   1.83e-02 (2.17) &   1.03e-02 (2.12)   \\
				&   59.2176 &   2.50e-01  &   6.21e-02 (2.30) &   2.76e-02 (2.16) &   1.55e-02 (2.11)   \\
				&   59.2176 &   2.50e-01  &   6.21e-02 (2.30) &   2.76e-02 (2.16) &   1.55e-02 (2.11)   \\
				&   59.2176 &   4.16e-01  &   9.95e-02 (2.36) &   4.38e-02 (2.19) &   2.45e-02 (2.13)   \\
				&   88.8264 &   4.10e-01  &   1.05e-01 (2.25) &   4.68e-02 (2.15) &   2.64e-02 (2.11)   \\
				
				\hline
				\multirow{6}{0.02\linewidth}{0}&   29.6088 &   1.64e-01  &   4.09e-02 (2.29) &   1.82e-02 (2.16) &   1.02e-02 (2.11)   \\
				&   59.2176 &   2.47e-01  &   6.17e-02 (2.29) &   2.75e-02 (2.16) &   1.55e-02 (2.11)   \\
				&   59.2176 &   2.47e-01  &   6.17e-02 (2.29) &   2.75e-02 (2.16) &   1.55e-02 (2.11)   \\
				&   59.2176 &   4.14e-01  &   9.92e-02 (2.36) &   4.37e-02 (2.19) &   2.45e-02 (2.13)   \\
				&   88.8264 &   4.06e-01  &   1.04e-01 (2.25) &   4.66e-02 (2.14) &   2.63e-02 (2.10)   \\
				\hline
				\multirow{6}{0.02\linewidth}{-1}&   29.6088 &   1.63e-01  &   4.08e-02 (2.29) &   1.82e-02 (2.16) &   1.02e-02 (2.11)   \\
				&   59.2176 &   2.45e-01  &   6.14e-02 (2.28) &   2.74e-02 (2.16) &   1.54e-02 (2.11)   \\
				&   59.2176 &   2.45e-01  &   6.14e-02 (2.28) &   2.74e-02 (2.16) &   1.54e-02 (2.11)   \\
				&   59.2176 &   4.13e-01  &   9.89e-02 (2.36) &   4.36e-02 (2.19) &   2.45e-02 (2.13)   \\
				&   88.8264 &   4.03e-01  &   1.04e-01 (2.24) &   4.65e-02 (2.14) &   2.63e-02 (2.10)   \\
				&   88.8264 &   4.03e-01  &   1.04e-01 (2.24) &   4.65e-02 (2.14) &   2.63e-02 (2.10)   \\
				\hline
				&$N$&4&8&16 &32\\
				\hline
				\hline
		\end{tabular}}
	\end{table}
	
	\begin{table}[!hbt]
		\centering
		{\setlength{\tabcolsep}{3.8pt}\footnotesize
			\caption{Test \ref{subsec:cube-domain}. Relative error and convergence behavior for the first six lowest computed eigenvalues in the unit cube domain for the three variants of the Nitsche method with $k=2$. The stabilization parameter is set to be $\alpha=10$.}
			\label{tabla:cube-convergence-k2}
			\begin{tabular}{|c|c|cccc|}
				\hline
				$\varepsilon$& Exacts & \multicolumn{4}{c|}{Relative error (rate)} \\
				\hline
				\multirow{6}{0.02\linewidth}{1}&   29.6088 &   7.37e-03  &   5.36e-04 (4.12) &   3.53e-05 (4.10) &   2.24e-06 (4.07)   \\
				&   59.2176 &   1.77e-02  &   1.38e-03 (4.01) &   9.28e-05 (4.07) &   5.93e-06 (4.06)   \\
				&   59.2176 &   1.77e-02  &   1.38e-03 (4.01) &   9.28e-05 (4.07) &   5.93e-06 (4.06)   \\
				&   59.2176 &   3.52e-02  &   2.73e-03 (4.02) &   1.85e-04 (4.06) &   1.19e-05 (4.05)   \\
				&   88.8264 &   4.04e-02  &   3.42e-03 (3.88) &   2.38e-04 (4.02) &   1.54e-05 (4.04)   \\
				\hline
				\multirow{6}{0.02\linewidth}{0}&   29.6088 &   7.62e-03  &   5.78e-04 (4.06) &   4.10e-05 (3.99) &   2.97e-06 (3.87)   \\
				&   59.2176 &   1.81e-02  &   1.46e-03 (3.96) &   1.05e-04 (3.97) &   7.52e-06 (3.89)   \\
				&   59.2176 &   1.81e-02  &   1.46e-03 (3.96) &   1.05e-04 (3.97) &   7.52e-06 (3.89)   \\
				&   59.2176 &   3.58e-02  &   2.83e-03 (3.99) &   1.99e-04 (4.00) &   1.36e-05 (3.96)   \\
				&   88.8264 &   4.08e-02  &   3.52e-03 (3.85) &   2.54e-04 (3.96) &   1.75e-05 (3.94)   \\
				\hline
				\multirow{6}{0.02\linewidth}{-1}&   29.6088 &   7.87e-03  &   6.18e-04 (4.00) &   4.65e-05 (3.90) &   3.68e-06 (3.74)   \\
				&   59.2176 &   1.85e-02  &   1.54e-03 (3.91) &   1.17e-04 (3.89) &   9.05e-06 (3.77)   \\
				&   59.2176 &   1.85e-02  &   1.54e-03 (3.91) &   1.17e-04 (3.89) &   9.05e-06 (3.77)   \\
				&   59.2176 &   3.64e-02  &   2.93e-03 (3.96) &   2.13e-04 (3.96) &   1.53e-05 (3.88)   \\
				&   88.8264 &   4.11e-02  &   3.62e-03 (3.82) &   2.70e-04 (3.92) &   1.96e-05 (3.87)   \\
				\hline
				&$N$&4&8&16 &32\\
				\hline
				\hline
		\end{tabular}}
	\end{table}
	
	\begin{figure}[!hbt]\centering
		\begin{minipage}{0.49\linewidth}\centering
			\includegraphics[scale=0.298,trim=0cm 0cm 1cm 1cm,clip]{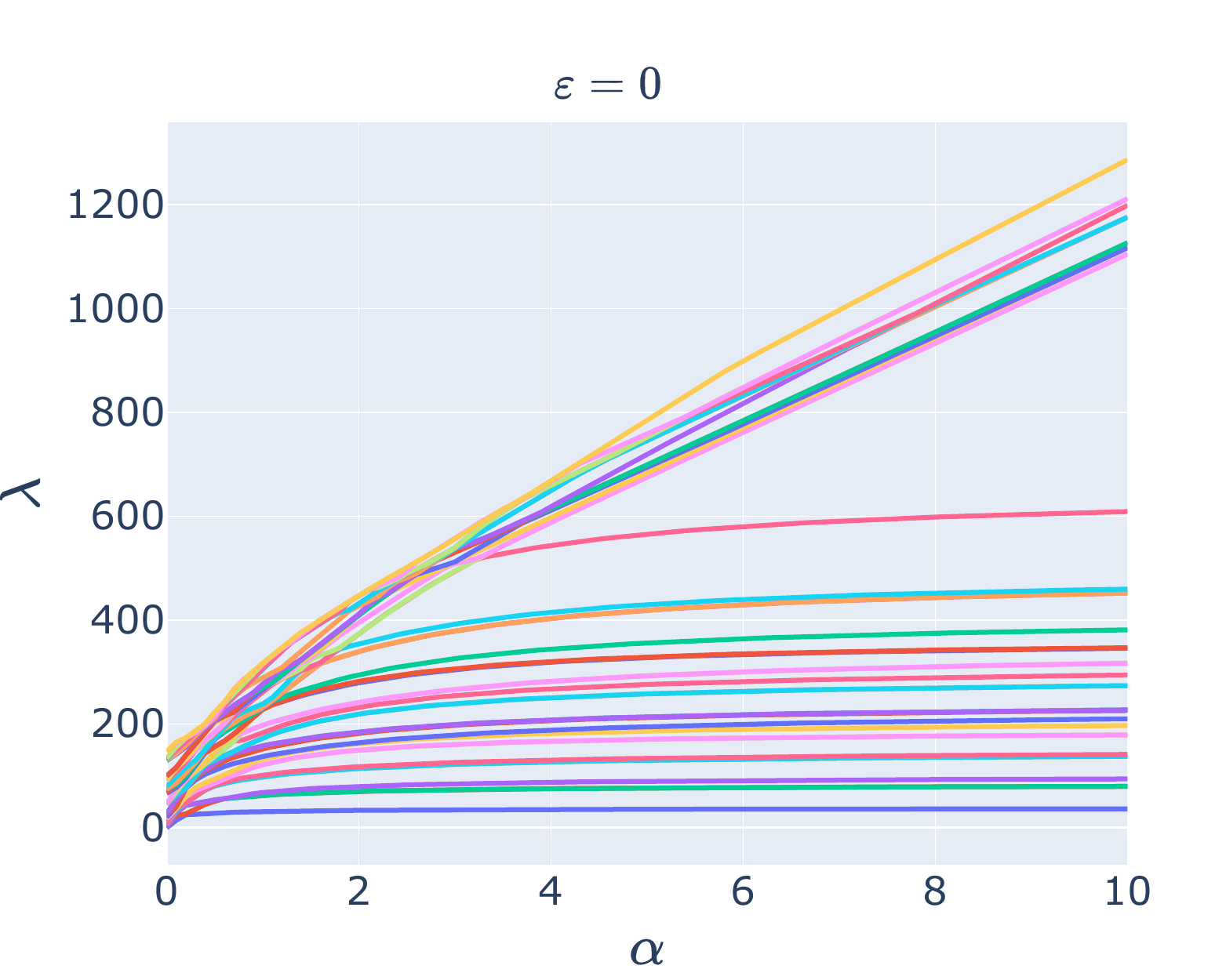}
		\end{minipage}
		\begin{minipage}{0.49\linewidth}\centering
			\includegraphics[scale=0.298,trim=0cm 0cm 1cm 1cm,clip]{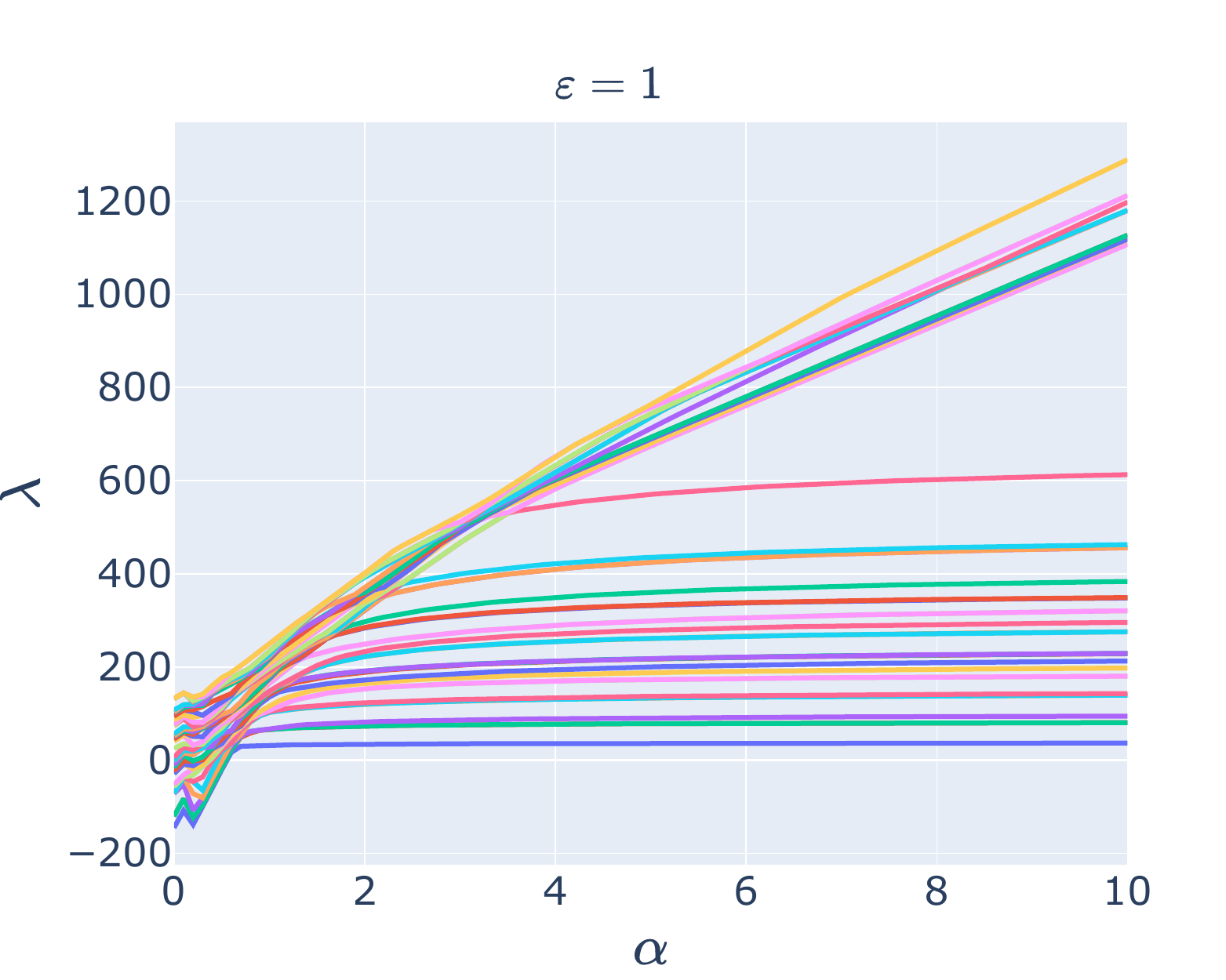}
		\end{minipage}\\
		\begin{minipage}{0.49\linewidth}\centering
			\includegraphics[scale=0.298,trim=0cm 0cm 1cm 1cm,clip]{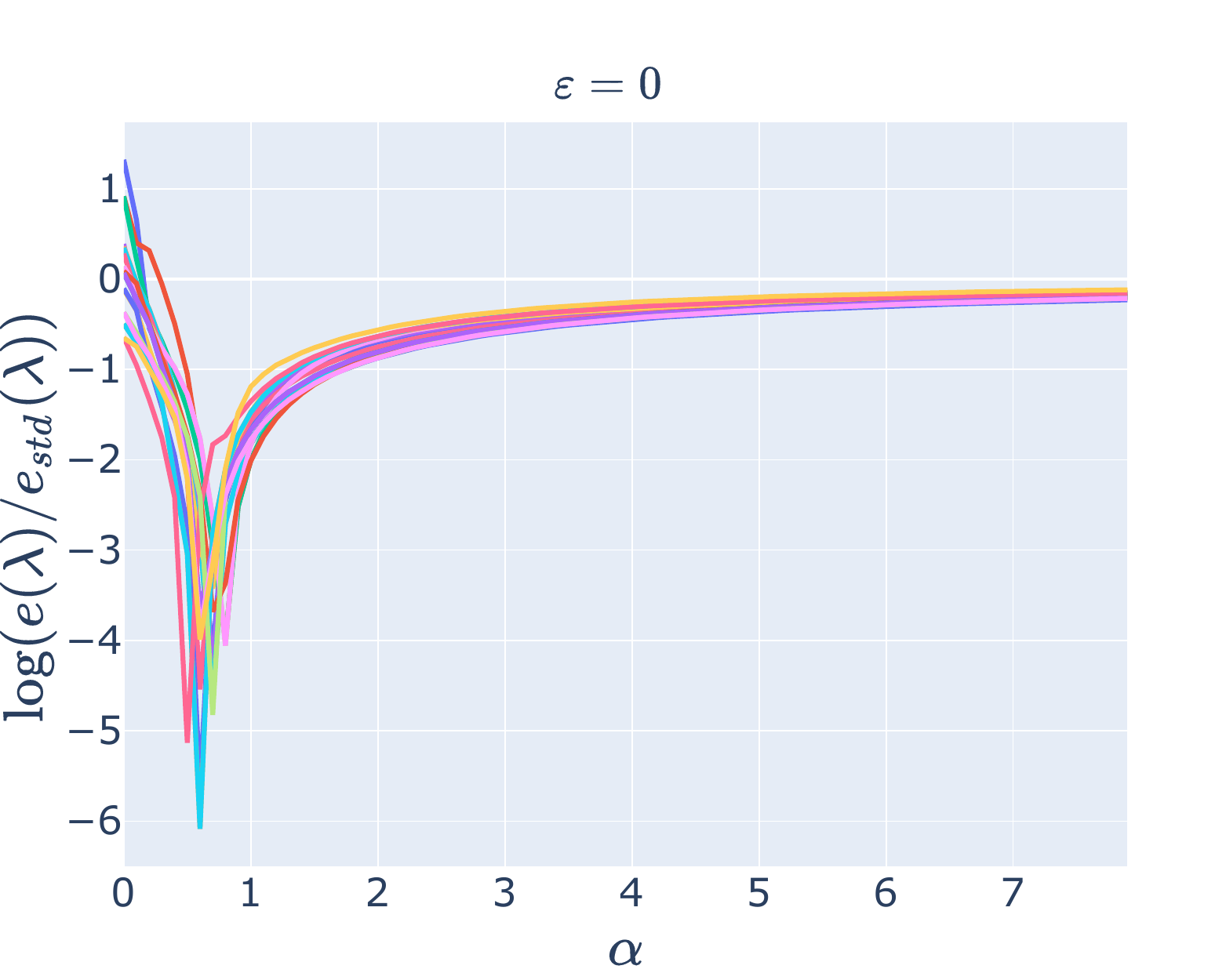}
		\end{minipage}
		\begin{minipage}{0.49\linewidth}\centering
			\includegraphics[scale=0.298,trim=0cm 0cm 1cm 1cm,clip]{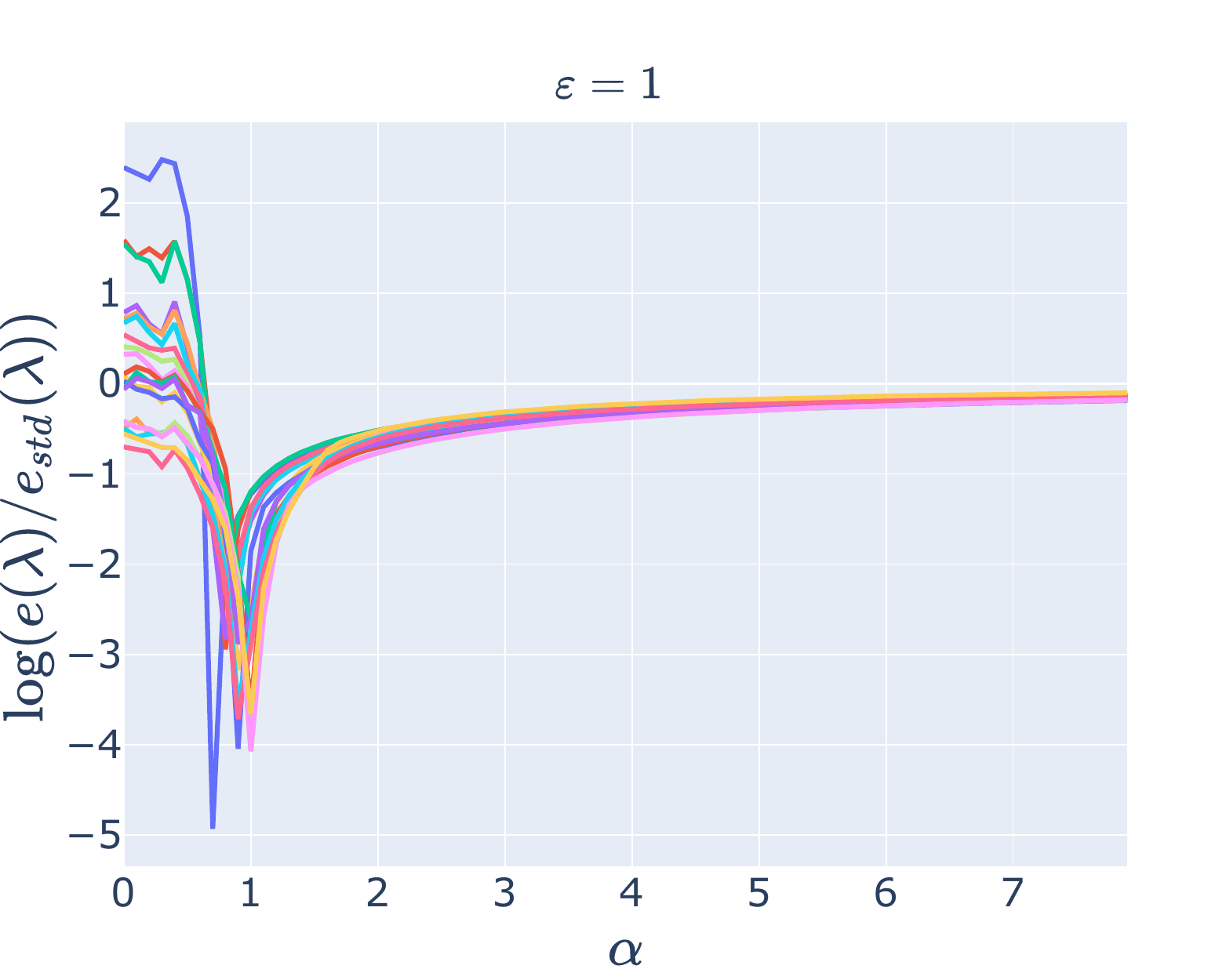}
		\end{minipage}\\
		\caption{Test \ref{subsec:cube-domain}. Dependence of the eigenvalues in the unit cube domain when using Nitsche's symmetric and incomplete variants with respect to the stabilization parameter $\alpha$, and $k=1, N=4$. Top: computation of the first 40 eigenvalues for each $\alpha$. Bottom: relative accuracy of first 20 computed eigenvalues.}
		\label{fig:alfa-dependence-cube}
	\end{figure}
	
	\begin{figure}[!hbt]\centering
		\begin{minipage}{0.32\linewidth}\centering
			{\footnotesize $\lambda_{h,1},\varepsilon=1$}\\
			\includegraphics[scale=0.12,trim=20cm 0cm 20cm 3cm,clip]{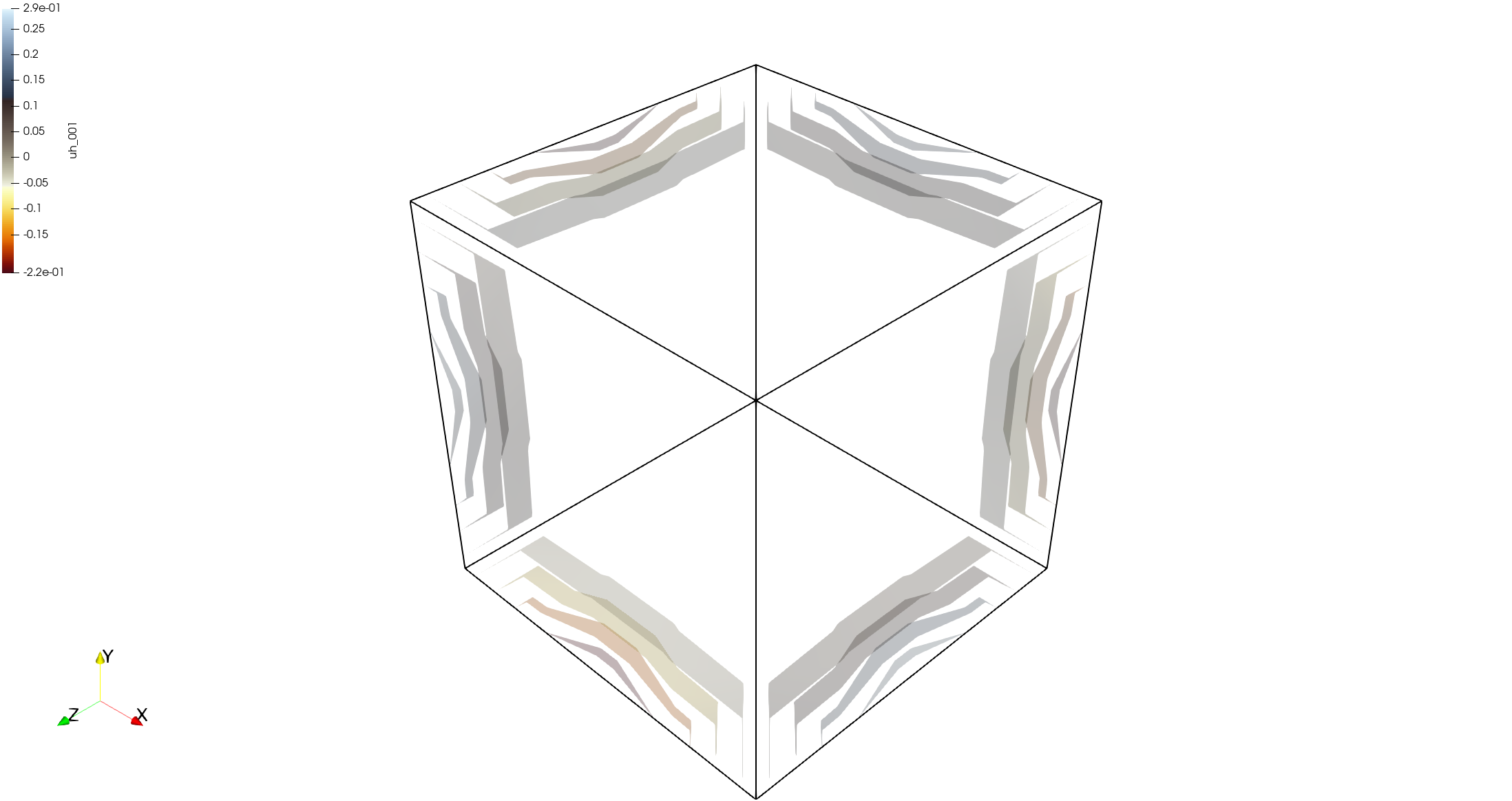}
		\end{minipage}
		\begin{minipage}{0.32\linewidth}\centering
			{\footnotesize $\lambda_{h,2},\varepsilon=1$}\\
			\includegraphics[scale=0.12,trim=20cm 0cm 20cm 3cm,clip]{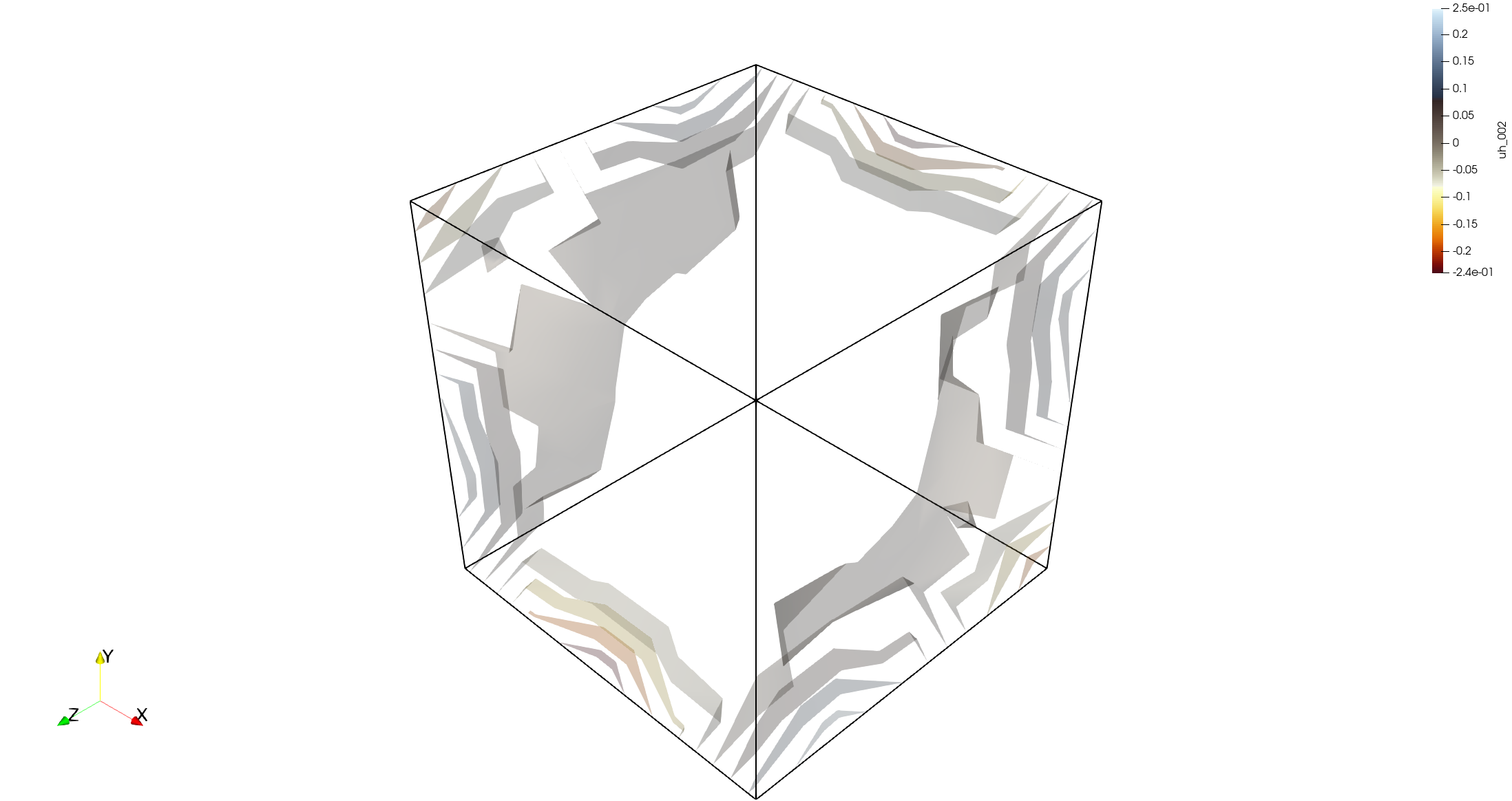}
		\end{minipage}
		\begin{minipage}{0.32\linewidth}\centering
			{\footnotesize $\lambda_{h,5},\varepsilon=1$}\\
			\includegraphics[scale=0.12,trim=20cm 0cm 20cm 3cm,clip]{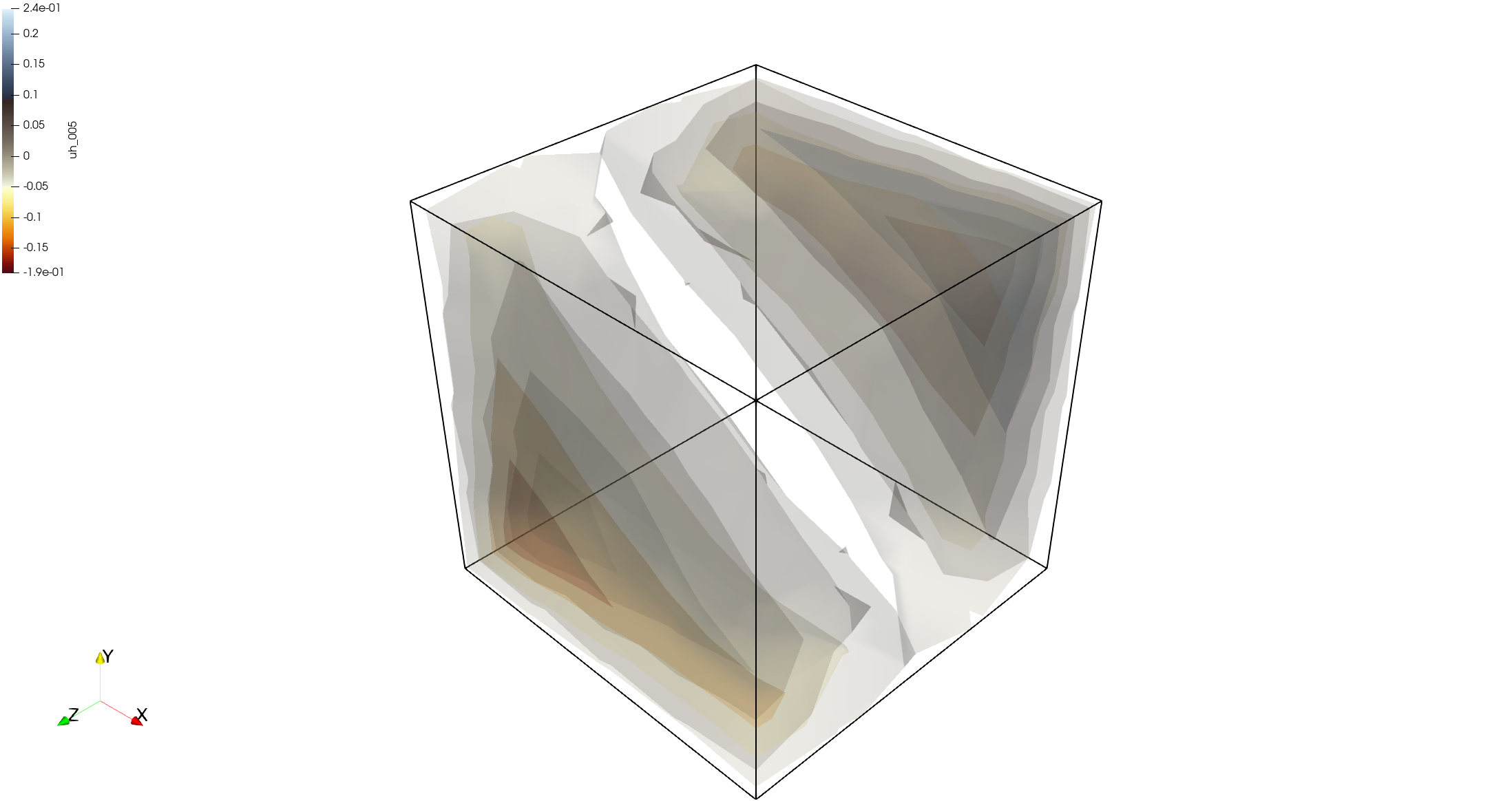}
		\end{minipage}\\
		\begin{minipage}{0.32\linewidth}\centering
			{\footnotesize $\lambda_{h,1},\varepsilon=0$}\\
			\includegraphics[scale=0.12,trim=20cm 0cm 20cm 3cm,clip]{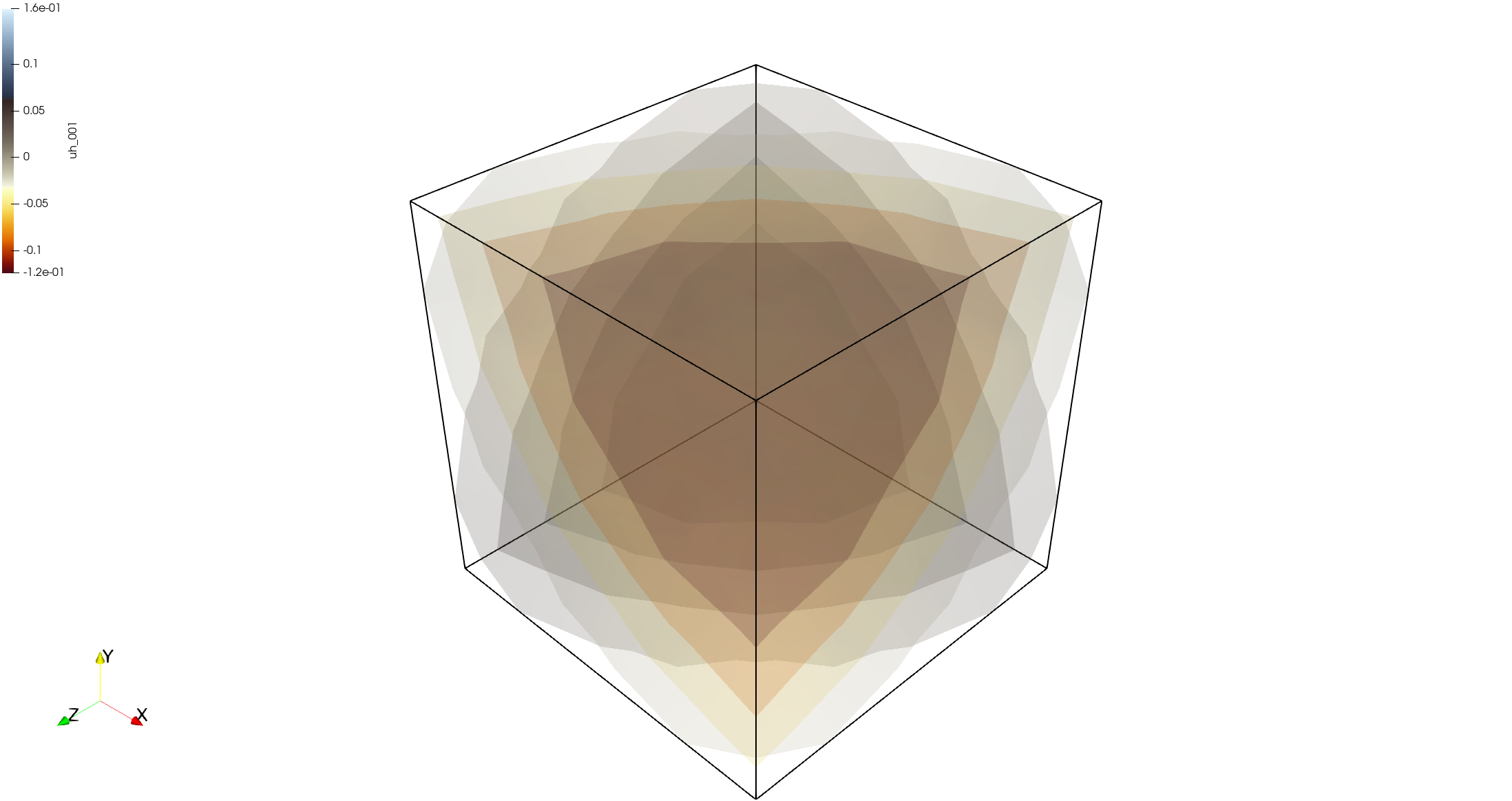}
		\end{minipage}
		\begin{minipage}{0.32\linewidth}\centering
			{\footnotesize $\lambda_{h,2},\varepsilon=0$}\\
			\includegraphics[scale=0.12,trim=20cm 0cm 20cm 3cm,clip]{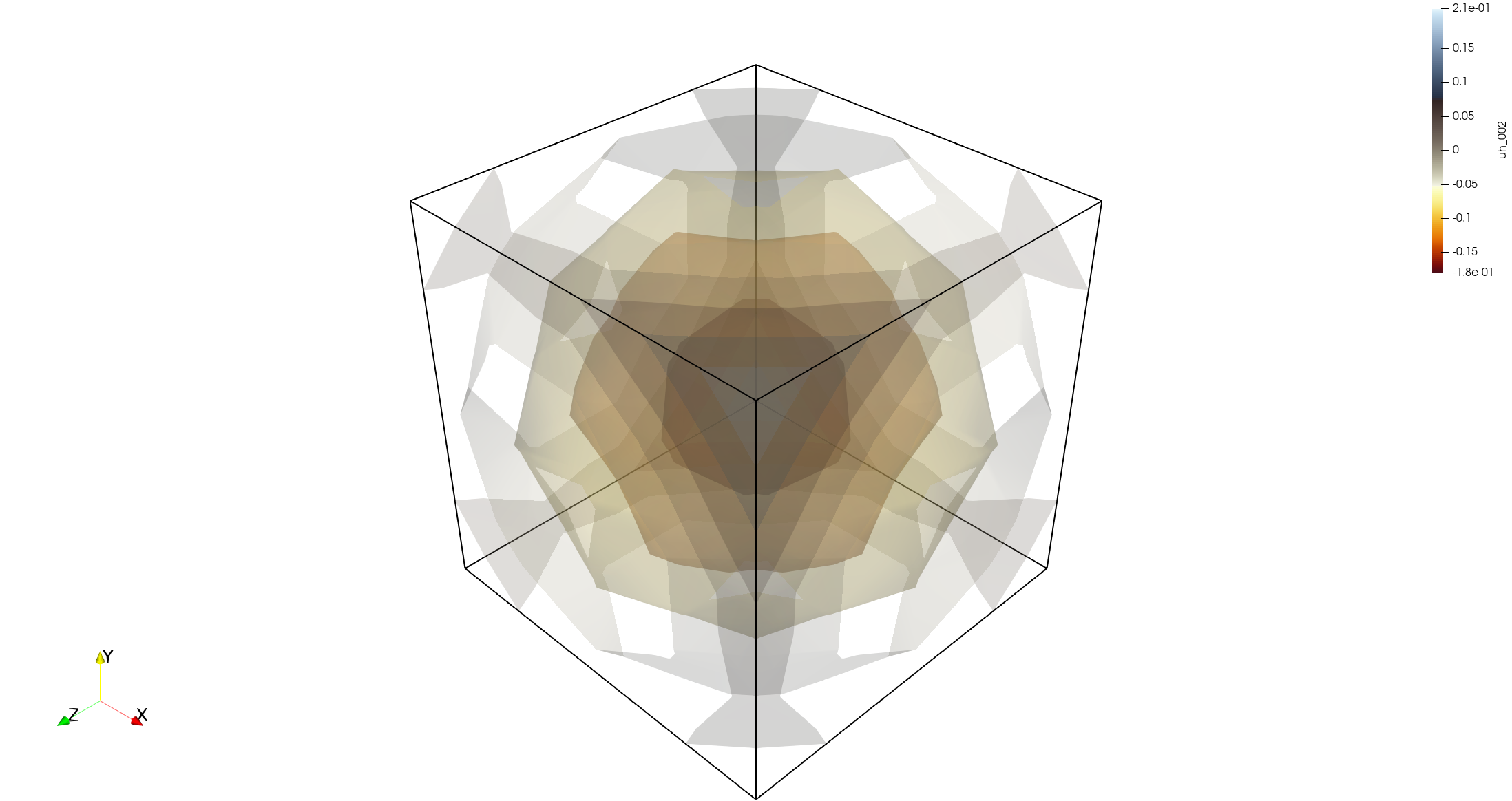}
		\end{minipage}
		\begin{minipage}{0.32\linewidth}\centering
			{\footnotesize $\lambda_{h,2},\varepsilon=0$}\\
			\includegraphics[scale=0.12,trim=20cm 0cm 20cm 3cm,clip]{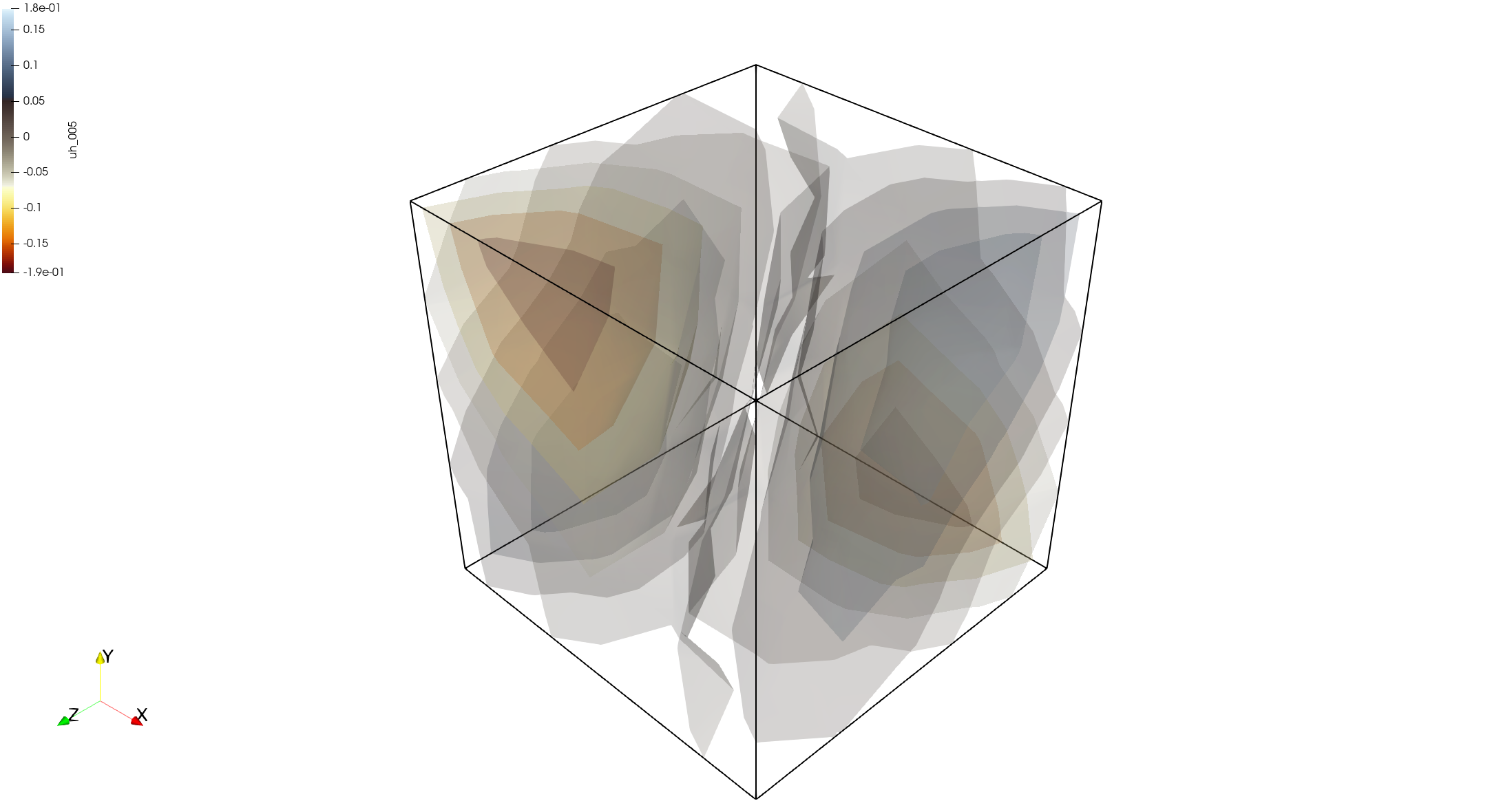}
		\end{minipage}\\
		\begin{minipage}{0.32\linewidth}\centering
			{\footnotesize $\lambda_{h,1},\varepsilon=-1$}\\
			\includegraphics[scale=0.12,trim=20cm 0cm 20cm 3cm,clip]{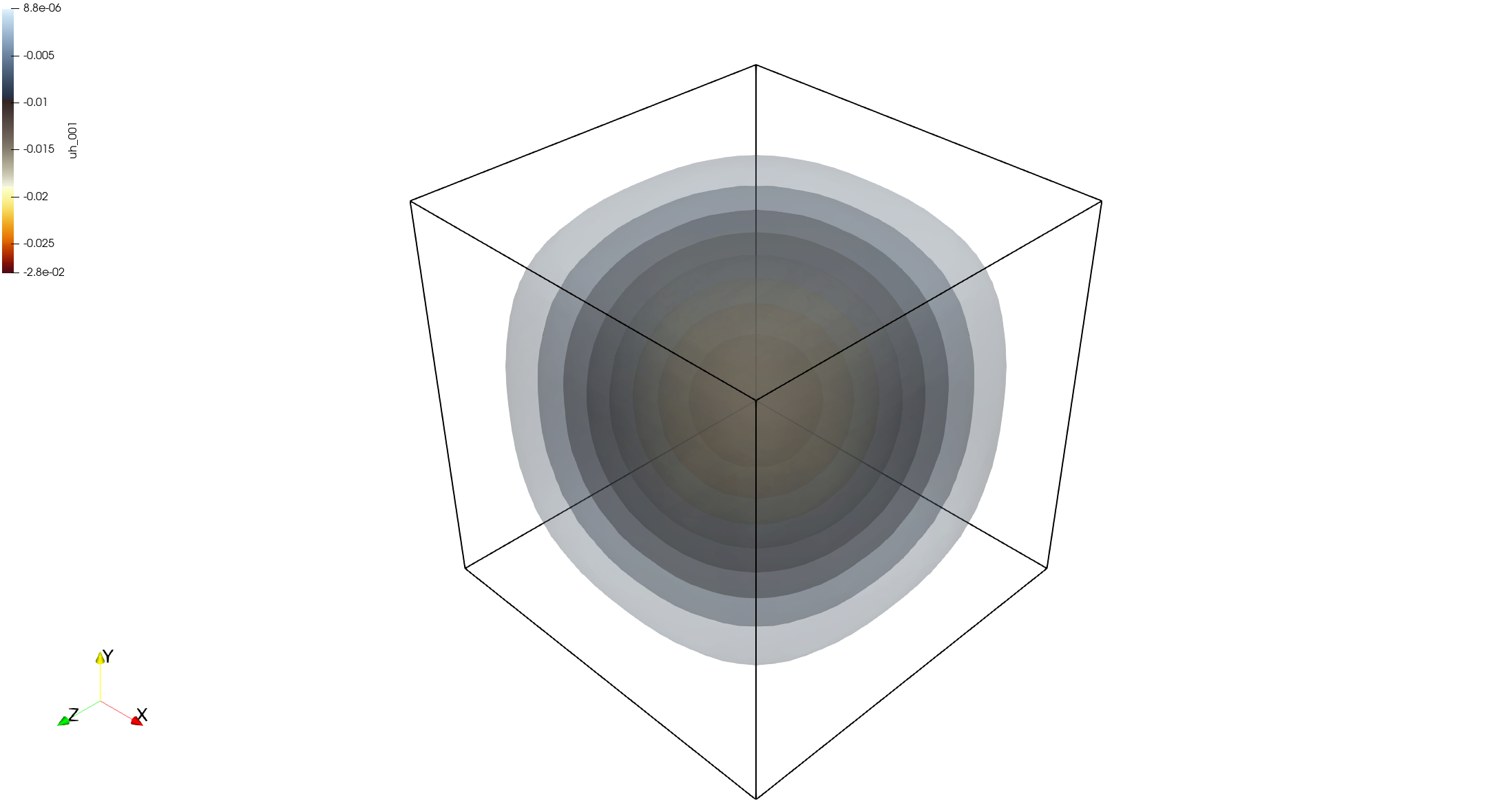}
		\end{minipage}
		\begin{minipage}{0.32\linewidth}\centering
			{\footnotesize $\lambda_{h,2},\varepsilon=-1$}\\
			\includegraphics[scale=0.12,trim=20cm 0cm 20cm 3cm,clip]{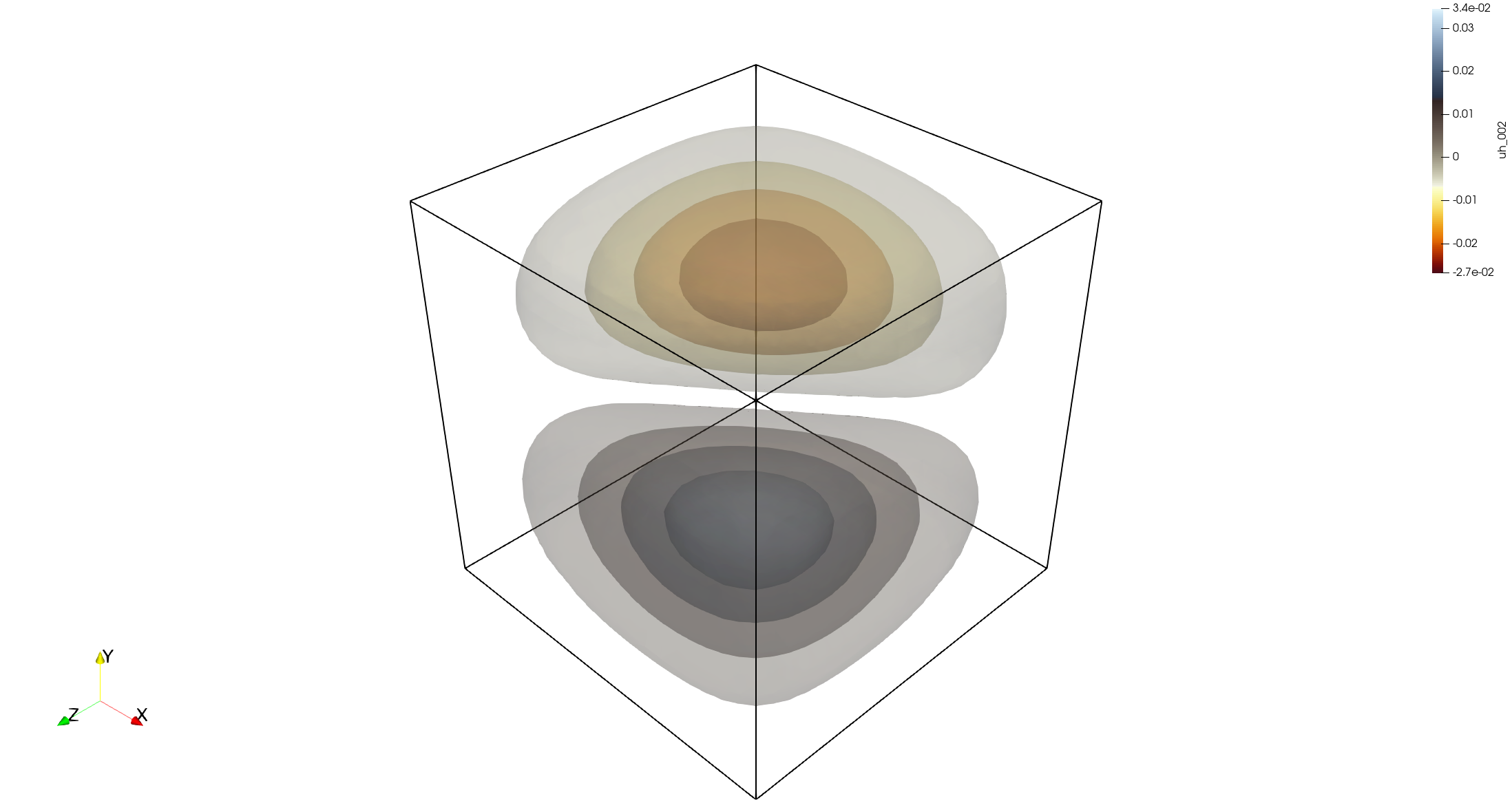}
		\end{minipage}
		\begin{minipage}{0.32\linewidth}\centering
			{\footnotesize $\lambda_{h,2},\varepsilon=-1$}\\
			\includegraphics[scale=0.12,trim=20cm 0cm 20cm 3cm,clip]{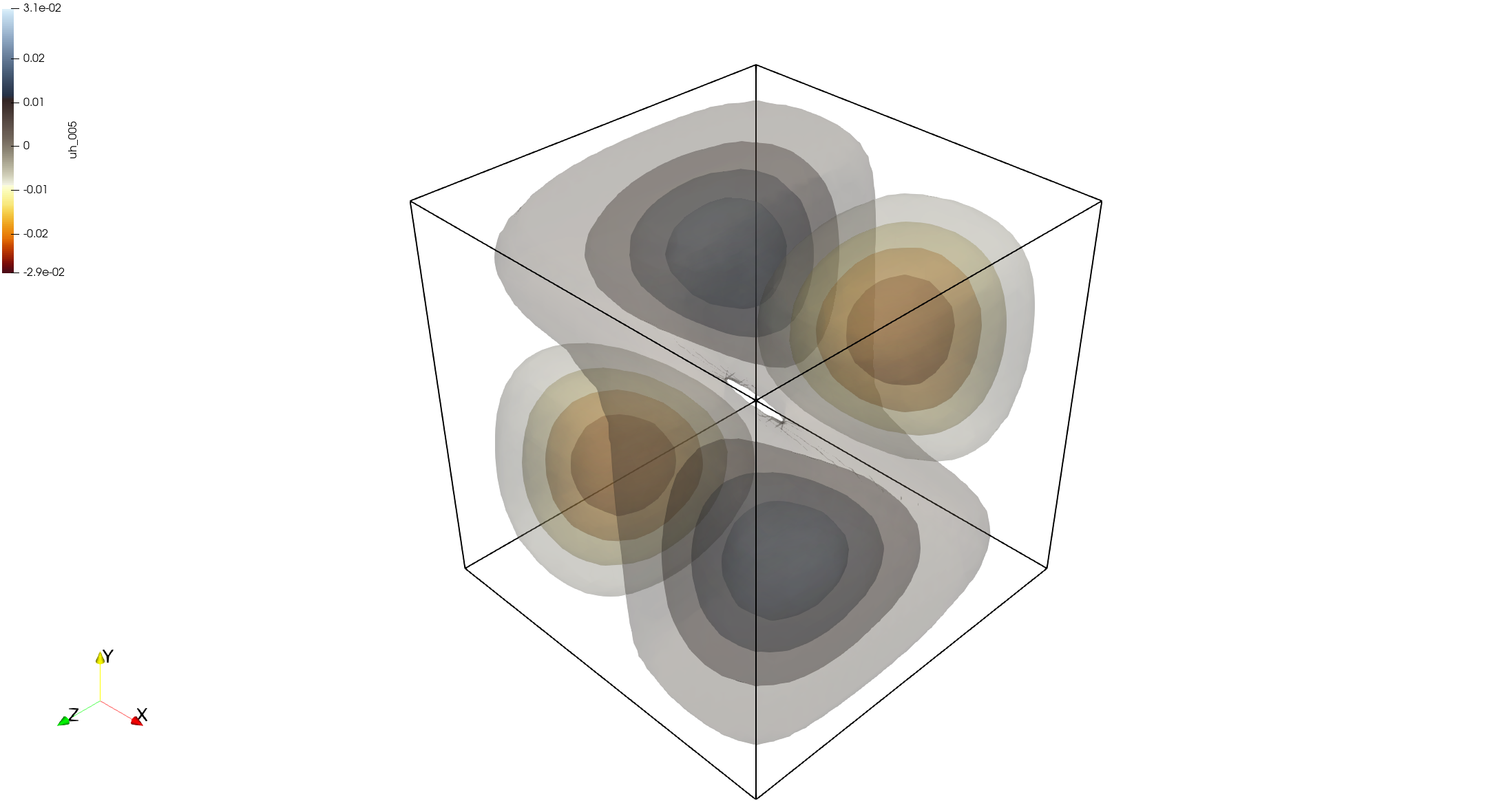}
		\end{minipage}
		\caption{Test \ref{subsec:cube-domain}. Contour plot of the first, second and fifth lowest computed eigenmodes in the cube domain when targeting the first 40 lowest eigenvalues and taking $\alpha=0.1$ in the three variants of Nitsche method for $k=1$ and $N=4$. }
		\label{fig:cube-eigenvalues-spurious}
	\end{figure}

	\subsubsection{Fichera-type three-dimensional domain}\label{subsec:fichera}

We finally consider a three-dimensional singular configuration. The computational domain is the Fichera-type corner $\Omega := (-1,1)^3 \setminus [0,1]^3,$ namely, the cube $(-1,1)^3$ with the positive octant removed. This geometry contains a re-entrant corner at the origin and several re-entrant edges, and therefore provides a
 benchmark for assessing the performance of the adaptive strategy in the
presence of reduced regularity. We consider two boundary configurations. In the first one, the whole boundary is treated as the Dirichlet boundary and the corresponding essential condition is imposed weakly by means of Nitsche's method. In the second one, we impose mixed boundary conditions by taking the Dirichlet part as
\[
\Gamma :=
\{x=0,\ 0<y<1,\ 0<z<1\}
\cup
\{y=0,\ 0<x<1,\ 0<z<1\}
\cup
\{z=0,\ 0<x<1,\ 0<y<1\},
\]
and the Neumann part as $\Sigma:=\partial\Omega\setminus\overline{\Gamma}$. This second configuration preserves the Fichera-type geometric singularity, while introducing Dirichlet--Neumann transitions along the re-entrant region.

For the fully Dirichlet configuration, the adaptive loop is driven by the estimator introduced in Section~\ref{sec:apost}. For the mixed boundary configuration, we use the natural extension of the same residual indicator: the Nitsche boundary contribution is restricted to the Dirichlet part $\Gamma$, whereas on the Neumann part $\Sigma$ we add the natural boundary residual
\[
   \sum_{F\subset \Sigma} h_F\|\partial_n u_h\|_{0,F}^2.
\]
Thus, in the mixed case, the local indicator is obtained by replacing the full-boundary contribution by
\[
   \sum_{F\subset\partial K\cap\Gamma}h_F^{-1}\|u_h\|_{0,F}^2
   +
   \sum_{F\subset\partial K\cap\Sigma}h_F\|\partial_n u_h\|_{0,F}^2.
\]
We emphasize that the a posteriori theory proved in Section~\ref{sec:apost} covers the fully Dirichlet case. The mixed Fichera experiment should therefore be understood as a numerical stress test of the estimator beyond the scope of the reliability and efficiency results established above.

Starting from an initial quasi-uniform tetrahedral mesh, the adaptive loop is driven by the corresponding residual indicator described above. Figure~\ref{fig:fichera-meshes} shows representative meshes obtained during the refinement process for both boundary configurations. In the fully Dirichlet case, the adaptive algorithm concentrates degrees
of freedom in a neighbourhood of the singular corner and along the re-entrant edges. In the mixed case, the refinement pattern is even more localized around the intersection of the re-entrant geometry and the Dirichlet--Neumann transition, which is consistent with the expected loss of regularity.

\begin{figure}[htbp]
\centering
\begin{minipage}{0.32\linewidth}\centering
{\footnotesize initial mesh}
\includegraphics[scale=0.12,trim=11cm 0cm 11cm 2cm,clip]{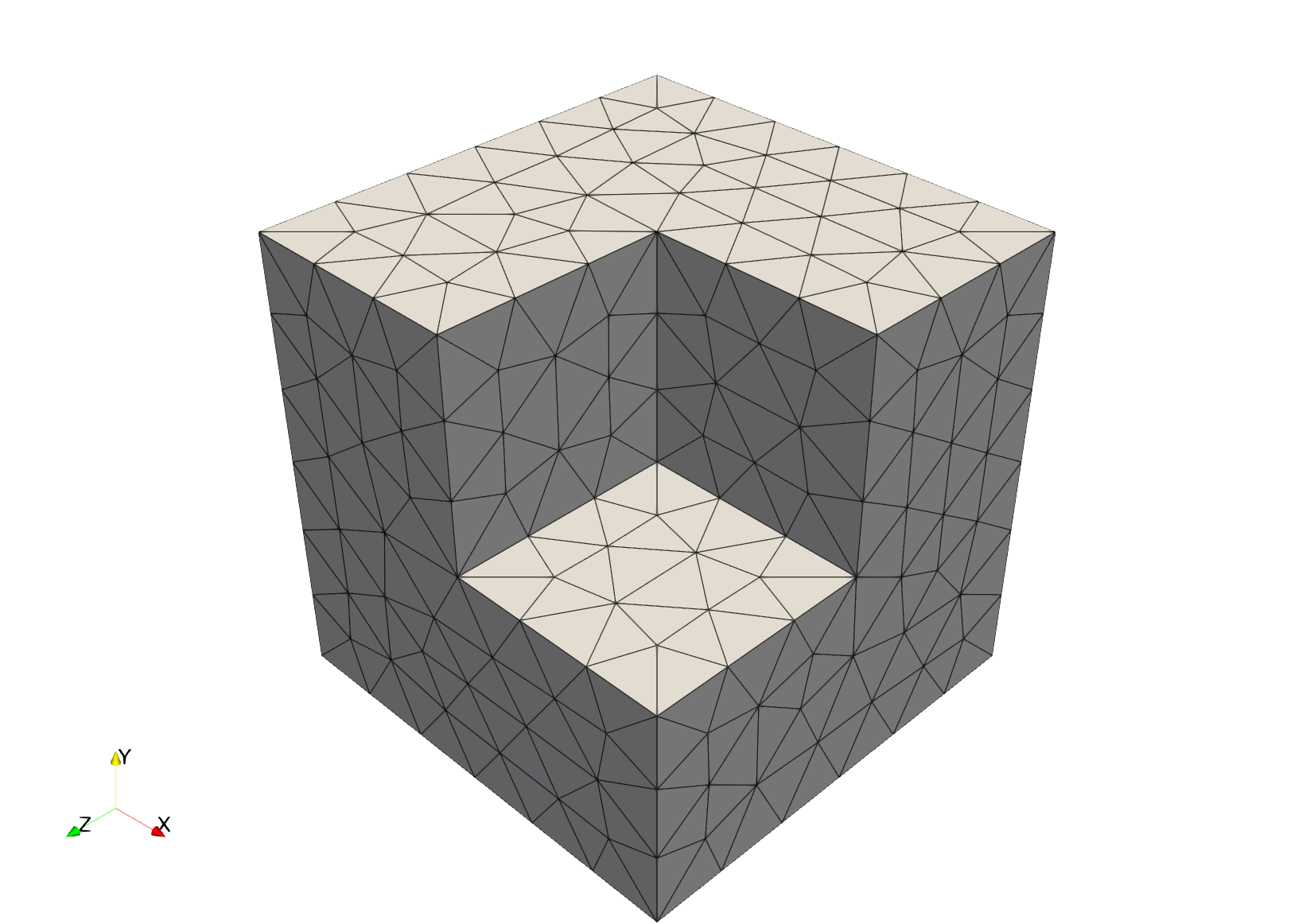}
\end{minipage}
\begin{minipage}{0.32\linewidth}\centering
{\footnotesize fully Dirichlet}
\includegraphics[scale=0.12,trim=11cm 0cm 11cm 2cm,clip]{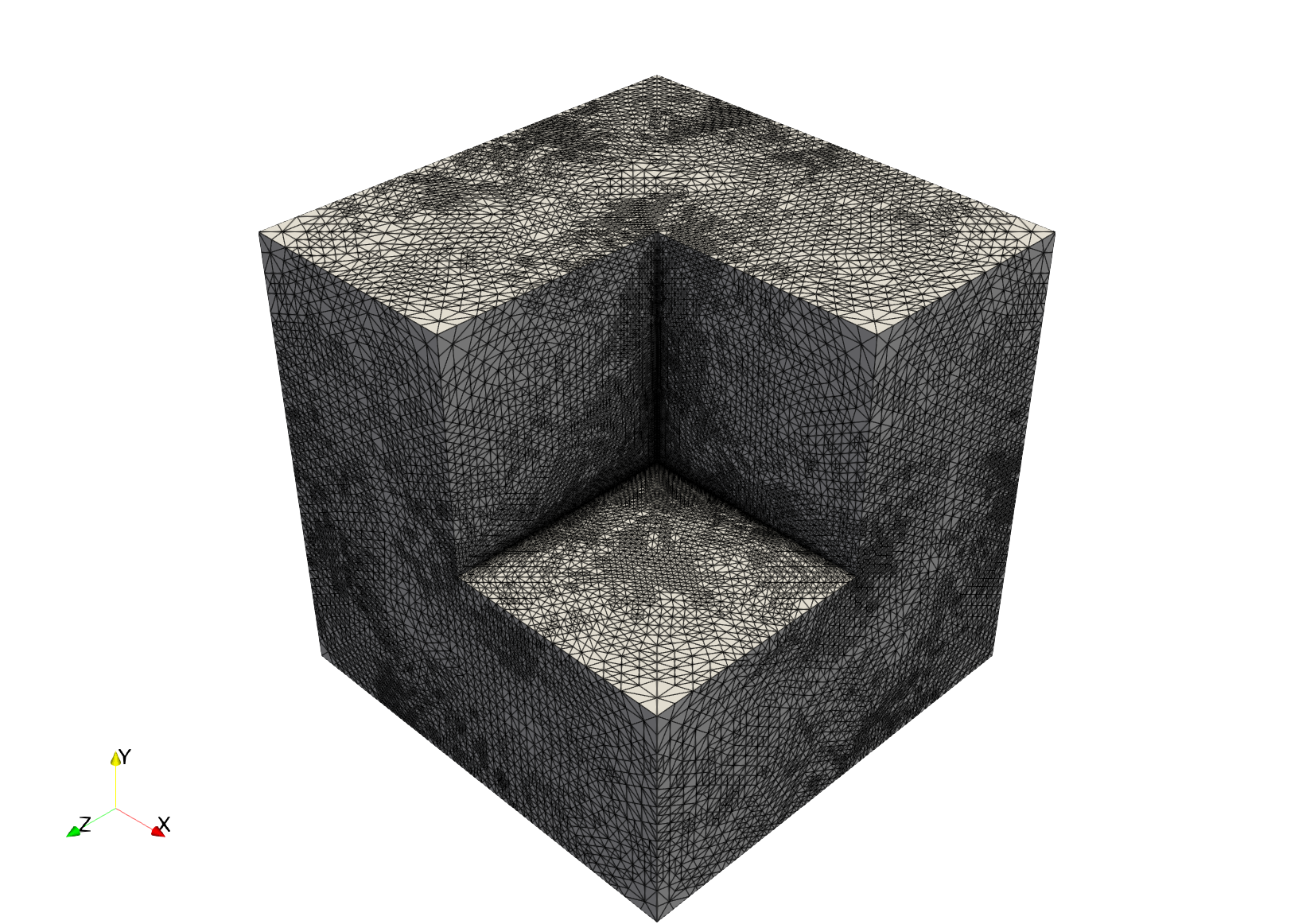}
\end{minipage}
\begin{minipage}{0.32\linewidth}\centering
{\footnotesize mixed boundary}
\includegraphics[scale=0.08,trim=22cm 0cm 22cm 3cm,clip]{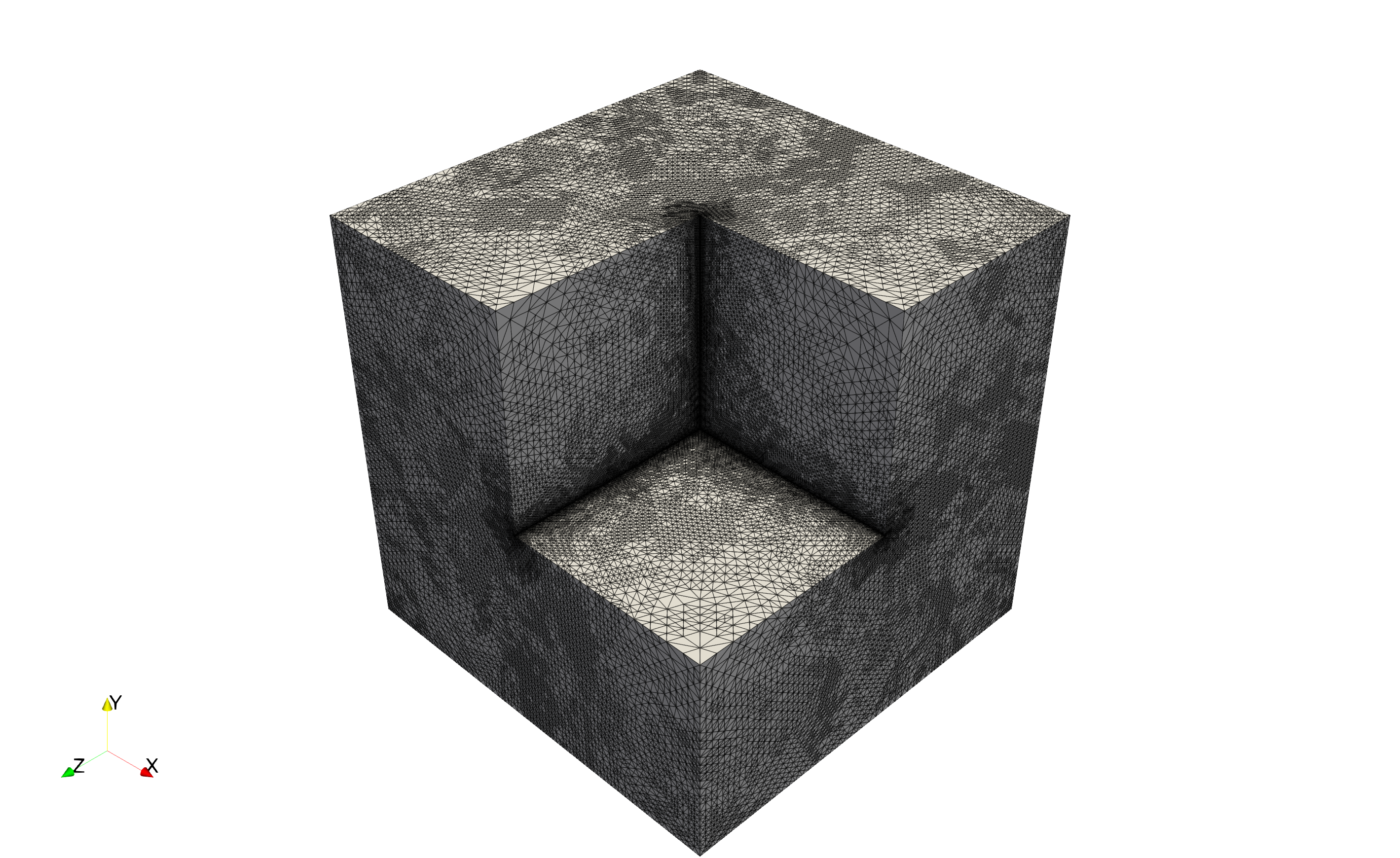}
\end{minipage}
\caption{Fichera-type domain. Initial mesh and final adaptively refined meshes for the
fully Dirichlet and mixed boundary configurations.}
\label{fig:fichera-meshes}
\end{figure}

Figure~\ref{fig:fichera-eigenfunction} displays the numerical approximation of the first
eigenfunction in both cases. The structure of the computed eigenmodes is consistent with
the singular character of the domain. In the fully Dirichlet case, the eigenfunction
exhibits the expected concentration around the re-entrant corner. In the mixed case, the
eigenfunction reflects the additional effect of the boundary transition, while preserving
a clear singular behavior in the neighbourhood of the Fichera corner.

\begin{figure}[htbp]
\centering
\begin{minipage}{0.49\linewidth}\centering
{\footnotesize fully Dirichlet}\\
\includegraphics[scale=0.08,trim=22cm 0cm 22cm 2cm,clip]{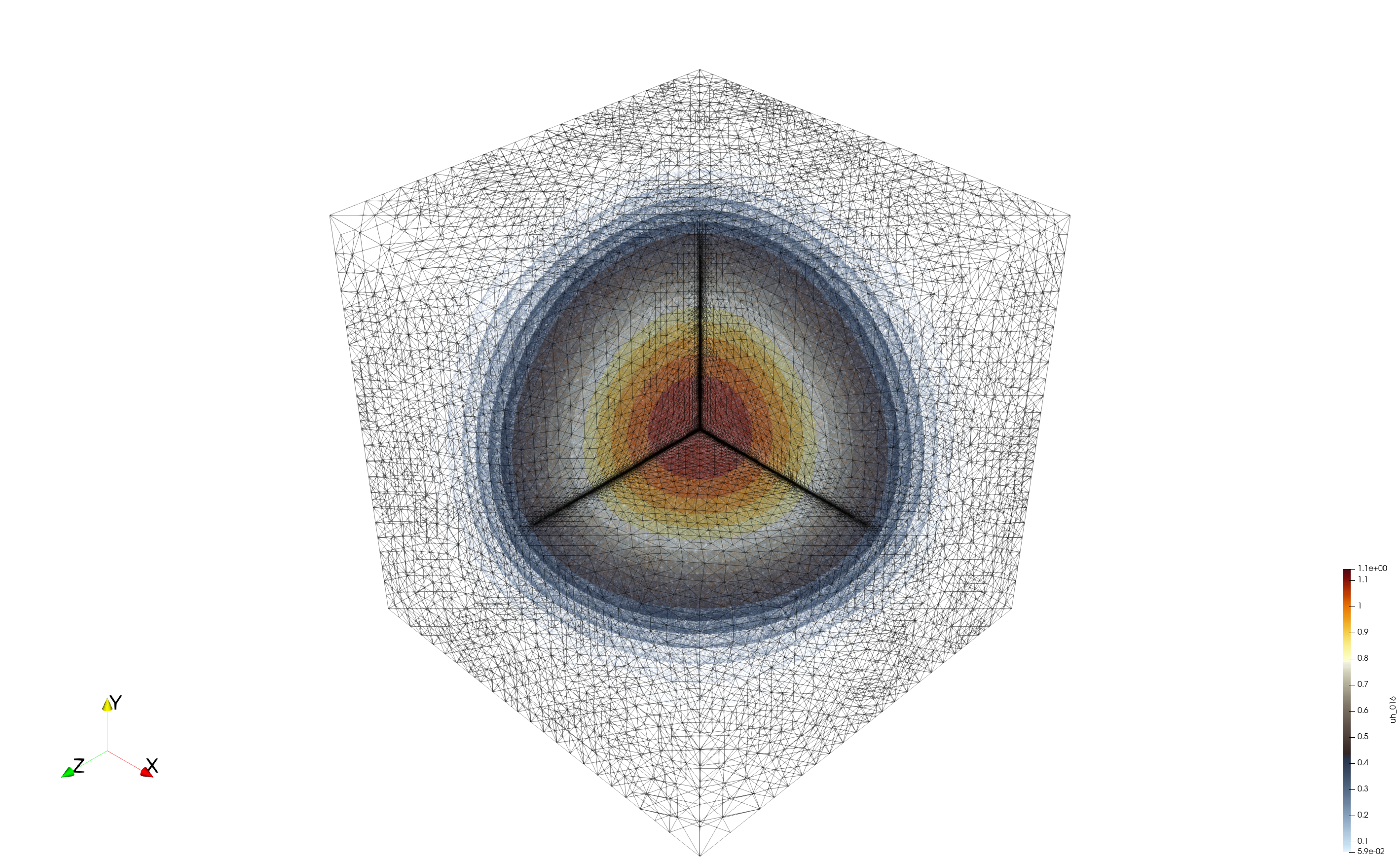}
\end{minipage}
\begin{minipage}{0.49\linewidth}\centering
{\footnotesize mixed boundary}\\
\includegraphics[scale=0.08,trim=22cm 0cm 20cm 2cm,clip]{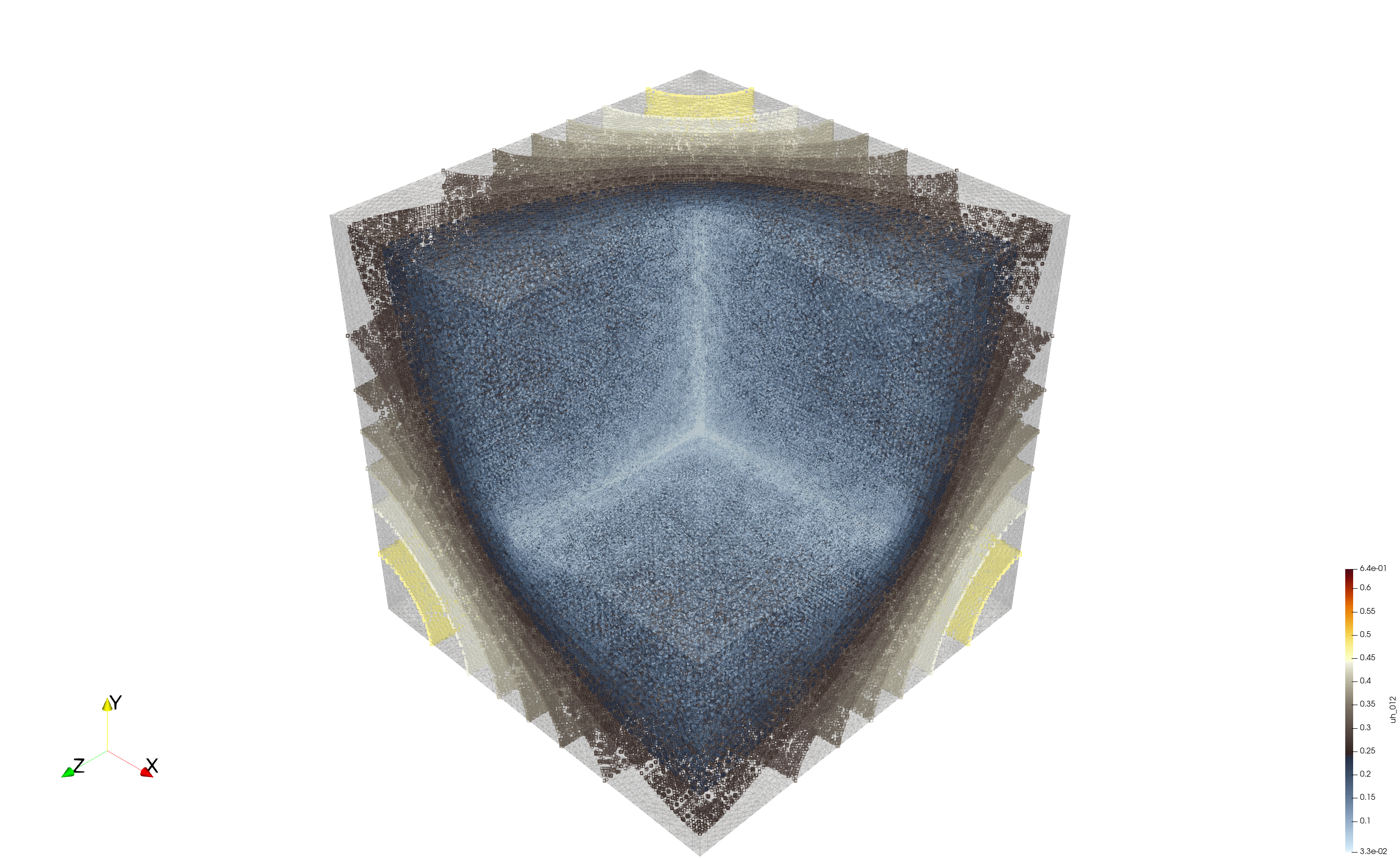}
\end{minipage}
\caption{Approximation of the first eigenfunction on the Fichera-type domain for the
fully Dirichlet and mixed boundary configurations.}
\label{fig:fichera-eigenfunction}
\end{figure}

Since the exact eigenvalues are not available, we estimate the errors by using extrapolated values. For the fully Dirichlet case we take $\lambda_{1,\mathrm{ext}}^{D} = 10.5320475849391,$ whereas for the mixed boundary case we take $\lambda_{1,\mathrm{ext}}^{M} = 0.950584841622746.$ 

Tables~\ref{tab:fichera-eig1-dirichlet} and~\ref{tab:fichera-eig1-mixed} report the number of degrees of freedom, the computed first eigenvalue, the estimated eigenvalue error, the squared residual indicator $\eta^2$, and the corresponding effectivity index. For the mixed boundary case, $\eta^2$ includes the additional Neumann contribution over $\Sigma$ described above. In both cases, the computed eigenvalues converge monotonically to the corresponding extrapolated value. Moreover, the effectivity index remains bounded and nearly constant along the adaptive sequence. In the fully Dirichlet case, the effectivity stabilizes around $3.2\times 10^{-2}$, whereas in the mixed case it stabilizes around $4.5\times 10^{-2}$. Thus, the estimator is able to capture the asymptotic behavior of the eigenvalue error with a stable proportionality factor.

\begin{table}[htbp!]
\centering
\caption{Adaptive approximation of the first eigenvalue on the Fichera-type domain with
fully Dirichlet boundary conditions.}
\label{tab:fichera-eig1-dirichlet}
\begin{tabular}{c c c c c}
\hline
$\mathrm{dof}$ & $\lambda_{1,h}$ &
$\textrm{e}(\lambda_{1})$ & $\eta^2$ &
$\mathrm{eff}(\lambda_1)$ \\
\hline
381     & 12.92404726 & 2.39199968 & 91.52429794 & 0.0261 \\
519     & 12.03318974 & 1.50114216 & 49.07205051 & 0.0306 \\
1066    & 11.39544802 & 0.86340044 & 26.19581333 & 0.0330 \\
1917    & 11.09767665 & 0.56562907 & 16.62839646 & 0.0340 \\
4488    & 10.84242253 & 0.31037495 & 9.42300529  & 0.0329 \\
10439   & 10.71653197 & 0.18448439 & 5.50997722  & 0.0335 \\
23813   & 10.63762803 & 0.10558045 & 3.20258868  & 0.0330 \\
51634   & 10.59565862 & 0.06361103 & 1.93347940  & 0.0329 \\
111416  & 10.57001928 & 0.03797170 & 1.15560480  & 0.0329 \\
281972  & 10.55258391 & 0.02053632 & 0.62884301  & 0.0327 \\
717724  & 10.54317523 & 0.01112764 & 0.33921781  & 0.0328 \\
1497991 & 10.53875508 & 0.00670749 & 0.20687976  & 0.0324 \\
\hline
\end{tabular}
\end{table}

\begin{table}[htbp!]
\centering
\caption{Adaptive approximation of the first eigenvalue on the Fichera-type domain with
mixed boundary conditions.}
\label{tab:fichera-eig1-mixed}
\begin{tabular}{c c c c c}
\hline
$\mathrm{dof}$ & $\lambda_{1,h}$ &
$\textrm{e}(\lambda_{1})$ & $\eta^2$ &
$\mathrm{eff}(\lambda_1)$ \\
\hline
381      & 1.05996361 & 0.10937876 & 1.58594289 & 0.0690 \\
557      & 1.01789834 & 0.06731349 & 0.99124691 & 0.0679 \\
998      & 0.99001544 & 0.03943060 & 0.69200999 & 0.0570 \\
2237     & 0.97439435 & 0.02380951 & 0.44588349 & 0.0534 \\
5134     & 0.96430818 & 0.01372334 & 0.27291937 & 0.0503 \\
10471    & 0.95956456 & 0.00897972 & 0.17611006 & 0.0510 \\
26354    & 0.95539531 & 0.00481047 & 0.09988994 & 0.0482 \\
69626    & 0.95319257 & 0.00260773 & 0.05322025 & 0.0490 \\
140679   & 0.95217816 & 0.00159332 & 0.03412030 & 0.0467 \\
367711   & 0.95143733 & 0.00085248 & 0.01804958 & 0.0472 \\
828157   & 0.95107447 & 0.00048963 & 0.01074049 & 0.0456 \\
1864416  & 0.95086470 & 0.00027986 & 0.00625179 & 0.0448 \\
\hline
\end{tabular}
\end{table}

The convergence histories are shown in Figure~\ref{fig:fichera-convergence}. For the fully Dirichlet case, the adaptive curve is consistently below the uniform one and follows a experimental convergence rate close to $\mathcal{O}(\mathrm{dof}^{-2/3})$, which corresponds to the expected optimal eigenvalue rate for linear elements in three dimensions. In the mixed boundary case, the adaptive strategy also outperforms uniform refinement. The uniform sequence exhibits a slower convergence, close to $\mathcal{O}(\mathrm{dof}^{-0.48})$, while the adaptive sequence recovers the optimal rate. This behavior provides numerical evidence that the residual indicator remains effective in compensating for the reduced regularity produced by the combined effect of the Fichera corner and the mixed boundary transition.
\begin{figure}[htbp]
\centering
\begin{minipage}{0.49\linewidth}\centering
{\footnotesize fully Dirichlet}
\includegraphics[scale=0.4,trim=0cm 0cm 2cm 2cm,clip]{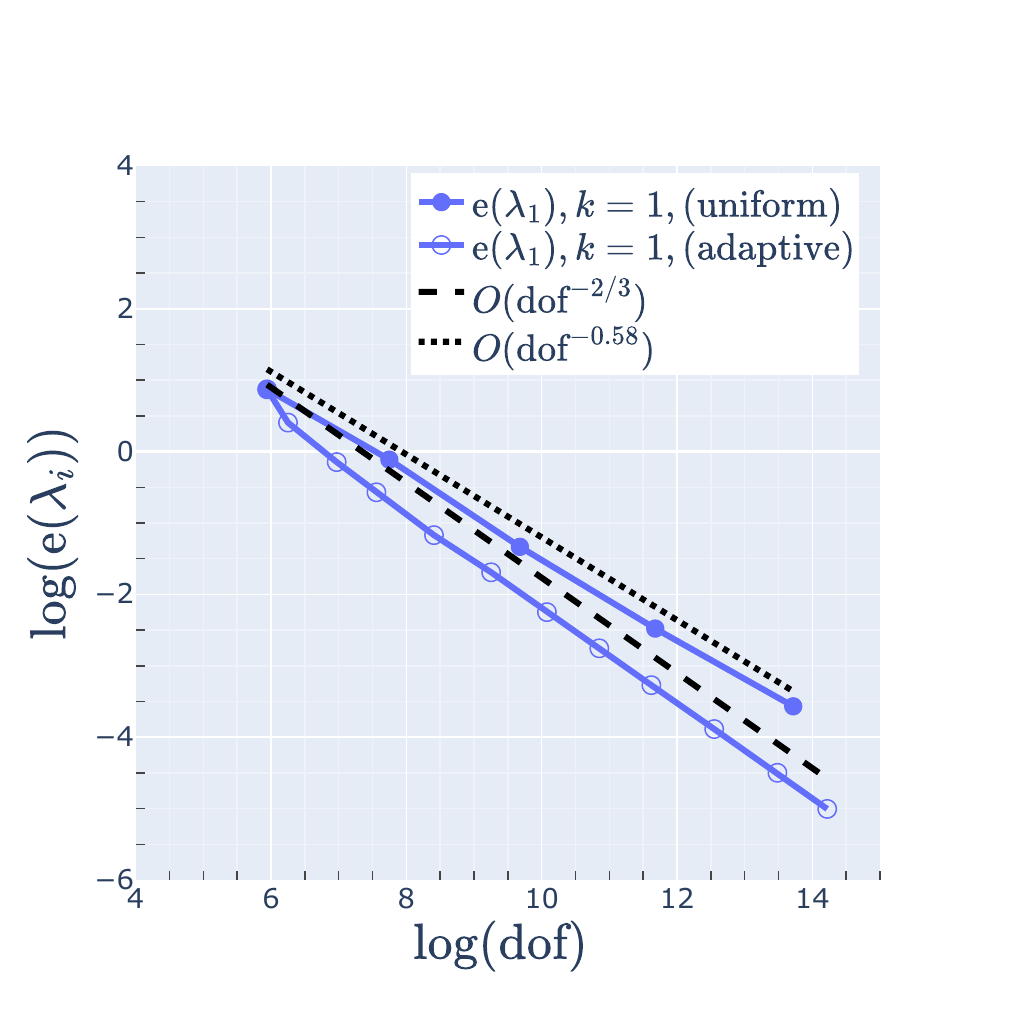}
\end{minipage}
\begin{minipage}{0.49\linewidth}\centering
{\footnotesize mixed boundary}
\includegraphics[scale=0.4,trim=0cm 0cm 2cm 2cm,clip]{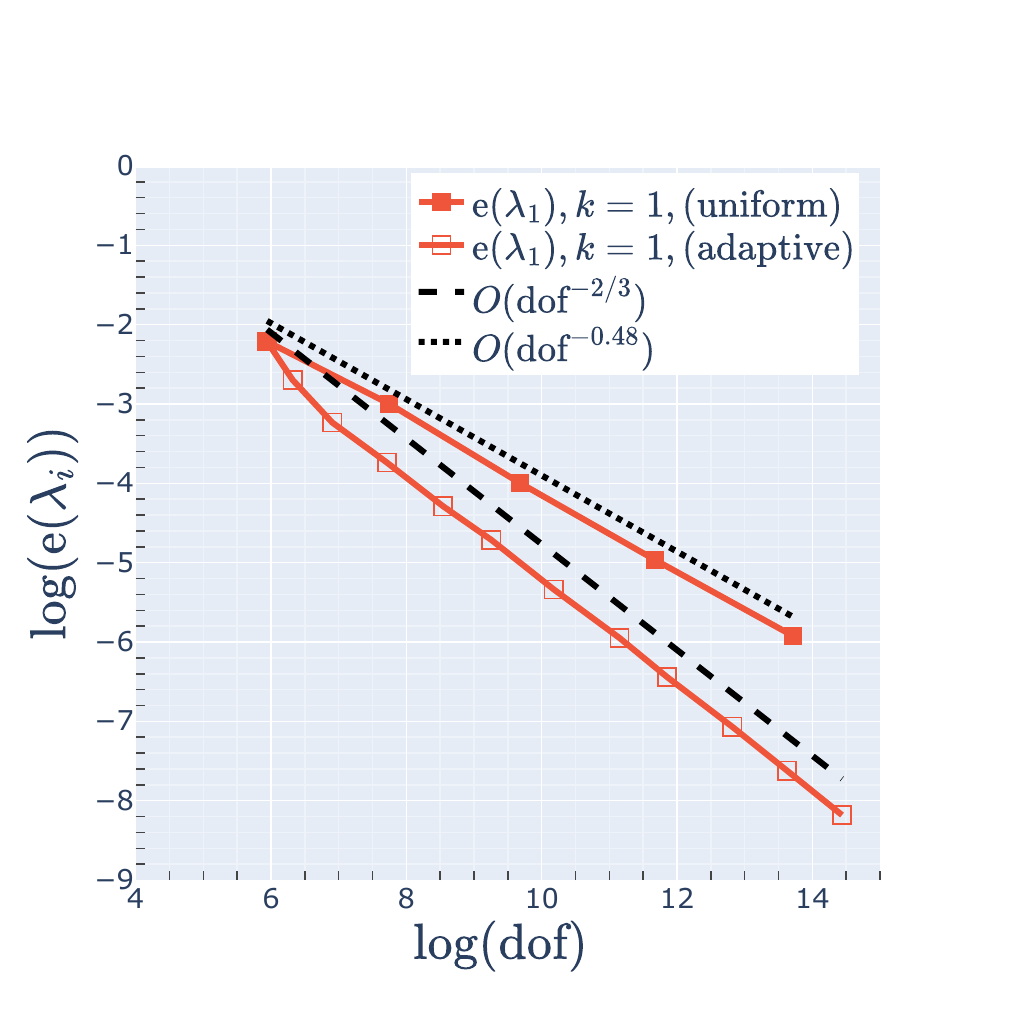}
\end{minipage}
\caption{Fichera-type domain. Convergence history for the first eigenvalue under uniform
and adaptive refinement. Left: fully Dirichlet boundary conditions. Right: mixed boundary
conditions.}
\label{fig:fichera-convergence}
\end{figure}

\begin{figure}[htbp]
\centering
\begin{minipage}{0.49\linewidth}\centering
{\footnotesize fully Dirichlet}\\
\includegraphics[scale=0.45,trim=1.1cm 0cm 2cm 2cm,clip]{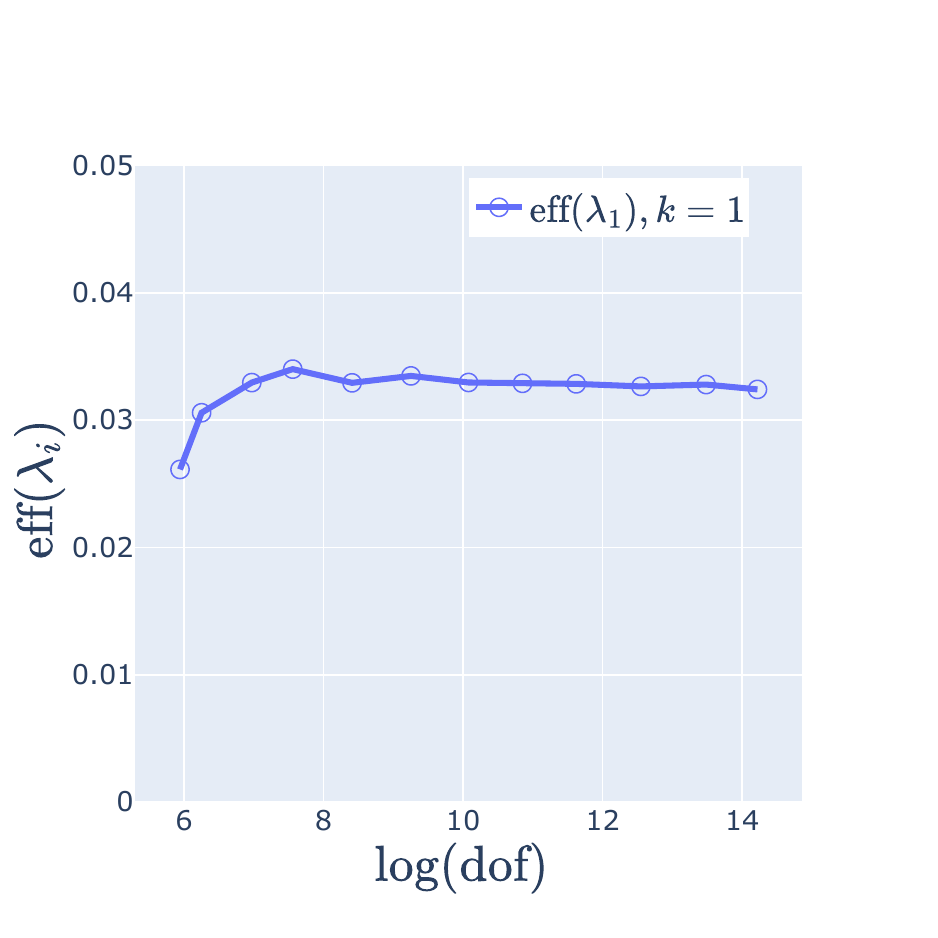}
\end{minipage}
\begin{minipage}{0.49\linewidth}\centering
{\footnotesize mixed boundary}\\
\includegraphics[scale=0.45,trim=1.1cm 0cm 2cm 2cm,clip]{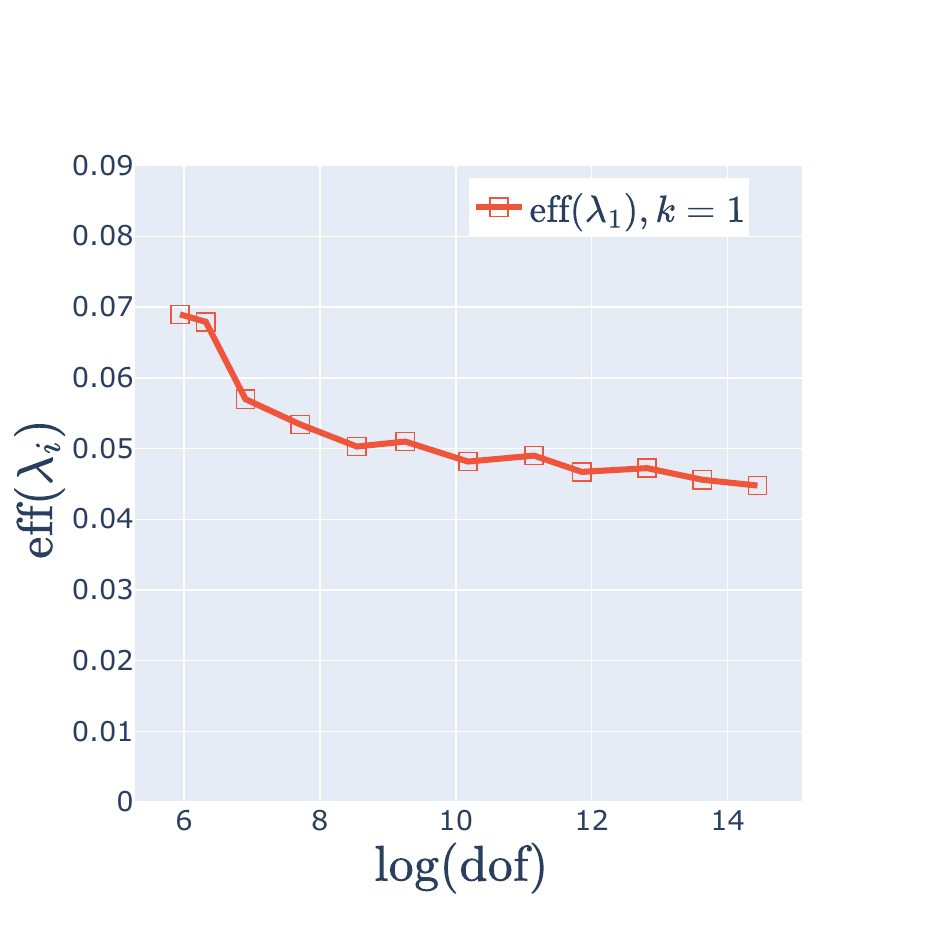}
\end{minipage}
\caption{Fichera-type domain. Effectivity index for the adaptive approximation of the
first eigenvalue. Left: fully Dirichlet boundary conditions. Right: mixed boundary
conditions.}
\label{fig:fichera-efficiency}
\end{figure}

These results show that the proposed adaptive Nitsche finite element scheme remains effective in three-dimensional non-convex domains with reduced regularity. The fully Dirichlet experiment, which is consistent with the estimator analyzed in Section~\ref{sec:apost}, shows that the adaptive procedure correctly identifies the singular region created by the Fichera corner and its re-entrant edges. The mixed boundary experiment goes beyond the theory proved in this work, since the residual indicator also includes the natural Neumann contribution on $\Sigma$. Nevertheless, the results indicate that the same adaptive strategy remains robust when the singular behavior is enhanced by the presence of Dirichlet--Neumann transitions. In both cases, the adaptive procedure yields a clear improvement over uniform refinement and produces a stable effectivity index along the refinement process.
	
	%%%%%% 
	
	\section{Conclusions}\label{sec:conclusions}

In this paper, we have presented the numerical analysis of a Laplace eigenvalue problem in which the Dirichlet boundary condition is imposed weakly by means of Nitsche's technique. We considered the symmetric, incomplete and skew-symmetric variants of the method and studied the corresponding discrete spectral problems within the framework of compact operator theory. The convergence in norm of the discrete solution operator allowed us to apply the Babu\v{s}ka--Osborn theory and to derive error estimates for eigenfunctions and eigenvalues. The predicted rates depend on the Nitsche variant. We observed that the symmetric formulation provides optimal convergence for the eigenfunctions and double order for the eigenvalues, whereas the nonsymmetric variants lead to suboptimal eigenvalue convergence.

We have also developed a residual-based a posteriori error analysis for the symmetric Nitsche formulation. For simple eigenvalues, and under a saturation assumption, we proved a reliability estimate. Local efficiency was obtained by standard bubble-function arguments. This analysis provides a theoretical basis for the adaptive refinement strategy used in the numerical experiments.

The numerical results confirm the predicted convergence rates and illustrate the robustness of the method for stabilization parameters that need not be excessively large. For small values of $\alpha$, however, the symmetric and incomplete variants may exhibit boundary layers, eigenvalue crossings or veering. Non-symmetric variants of the Nitsche method give complex eigenvalues in some cases. This behavior is consistent with the coercivity requirements for the bilinear form $A_h(\cdot,\cdot)$ and the non-self adjoint nature of the scheme. Similar to what was observed in \cite{burman2012penalty}, the skew-symmetric scheme appears to be the most robust variant in terms of spectral computation, at the price of suboptimal convergence rates. Therefore, although Nitsche's method provides a flexible mechanism for imposing boundary conditions weakly, the choice of the stabilization parameter remains essential for a correct approximation of the spectrum.

Finally, the adaptive experiments show that the proposed estimator correctly detects singular regions and improves the performance of the method in non-convex configurations. In particular, the three-dimensional Fichera-type tests indicate that the adaptive strategy remains effective beyond the regularity setting covered by the theory, including mixed boundary configurations where a complete three-dimensional regularity result is not assumed. The extension of the present analysis to more complex eigenvalue problems, such as transmission eigenvalue problems and other multiphysics spectral models, will be the subject of future work.

	\bibliographystyle{siam}
	\bibliography{references}
	
\end{document}